\documentclass[10pt]{article}

\usepackage{amssymb,amsmath,amsthm,amsfonts}
\usepackage{graphicx}
\usepackage{mathrsfs}
\usepackage{dsfont}
\usepackage{subfig}
\usepackage{comment}
\usepackage[title,titletoc]{appendix} 
\usepackage{bm}
\graphicspath{{./pics/}}
\usepackage[colorlinks,linktocpage,linkcolor=blue]{hyperref}
\usepackage{verbatim}
\usepackage{algorithm,algorithmic}
\usepackage{pdflscape}
\usepackage{rotating}
\usepackage{relsize}
\usepackage{subcaption}
\usepackage[compact]{titlesec}
\titlespacing*{\section}{0pt}{*0.8}{*0.5} 
\allowdisplaybreaks[4] 

\newtheorem{theorem}{Theorem}[section]
\newtheorem{remark}{Remark}[section]
\newtheorem{definition}{Definition}[section]
\newtheorem{lemma}{Lemma}[section]

\newtheorem{proposition}{Proposition}[section]

\title{Analysis of Robin-boundary control for the Boussinesq equations}
\author{Wei Gong$^\diamond$, Dongdong Liang$^\dag$, Xianbing Luo$^\ddag$ and Changlun Ye$^{\star}$}
 \date{\footnotetext{$^\diamond$SKLMS \& NCMIS, Institute of Computational Mathematics, Academy of Mathematics and Systems Science, Chinese Academy of Sciences, Beijing 100190, China. Email: wgong@lsec.cc.ac.cn. This author was supported by the National Key Research and Development Program of China (2022YFA1004402), the NSFC under grant no. 12494543 and 12471393, the Strategic Priority Research Program of the Chinese Academy of Sciences XDB0640000 \& XDB0640200.\\
 $^\dag$School of Mathematics, Xi'an University of Technology, Xi'an 710048, China. 
 Email: liang\_d1221@163.com.\\
 $^\ddag$School of Mathematics and Statistics, Guizhou University, Guiyang 550025, China. Email: xbluo1@gzu.edu.cn.\\
 $^\star$School of Mathematical Sciences, Guizhou Normal University, Guiyang 550025, China. Email: chlunye@gznu.edu.cn.}}

\begin{document}
\maketitle

\begin{abstract}
In this paper, we study an optimal boundary control problem for the Boussinesq equations, which couple the time-dependent Navier-Stokes system with a heat equation, where the control enters through a Robin boundary condition on temperature. We begin by establishing the well-posedness of the optimization problem via a variational framework. We then derive both first- and second-order optimality conditions, including explicit characterizations of the adjoint state and the optimal control. Next, we perform a detailed numerical analysis of a fully discrete scheme: using finite elements in space and a semi-implicit scheme in time, combined with variational discretization for the control. We present rigorous a prior error estimates for the state, adjoint state, and control variables. Numerical experiments are provided to validate our theoretical results.



\end{abstract}

\section{Introduction}\setcounter{equation}{0}
Let $\Omega\subseteq \mathbb{R}^2$ be a convex polygonal domain with boundary $\Gamma$ and $I:=(0,T)$ for some constant $T>0$. Given the regularization parameter $\alpha>0$, the target states $\mathbf y_d\in  L^2(I;\mathbb{L}^2(\Omega))$ and $\mathbf \theta_d\in L^2(0,T;L^2(\Omega))$, consider the following optimal Robin boundary control problem:
\begin{equation}
\mathrm{(P)}\quad\begin{aligned}
		\min_{\substack{\tiny u\in \mathcal U_{ad}}}\quad J(u,(\mathbf y,\theta))=\frac{1}{2}\int_{I}\int_\Omega |\mathbf{y}-\mathbf{y}_d|^2dx dt&+\frac{1}{2}\int_{I}\int_\Omega |  \theta- \theta_d|^2dx dt+\frac{\alpha}{2}\int_{I}\int_\Gamma| u|^2ds dt,
	\end{aligned}\label{Goal:Tracking}
\end{equation}
where $u$ is the control and $(\mathbf y,\theta)$ is the state, which satisfy the following Boussinesq equations
	\begin{equation}\label{intr:Rb_state}
		\begin{cases}
			\partial_t \mathbf y-\nu \Delta \mathbf y+(\mathbf y\cdot\nabla)\mathbf y+\nabla p+\beta\theta \mathbf g=\mathbf  h &\mbox{in}\  \Omega_T:=I\times\Omega, \\ 
			\mathrm{div}\, \mathbf y= 0  &\mbox{in}\ \Omega_T,\\
            \mathbf y=\mathbf 0 & \mbox{on}\  \Sigma_T:=I\times\Gamma,\\
          \partial_t  \theta-\chi\Delta\theta+(\mathbf y\cdot\nabla)\theta=f &\mbox{in}\ \Omega_T,\\
             \frac{\chi}{\eta}\frac{\partial \theta}{\partial \mathbf n}+\gamma \theta=u& \mbox{on}\  \Sigma_T
		\end{cases}
	\end{equation}
with initial values 
\begin{equation}\label{intr:Rb_state_initial}
    \mathbf y(0)=\mathbf y_0\quad   \mbox{in}\ \Omega,\quad
    \theta(0)=\theta_0\quad  \mbox{in}\ \Omega.
\end{equation}
Here  $\mathbf g\in \mathbb{R}^2$ is the acceleration of gravity, $\nu$, $\chi$, $\gamma$ and $\eta$ are some given positive constants, and $\beta$ is a given constant. The admissible set of controls $\mathcal U_{ad}$ is defined as 
	\begin{equation}
		\mathcal{U}_{ad}:=\big\{u\in L^2(I;L^2(\Gamma)):\quad u_a\le u(t,x)\le u_b,\quad  \mbox{a.e.}\  (t,x)\in \Sigma_T\big\}\nonumber
	\end{equation}
with $u_a,\,u_b\in \mathbb{R}\cup\{\infty\}$ satisfying  $u_a<u_b$.

Thermally driven flows are commonly described by the Boussinesq equations. In practical applications, thermal convection plays a crucial role in the regulation of a variety of physical processes, including crystal growth, solidification, cooling of fluids in channels surrounding nuclear reactor cores, and the design and control of energy-efficient building systems (see, e.g., \cite{Kunisch_mardue, BarwolffHinze2007, ItoRavindran, AbergelTemam}). Given its wide applicability, the control of thermal convection has been extensively studied from various perspectives over the past few decades. Firstly, for the derivation of the optimality conditions and numerical simulations related to the optimal control of stationary thermally convected fluid flows, we refer to \cite{LeeImanuvilov, Gunzburger_Hou_1993, ItoRavindran, Ravindran}. Specifically, \cite{ItoRavindran, Gunzburger_Lee_1994} derived finite element error estimates within the framework of BRR theory (see \cite{Girault_1986}), which is, however, not applicable for the control constrained case.  Secondly, although the derivation of optimality conditions and numerical simulations for the optimal control problem of unsteady thermal convection fluids can be found in the literature (see, e.g., \cite{BarwolffHinze2007, Hinze_Matthes_2007, Jork_Simon, Peralta, SongYuanYue}), the analysis of its numerical convergence remains unexplored. This is mainly due to the high nonlinearity and coupling characteristics of the problem $(\mathrm{P})$ that make the numerical analysis within the classical framework highly challenging.

 Theoretical and numerical analysis for optimal controls of fluid flows has been a hot research topic since the 1990s (cf. \cite{AbergelTemam,DeckelnickHinze2004,Gunzburger,Gunzburger_Hou_1993,Gunzburger_Hou_Svobodny,Hinze_2001}). For error estimates of control problems governed by isothermal steady fluid dynamics, we refer the reader to \cite{Casas_Mateos_Raymond_2007, ZhouGong2022, RoschVexler2006, LiuYan2006}. In particular, the development of error estimates for the unsteady case has produced a substantial body of work focusing primarily on distributed controls (see \cite{AnhNguyet, DeckelnickHinze2004, Casas_Chrysafinos_2016, Casas_Chrysafinos_2015, Casas_Chrysafinos_2012, CasasChrysafinos2017, PeraltaSimon2021, LeykekhmanVexlerWagner2025}).  The problem $(\mathrm{P})$ belongs to the category of Robin boundary control, whose numerical analysis in the literature is rather limited. \cite{ChrysafGunHou2006} establishes the existence of an optimal control and derives the corresponding optimality system for a Robin boundary control problem governed by semilinear parabolic PDEs. A semidiscrete finite element scheme is then developed, for which rigorous error estimates are proved using the BRR theory.  \cite{Chrysafinos2014} establishes fully discrete finite element error estimates for an optimal Robin boundary control problem governed by linear parabolic PDEs with low-regularity initial data, employing conforming spatial discretization and discontinuous Galerkin time-stepping under weak regularity assumptions. Here, it is important to note that the time step and the mesh size are coupled.   Additionally, temporal discretization was performed using the discontinuous Galerkin time-stepping method in the previous analysis, primarily due to its consistency in discretizing the time direction, which facilitates numerical analysis. In this context, the main novelty of this paper lies in the following several aspects:
\begin{enumerate}
    \item Although there are a lot of theoretical works for optimal control problems governed by multiphysical problems, we are not aware of any numerical analysis for such kind of problems, and this work contributes to a first rigorous error estimate without  the coupling condition between the spatial and temporal step sizes;
    
    \item  All numerical analysis of time-dependent optimal control problems governed by nonlinear state equations considers fully implicit time discretization method or implicit-explicit scheme with fixed time step size,  this work gives a first numerical analysis for an implicit-explicit scheme with variable time step size  under reasonable regularity assumption. 
\end{enumerate}

Specifically, for the problem $(\mathrm{P})$ addressed in this paper, the second-order optimality conditions have not been established in the literature, but are indispensable for numerical analysis (cf. \cite{Casas_Troltzsch_2012}). Under low-regularity assumptions on the data, we derive both necessary and sufficient second-order optimality conditions for the problem. Furthermore, assuming higher data regularity, we employ a bootstrapping argument to obtain higher-order regularity estimates for the optimal control, as well as the corresponding state and adjoint variables. The another main contributions of this work consist of deriving optimal-order space-time error estimates for a fully discrete scheme under suitable regularity assumption, and presenting a detailed error analysis for the uncontrolled state and adjoint equations associated with the proposed scheme. To decouple the dependence between the time step size and the spatial mesh size in the error analysis, we first introduce a spatially semi-discrete formulation of the state and adjoint equations, providing stability analysis and spatial error estimates. On this basis, we discretize the time direction using a variable-step implicit-explicit (IMEX) scheme, leading to a fully discrete state equation for which corresponding stability and error estimates are established. Subsequently, we discretize the optimal control problem $(\mathrm{P})$ via a variational discretization approach. This yields a discrete first-order optimality system, and we provide stability and error analysis for the discrete adjoint equation. In particular, our analysis reveals that a relatively mild condition on the time step is required to obtain optimal-order error estimates for  the adjoint variable in the $L^2(I;L^2(\Omega))$-norm—a requirement stemming from the shifted temporal grid structure in the discretized adjoint equation. To relate the error estimates for the optimal control problem to those for the uncontrolled state and adjoint equations, we adapt the general analytical framework of \cite{Casas_Chrysafinos_2012} that concerns particularly a distributed optimal control problem. With careful modifications tailored to the nonlinear unsteady boundary control problem studied here, we ultimately derive optimal error estimates.

The remainder of this paper is organized as follows.  Section 2 introduces the necessary preliminaries. In Section 3, we derive the first- and second- order optimality conditions and investigate the regularity of the optimal solution pair for the Robin boundary control problem. Our main theoretical results are presented in Section 4, which includes the analysis of both semi- and fully-discrete schemes for the state and adjoint equations, the proof of convergence for the discrete control problem, and the derivation of optimal space-time error estimates.  Numerical experiments that validate the theoretical convergence rates are provided in Section 5. 

\section{Preliminaries}\setcounter{equation}{0}
We follow the standard notation for function spaces, norms and differential
operators that can be found for example in \cite{Adam_2005,SB2008}. The inner products and norms for $L^2(\Omega)$ and $L^2(\Gamma)$ are denoted by $(\cdot,\cdot)$, $\|\cdot\|$ and $(\cdot,\cdot)_\Gamma$, $\|\cdot\|_\Gamma$, respectively.
Let $H^{-s}(\Gamma):=H^s(\Gamma)^*$ and \({ H}^{-1}(\Omega) := ({H}^1_0(\Omega))^* \). The notation $\langle\cdot,\cdot\rangle_{H^s(\Gamma),H^{-s}(\Gamma)}$ (abr. $\langle\cdot,\cdot\rangle_{\Gamma}$ for $s=\frac{1}{2}$) denotes the duality pairing between $H^s(\Gamma)$ and $H^{-s}(\Gamma)$, and the analogue for the notation \( \langle \cdot, \cdot \rangle_{H^1(\Omega), H^{-1}(\Omega)}\). 
Define
$$L_0^2(\Omega) := \{ v \in L^2(\Omega) : (v,1) = 0 \}.$$
Furthermore, we define the Cartesian spaces \( \mathbb{L}^p(\Omega) := L^p(\Omega)^2 \), \( \mathbb{H}^1(\Omega) := H^1(\Omega)^2 \), \( \mathbb{H}_0^1(\Omega) := H_0^1(\Omega)^2 \), \( \mathbb{W}^{s,p}(\Omega) := W^{s,p}(\Omega)^2 \) and \(\mathbb{ H}^{-1}(\Omega) := (\mathbb{H}_0^1(\Omega))^* \). Define the divergence-free vector fields as follows:
\begin{align*}
\mathbb{X}: = \left\{ \mathbf y \in \mathbb{H}_0^1(\Omega) : \text{div} \, \mathbf y = 0\ \, \text{in}\ \, \Omega \right\},\quad
\mathbb{H}: = \left\{ \mathbf y \in \mathbb L^2(\Omega) : \text{div} \,\mathbf y = 0\ \, \text{in}\ \, \Omega,\ \,\mathbf y \cdot \mathbf n = 0 \, \text{ on } \, \Gamma \right\}.
\end{align*}
The dual space of $\mathbb{X}$ is denoted as $\mathbb{X}^*$, with the duality pairing \( \langle \cdot, \cdot \rangle_{\mathbb X, \mathbb X^*} \).   

For given Banach spaces $X$, we also use the standard notation $L^p(I; X)$ and $W^{s,p}(I; X)$ $(p\ge 1, s\geq 0)$ for Bohner spaces with associated norms (cf. \cite{Lions_1972}). 
The dual space of $L^p(I; X)(1<p<\infty) $ is \( L^{p'}(I; X^*)\) with \( \frac{1}{p} + \frac{1}{p'} = 1 \).
For simplicity, we abbreviate  $H^s(I; X):=W^{s,2}(I; X)$ for $p=2$.
The duality relationship between \( L^p(I; X) \) and \(L^{p'}(I; X^*) \) is defined as
\[
\langle u, v \rangle_{L^p(I;X), L^{p'}(I;X^*)} = \int_I\, \langle u(t), v(t) \rangle_{X, X^*} \, dt\qquad\forall (u, v)\in L^p(I; X)\times L^{p'}(I; X^*).
\]
Let $H^s(I; X) := [H^k(I; X), L^2(I; X)]_\theta$ with $(1-\theta)k = s$ (cf. \cite[p. 47]{Lions_1972}). 

Define the following spaces:
\begin{align*}
&\mathbb{W}(I) := \left\{ \mathbf{y} \in L^2(I; \mathbb{X}) : \partial_t \mathbf{y} \in L^2(I; \mathbb{X}^*) \right\}, \quad W(I) := \left\{ \theta \in L^2(I; H^1(\Omega)) : \partial_t \theta \in L^2(I; H^{1}(\Omega)^*) \right\},\\
&\mathbb{V}(I) := \left\{ \mathbf{y} \in L^2(I; \mathbb{H}^2(\Omega) \cap \mathbb{X}) : \partial_t \mathbf{y} \in L^2(I; \mathbb{H}) \right\}, \quad V(I) := \{ \theta \in L^2(I; H^2(\Omega)) : \partial_t \theta \in L^2(I; L^2(\Omega))\}
\end{align*}
endowed with the following norms for any $y\in \mathbb{W}(I)$, $\theta\in W(I)$, $\hat y\in \mathbb{V}(I)$ and $\hat\theta\in V(I)$:
\begin{align*}
&\|\mathbf{y}\|_{\mathbb{W}(I)} = \left( \|\mathbf{y}\|^2_{L^2(I; \mathbb{X})} + \|\partial_t \mathbf{y}\|^2_{L^2(I; \mathbb{X}^*)} \right)^{\frac{1}{2}}, \quad \|\theta\|_{W(I)} = \left( \|\theta\|^2_{L^2(I; H^1(\Omega))} + \|\partial_t \theta\|^2_{L^2(I; H^{1}(\Omega)^*)} \right)^{\frac{1}{2}},\\
&\|\mathbf{\hat y}\|_{\mathbb{V}(I)} = \left( \|\mathbf{\hat y}\|^2_{L^2(I; \mathbb{H}^2(\Omega) \cap \mathbb{X})} + \|\partial_t \mathbf{\hat y}\|^2_{L^2(I; \mathbb{H})} \right)^{\frac{1}{2}}, \quad \|\hat\theta\|_{V(I)} = \left( \|\hat\theta\|^2_{L^2(I; H^2(\Omega))} + \|\partial_t \hat\theta\|^2_{L^2(I; L^2(\Omega))} \right)^{\frac{1}{2}}.
\end{align*} 
Let $\mathbb W:=\mathbb{W}(I)\times {W}(I)$ and $\mathbb V:=\mathbb V(I)\times V(I)$, then we have \(\mathbb{W} \hookrightarrow C(\bar{I}; \mathbb{H}) \times C(\bar{I};L^2(\Omega))\) and \(\mathbb V\hookrightarrow C(\bar{I}; \mathbb{X})\times C(\bar{I}; H^1(\Omega)) \) (cf. \cite{Lions_1972}).

Define the bilinear forms \( \mathbf a: \mathbb{H}^1(\Omega) \times \mathbb{H}^1(\Omega) \rightarrow \mathbb{R} \) and \( a: {H}^1(\Omega) \times {H}^1(\Omega) \rightarrow \mathbb{R}\) such that
\begin{align*}
&\mathbf  a(\mathbf u,\mathbf v) = \int_{\Omega} \nabla\mathbf  u:\nabla \mathbf vdx,\quad a(u,v)=\int_{\Omega} \nabla u\cdot\nabla v dx,
\end{align*}
and the trilinear forms \(\mathbf b: \mathbb{L}^4(\Omega) \times \mathbb{H}^1(\Omega) \times \mathbb{H}^1(\Omega) \rightarrow \mathbb{R}\), \( b: \mathbb{L}^4(\Omega) \times {H}^1(\Omega) \times {H}^1(\Omega) \rightarrow \mathbb{R}\), \(\mathbf c: \mathbb{L}^4(\Omega) \times \mathbb{H}^1(\Omega) \times \mathbb{H}^1(\Omega) \rightarrow \mathbb{R}\) and \( c: \mathbb{L}^4(\Omega) \times {H}^1(\Omega) \times {H}^1(\Omega) \rightarrow \mathbb{R}\) such that
\begin{align*}
&\mathbf b(\mathbf u,\mathbf v,\mathbf w) := \frac{1}{2} \left[ \mathbf c(\mathbf u,\mathbf v,\mathbf w) - \mathbf c(\mathbf u,\mathbf w,\mathbf v)\right], \quad b(\mathbf u,v,w):=\frac{1}{2} \left[ c(\mathbf u, v, w) - c(\mathbf u, w, v) \right],\\
&\mathbf c(\mathbf u,\mathbf v,\mathbf w) :=  \int_{\Omega}(\mathbf u\cdot\nabla)\mathbf v\cdot \mathbf w dx,\quad c(\mathbf u, v, w):=\int_{\Omega}(\mathbf u\cdot\nabla) v\cdot w dx.
\end{align*}
 We refer to Appendix \ref{Appendix:A} and  Lemma \ref{lem:b:xingzhi} for the properties of the trilinear form $\mathbf b$. Moreover, there hold
\begin{eqnarray}
|\mathbf b(\mathbf u, \mathbf v, \mathbf w)| &\leq& \|\mathbf u\|_{\mathbb L^4(\Omega)} \|\nabla \mathbf v\|_{\mathbb L^2(\Omega)} \|\mathbf w\|_{\mathbb L^4(\Omega)},\label{b_estimate}\\
|\mathbf b(\mathbf u, \mathbf v, \mathbf w)| &\leq& C \|\mathbf u\|_{\mathbb L^2(\Omega)}^{1/2} \|\nabla \mathbf u\|_{\mathbb L^2(\Omega)}^{1/2} \|\nabla \mathbf v\|_{\mathbb L^2(\Omega)} \|\mathbf w\|_{\mathbb L^2(\Omega)}^{1/2} \|\nabla \mathbf w\|_{\mathbb L^2(\Omega)}^{1/2}\label{pre:estimate}
\end{eqnarray}
for any $(\mathbf u, \mathbf w, \mathbf v)\in \mathbb{X}\times \mathbb{X}\times \mathbb{H}^1(\Omega)$, by using the following interpolation inequality (cf. \cite[Lemma 3.2]{vexler_wagner_2024}) 
\begin{align}
\|z\|_{L^4(\Omega)} \leq C \|z\|_{L^2(\Omega)}^{1/2} \| z\|_{H^1(\Omega)}^{1/2}\quad \forall z \in H^1(\Omega).\label{Interpolation:conti:1}
\end{align}
The properties of $\mathbf b$ also hold for the trilinear form \( b\). 
Throughout this article, the positive constants \( C \)  may vary on different occasions. We abbreviate the duality pairings \( \langle \cdot, \cdot \rangle_{H^1(\Omega), H^{-1}(\Omega)} \) and \( \langle \cdot, \cdot \rangle_{\mathbb X, \mathbb X^*} \) to \( \langle \cdot, \cdot \rangle \), if this does not lead to ambiguity.

\section{Optimal control problem}\setcounter{equation}{0}
In the following subsection \ref{subsec:eq}-\ref{subsec:Regular}, we will derive the optimality condition for the optimal control problem $\mathrm{(P)}$ and analyze the regularity of the optimal solutions.

\subsection{The Boussinesq equations}\label{subsec:eq}
First, we give the definition of the weak solution to Boussinesq system \eqref{intr:Rb_state}.
\begin{definition}\label{def:1}For any $u \in L^2(I; H^{-\frac{1}{2}}(\Gamma))$, $f\in L^2(I;H^{1}(\Omega)^*)$,  $\mathbf h\in L^2(I;\mathbb X^*)$, $\mathbf y_0\in \mathbb H$ and $\theta_0\in L^2(\Omega)$, a pair $(\mathbf{y}, \theta) \in \mathbb{W}$ is called the weak solution to system \eqref{intr:Rb_state} if they satisfy 
\begin{subequations}\label{Exist:weak:solution}
\begin{align}
&\langle \partial_t  \mathbf y,\mathbf v\rangle+\nu \mathbf a(\mathbf y,\mathbf v)+\mathbf b(\mathbf y,\mathbf y,\mathbf v)+\beta(\theta \mathbf g,\mathbf v)= \langle \mathbf h,\mathbf v\rangle,\label{Exist:weak:solution:a}\\
&\langle  \partial_t \theta, \psi\rangle+\chi a(\theta, \psi)+b(\mathbf y, \theta,\psi)+\eta \gamma\langle \theta, \psi\rangle_\Gamma=\langle f,\psi\rangle+\eta\langle u, \psi\rangle_\Gamma,\label{Exist:weak:solution:c}\\
&\mathbf y(0)=\mathbf y_0,\quad \theta(0)=\theta_0
\end{align}
\end{subequations}
for any $(\mathbf v,\psi)\in \mathbb X\times H^1(\Omega)$ and almost all (a.a.) $t\in (0,T)$. 
\end{definition}

The following lemma (cf. \cite[Theorem 2.1, p. 151]{Hinze_Matthes_2007}) reveals the existence and uniqueness of weak solutions to the system \eqref{intr:Rb_state}. 
\begin{lemma}\label{thm:exist:state:1}
The problem \eqref{intr:Rb_state} admits a unique weak solution $(\mathbf y, \theta )\in \mathbb W$ such that 
\begin{align}
&\|\mathbf y\|_{\mathbb W(I)}+\| \theta\|_{W(I)}\leq C\Big(\|f\|_{L^2(I;H^{1}(\Omega)^*)},\|\mathbf h\|_{L^2(I;\mathbb X^*)},\|\mathbf y_0\|_{\mathbb L^2(\Omega)},\|\theta_0\|_{L^2(\Omega)},\|u\|_{L^2(I;H^{-\frac{1}{2}}(\Gamma))}\Big).\label{weak:est:1}
\end{align}
\end{lemma}


Denote by $\mathcal S$ the control-to-state mapping, defined by $\mathcal{S}: L^2(I;L^2(\Gamma))\rightarrow {\mathbb W}$ such that $u\mapsto \mathcal{S}(u):=(\mathbf y_u, \theta_u)$, where $(\mathbf y_u, \theta_u)$ is the unique pair of solutions of \eqref{intr:Rb_state} defined in Definition \ref{def:1}.
The following lemma states the weak--weak continuity and differentiability of the mapping $\mathcal S$.

\begin{lemma}\label{thm:Weak_Weak:S}Let $u_n \rightharpoonup  u$ in $L^2(I;L^2(\Omega))$, then  we have  $(\mathbf y_{u_n}, \theta_{u_n}) \rightharpoonup (\mathbf y_{u}, \theta_{u})$ in the space $\mathbb W$.
\begin{proof} 
From Lemma \ref{thm:exist:state:1} it follows that the sequence $\{(\mathbf{y}_{u_n}, \theta_{u_n})\}_n$ is bounded in the space $\mathbb{W}$. Therefore, we can extract a subsequence that weakly converges to some $(\mathbf{y}, \theta) \in \mathbb{W}$. Subsequently, passing to the limit in \eqref{Exist:weak:solution} and by the uniqueness we can show that $(\mathbf{y}, \theta) = (\mathbf{y}_u, \theta_u)$. Given the uniqueness of the solution to the system \eqref{Exist:weak:solution}, we can conclude that the entire sequence converges. More details can be found in  \cite{Hinze_Matthes_2007}.
\end{proof}

\end{lemma}

\begin{proposition}\label{thm:Lineared:state}
 The mapping $\mathcal S$ is infinitely often Fr\'echet-differentiable. If we set $(\mathbf z_{u,v}, \xi_{u,v})=\mathcal{S}'(u)v$, $(\mathbf \mathbf z_{u,vv}, \xi_{u,vv})=\mathcal{S}''(u)v^2$. Then \( (\mathbf z_{u,v}, \xi_{u,v}) \) and $(\mathbf  z_{u,vv}, \xi_{u,vv})$ solve uniquely the following system:
\begin{equation}\label{Diff:first:deri}
\left\{\begin{aligned}
&\langle\partial_t  \mathbf z_{u,v},\mathbf v\rangle+\nu \mathbf a(\mathbf z_{u,v} ,\mathbf v)+\mathbf b(\mathbf z_{u,v},\mathbf y_u,\mathbf v)+\mathbf b(\mathbf y_{u},\mathbf z_{u,v},\mathbf v)+\beta(\xi_{u,v} \mathbf g,\mathbf v)= 0,\\
&\langle \partial_t  \xi_{u,v}, \psi\rangle+\chi a(\xi_{u,v}, \psi)+b(\mathbf z_{u,v}, \theta_u,\psi)+b(\mathbf y_u, \xi_{u,v},\psi)+\eta \gamma\langle \xi_{u,v}, \psi\rangle_\Gamma=\eta\langle v, \psi\rangle_\Gamma,\\
&\mathbf  z_{u,v}(0)=0,\quad \xi_{u,v}(0)=0
\end{aligned}\right.
\end{equation}
and
\begin{equation}\label{Diff:second:deri}
\left\{\begin{aligned}
&\langle  \partial_t  \mathbf z_{u,vv},\mathbf v\rangle+\nu \mathbf a(\mathbf z_{u,vv} ,\mathbf v)+\mathbf b(\mathbf z_{u,vv},\mathbf y_u,\mathbf v)+\mathbf b(\mathbf y_{u},\mathbf z_{u,vv},\mathbf v)+\beta(\xi_{u,vv} \mathbf g,\mathbf v)= -2\mathbf b(\mathbf z_{u,v},\mathbf z_{u,v},\mathbf v),\\
&\langle \partial_t \xi_{u,vv}, \psi\rangle+\chi a(\xi_{u,vv}, \psi)+b(\mathbf z_{u,vv}, \theta_u,\psi)+b(\mathbf y_u, \xi_{u,vv},\psi)+\eta \gamma\langle \xi_{u,vv}, \psi\rangle_\Gamma=-2 b(\mathbf z_{u,v},\xi_{u,v},\psi),\\
&\mathbf  z_{u,vv}(0)=0,\quad \xi_{u,vv}(0)=0
\end{aligned}\right.
\end{equation}
for any $(\mathbf v, \psi)\in \mathbb X\times H^1(\Omega)$ and a.a. $t\in(0,T)$, respectively.
\end{proposition}
\begin{proof} 
The details of this proof can be found in \cite{BarwolffHinze2007}.
\end{proof}

\subsection{The optimality conditions}
 In this subsection, we intend to derive the first- and second-order optimality conditions. To do this, we introduce the reduced functional ${J}(u) := J(u,\mathcal{S}(u))$ with the help of the control-to-state mapping $\mathcal{S}$. 

\begin{theorem}
\label{thm:first:second} The reduced functional ${J}: L^2(I;L^2(\Gamma))\rightarrow\mathbb R $ is of class $C^\infty$. For any $u,v\in L^2(I;L^2(\Gamma))$, the first and second order derivatives of $ J$ satisfy
\begin{align}
{J}'(u)v&=\int_{\Sigma_T}(\eta\kappa_u+\alpha u)v\, ds dt,\label{thm:J:diff1}\\\
{J}''(u)v^2&=\int_{\Omega_T}\left(|\mathbf z_{u,v}|^2+|\xi_{u,v}|^2-2(\mathbf z_{u,v}\cdot\nabla)\mathbf z_{u,v}\cdot\bm \mu_u-2(\mathbf z_{u,v}\cdot\nabla)\xi_{u,v} \kappa_u\right)dxdt+\int_{\Sigma_T}\alpha |v|^2ds dt,\label{thm:J:diff2}
\end{align}
where $(\bm{\mu}_u, \kappa_u)\in L^2(I;\mathbb{X})\times L^2(I;H^1(\Omega))$ solves the adjoint system
\begin{subequations}\label{adjoint:weak:solution}
\begin{align}
&\langle -\partial_t \bm \mu,\mathbf w\rangle+\nu \mathbf a(\bm \mu,\mathbf w)+\mathbf b(\mathbf w ,\mathbf y,\bm \mu)+\mathbf b(\mathbf y ,\mathbf w,\bm \mu)+b(\mathbf w,\theta,\kappa)=(\mathbf y-\mathbf y_d,\mathbf w),\label{adjoint:Exist:weak:solution:a}\\
&\langle -\partial_t \kappa, \zeta\rangle+\chi a(\kappa, \zeta)+b(\mathbf y,  \zeta,\kappa)+\beta(\mathbf g\cdot\bm \mu,\zeta)+\eta \gamma \langle \kappa, \zeta\rangle_\Gamma=(\theta-\theta_d,\zeta),\label{adjoint:eq:d:311}\\
&\bm \mu(T)=\mathbf 0,\quad \kappa(T)=0
\end{align}
\end{subequations}
for any $(\mathbf w,\zeta)\in \mathbb X\times H^1(\Omega)$ and a.a. $t\in (0,T)$, and possesses the following regularity
\[
\bm{\mu}_u \in L^2(I; \mathbb{X}), \quad \partial_t \bm{\mu}_u \in L^{\frac{4}{3}}(I; \mathbb{X}^*) \cap \mathbb{W}(I)^*, \quad \kappa_u \in L^2(I; H^1(\Omega)), \quad \partial_t \kappa_u \in L^{\frac{4}{3}}(I; H^{1}(\Omega)^*) \cap W(I)^*,
\]
where $(\mathbf{y}, \theta)$ is the solution to \eqref{Exist:weak:solution}. Moreover, the following estimate holds:
\begin{align}\label{adjoint:stab: 312}
\|\bm \mu_u\|_{L^2(0, T; \mathbb H^{1}(\Omega))} &+ \|\kappa_u\|_{L^2(0, T; H^1(\Omega))}+\big\| \partial_t \bm\mu_u\big\|_{L^\frac{4}{3}(I;\mathbb{X}^*)} +
\big\|\partial_t  \kappa_u\big\|_{L^\frac{4}{3}(I;H^{1}(\Omega)^*)}\nonumber\\
&\leq C\Big(\|\mathbf y_u\|_{\mathbb W(I)},\|\theta_u\|_{W(I)}, \|\mathbf y_d\|_{L^2(I;L^2(\Omega))},\|\theta_d\|_{L^2(I;L^2(\Omega))} \Big).
\end{align}
\end{theorem}
\begin{proof}
The details of the proof can be found in \cite{BarwolffHinze2007}.
\end{proof}

Using standard arguments (cf. \cite{Troltzsh_2010}), it is easy to show that problem (P) has at least one optimal solution. Since problem (P) is non-convex, we usually consider local optimal solutions. A control \( \bar{u} \in \mathcal{U}_{ad} \) is called a local optimal control  of \( \mathrm{(P)} \) if there exists \( \rho > 0 \) such that \( J(\bar{u}) \leq J(u) \) for every \( u \in \mathcal{U}_{{ad}} \cap \mathcal{B}_{\rho}(\bar{u}) \), where \( \mathcal{B}_{\rho}(\bar{u}) \) denotes the open ball of \( L^2(I; L^2(\Gamma)) \) centered at \( \bar{u} \) with radius \( \rho \). The next theorem establishes the first-order necessary optimality condition for local minimum of $(\mathrm{P})$.
 
\begin{theorem} \label{Theorem:fisrt_order_condtion}Let \( \bar{u} \) be a local optimal control of problem $\mathrm{(P)}$ and  \( (\bar{\mathbf y},\bar{\theta}) \in \mathbb W\) be  the corresponding state. Then there exist unique adjoint states \( (\bar{\bm \mu},\bar{\kappa})\in L^2(I; \mathbb X)\times  L^2(I; H^{1}(\Omega))\), satisfying $\partial_t \bar{\bm{\mu}}\in L^{\frac{4}{3}}(I;\mathbb X^*)\cap \mathbb W(I)^*$ and $\partial_t \bar{k}\in L^{\frac{4}{3}}(I;H^{1}(\Omega)^*)\cap W(I)^*$, such that  for a.a. $t\in (0,T)$, 
\begin{equation}\label{first:optimal:state}
\left\{\begin{aligned}
&\langle  \partial_t \bar{\mathbf y},\mathbf v\rangle+\nu \mathbf a(\bar{\mathbf y},\mathbf v)+\mathbf b(\bar{\mathbf y},\bar{\mathbf y},\mathbf v)+\beta(\bar{\theta} \mathbf g,\mathbf v)= \langle \mathbf h,\mathbf v\rangle\quad \forall  \mathbf v \in \mathbb X,\\
&\langle  \partial_t \bar{\theta}, \psi\rangle+\chi a(\bar{\theta}, \psi)+b(\bar{\mathbf y}, \bar{\theta},\psi)+\eta \gamma\langle \bar{\theta}, \psi\rangle_\Gamma=\langle f,\psi\rangle+\eta\langle \bar{u}, \psi\rangle_\Gamma \quad  \forall \psi \in H^1(\Omega),\\
&\bar{\mathbf y}(0)=\mathbf y_0,\quad \bar{\theta}(0)=\theta_0,
\end{aligned}\right.
\end{equation}
\begin{equation}\label{first:optimal:adjoint}
\left\{\begin{aligned}
&\langle - \partial_t \bar{\bm  \mu},\mathbf w\rangle+\nu \mathbf a(\bar{\bm \mu},\mathbf w)+\mathbf b(\mathbf w ,\bar{\mathbf y},\bar{\bm \mu})+\mathbf b(\bar{\mathbf y} ,\mathbf w,\bar{\bm \mu})+b(\mathbf w,\bar{\theta},\bar{\kappa})=(\bar{\mathbf y}-\mathbf y_d,\mathbf w)\quad  \forall \mathbf w\in \mathbb X,\\
&\langle  -\partial_t \bar{\kappa}, \zeta\rangle+\chi a(\bar{\kappa}, \zeta)+b(\bar{\mathbf y}, \zeta,\bar{\kappa})+\beta(\mathbf g\cdot\bar{\bm \mu},\zeta)+\eta \gamma\langle \bar{\kappa}, \zeta\rangle_\Gamma=(\bar{\theta}-\theta_d,\zeta) \quad  \forall \zeta\in H^1(\Omega),\\
&\bar{\bm \mu}(T)=\mathbf 0,\quad \bar{\kappa}(T)=0
\end{aligned}\right.
\end{equation}
\begin{align}
\int_{\Sigma_T}(\eta\bar{\kappa}+\alpha \bar{u})(u-\bar{u})\,dsdt\geq 0\quad \forall u\in \mathcal U_{ad}.\label{first_cond:varition_ineq}
\end{align}
\end{theorem}
\begin{remark} Equivalently, using the projection operator, the inequality \eqref{first_cond:varition_ineq} can be rewritten as
\begin{align}\label{Proj:equation}
\bar{u}(t,x)=\mathrm{Proj}_{[u_a,u_b]}\Big(-\frac{\eta}{\alpha}\bar{\kappa}(t,x)\Big)\quad \text{for a.e. } (t, x)\in \Sigma_T.
\end{align}
\end{remark}
Before investigating the second-order necessary and sufficient optimality conditions, we first study the Lipschitz continuity of the solutions to the state and adjoint systems with respect to the control. The following lemma can be obtained using standard energy estimates; we refer to \cite{Hinze_Matthes_2007} for details.
\begin{lemma}\label{lem:estimate:Y}
There exists a constant $C_\rho>0$ such that for any $u_1, u_2\in \mathcal B_\rho(\bar{u})\cap\mathcal{U}_{ad}$, there holds
\begin{align}
\|\mathbf y_{u_1}-\mathbf y_{u_2}\|_{L^{\infty}(I;\mathbb L^2(\Omega))}&+\|\theta_{u_1}-\theta_{u_2}\|_{L^\infty(I;L^2(\Omega))}+\|\mathbf y_{u_1}-\mathbf y_{u_2}\|_{L^{2}(I;\mathbb H^1(\Omega))}+\|\theta_{u_1}-\theta_{u_2}\|_{L^2(I;H^1(\Omega))}\nonumber\\
&\leq C_\rho\|u_1-u_2\|_{L^2(I;L^2(\Gamma))}.\nonumber
\end{align}
\end{lemma}

\begin{lemma}\label{lem:estimate:Z}
There exists a constant $C_\rho>0$ such that for all $u_1, u_2 \in \mathcal{B}_\rho(\bar{u}) \cap \mathcal{U}_{ad}$ and all $v \in L^2(I; L^2(\Gamma))$, it holds that
\begin{align}\label{estimate:Z:results}
\|\mathbf z_{u_1,v}-\mathbf z_{u_2,v}\|_{L^\infty(I;\mathbb L^2(\Omega))}&+\|\mathbf z_{u_1,v}-\mathbf z_{u_2,v}\|_{L^2(I;\mathbb H^{1}(\Omega))}+\|\xi_{u_1,v}-\xi_{u_2,v}\|_{L^\infty(I; L^2(\Omega))}\nonumber\\
&+\|\xi_{u_1,v}-\xi_{u_2,v}\|_{L^2(I;H^1(\Omega))}\leq C_\rho\|u_1-u_2\|_{L^2(I;L^2(\Gamma))}\|v\|_{L^2(I;L^2(\Gamma))}.
\end{align}
\end{lemma}
\begin{proof}
Let \( (\mathbf{e}_{\mathbf z}, e_\xi) = (\mathbf{z}_{u_1,v} - \mathbf{z}_{u_2,v}, \xi_{u_1,v} - \xi_{u_2,v}) \). From \eqref{Diff:first:deri}, we know that \( (\mathbf{e}_{\mathbf z}, e_\xi) \) satisfies
\begin{equation}\label{estimate:Z:1}
\left\{
\begin{aligned}
&\langle  \partial_t \mathbf{e}_{\mathbf z}, \mathbf{v} \rangle + \nu \mathbf{a}(\mathbf{e}_{\mathbf z}, \mathbf{v}) + \mathbf{b}(\mathbf{e}_{\mathbf z}, \mathbf{y}_{u_1}, \mathbf{v})+ \mathbf{b}(\mathbf{z}_{u_2,v}, \mathbf{y}_{u_1} - \mathbf{y}_{u_2}, \mathbf{v}) + \mathbf{b}(\mathbf{y}_{u_1} - \mathbf{y}_{u_2}, \mathbf{z}_{u_1,v}, \mathbf{v}) \\
&+ \mathbf{b}(\mathbf{y}_{u_2}, \mathbf{e}_{\mathbf z}, \mathbf{v}) + \beta ({e}_\xi \mathbf{g}, \mathbf{v}) = 0 \quad \forall \mathbf{v} \in \mathbb{X}, \\
&\langle \partial_t {e}_\xi, \psi \rangle + \chi a(e_\xi, \psi) + b(\mathbf{e}_{\mathbf z}, \theta_{u_1}, \psi) + b(\mathbf{z}_{u_2,v}, \theta_{u_1}-\theta_{u_2}, \psi) + b(\mathbf{y}_{u_1} - \mathbf{y}_{u_2}, \xi_{u_1,v}, \psi) \\
&+ b(\mathbf{y}_{u_2}, e_\xi, \psi) + \eta \gamma \langle {e}_\xi, \psi \rangle_\Gamma = 0 \quad \forall \psi \in H^1(\Omega), \\
&\mathbf{e}_z(0) = 0, \quad  {e}_\xi(0) = 0.
\end{aligned}
\right.
\end{equation}
Taking \((\mathbf v,\psi)=(\mathbf e_{\mathbf z}, e_\xi ) \) in \eqref{estimate:Z:1} and integrating over the interval \( (0, t) \) for any $t\in (0,T)$, there holds
\begin{align}\label{estimate:Z:2}
&\frac{1}{2}\big(\|\mathbf e_{\mathbf z}(t)\|^2+\|e_\xi(t)\|^2\big)+\nu\int_{0}^t\|\nabla \mathbf e_{\mathbf z} \|^2ds+\chi\int_0^t\|\nabla e_\xi\|^2ds+\gamma \eta \int_0^t\|e_\xi\|^2_\Gamma ds\nonumber\\
&\leq  \int_{0}^t |\mathbf{b}(\bm{e}_{\mathbf z}, \mathbf{y}_{u_1}, \mathbf e_{\mathbf z})|ds+ \int_{0}^t|\mathbf{b}(\mathbf{z}_{u_2,v}, \mathbf{y}_{u_1} - \mathbf{y}_{u_2}, \mathbf e_{\mathbf z})|ds + \int_{0}^t|\mathbf{b}(\mathbf{y}_{u_1} - \mathbf{y}_{u_2}, \mathbf{z}_{u_1,v}, \mathbf e_{\mathbf z})|ds +\int_0^t |\beta({e}_\xi \mathbf{g}, \mathbf e_{\mathbf z})|ds\nonumber\\
&+\int_0^t |b(\bm{e}_{\mathbf z}, \theta_{u_1}, e_\xi)|ds + \int_0^t|b(\mathbf{z}_{u_2,v}, \theta_{u_1}-\theta_{u_2}, e_\xi)|ds+\int_0^t |b(\mathbf{y}_{u_1} - \mathbf{y}_{u_2}, \xi_{u_1,v}, e_\xi)|ds=:\sum_{i=1}^{7}J_i.
\end{align}
In the following, we estimate each term on the right-hand side of \eqref{estimate:Z:2}. For the term $J_1$, we have
\begin{align*}
J_1 &\leq C \int_0^t \|\mathbf{e}_{\mathbf z}\|\|\nabla \mathbf{e}_{\mathbf z}\| \|\nabla \mathbf{y}_{u_1}\|  ds 
    \leq \frac{C}{\nu} \int_0^t \|\nabla \mathbf{y}_{u_1}\|^2 \|\mathbf{e}_{\mathbf z}\|^2 \, ds + \frac{\nu}{6} \int_0^t \|\nabla \mathbf{e}_{\mathbf z}\|^2\, ds,
\end{align*}
where inequality \eqref{pre:estimate} is used. Using \eqref{b_estimate}, the term $J_2+J_3$ can be estimated by 
\begin{equation}
\begin{aligned}\label{estimate:Z:4}
J_2+J_3
&\leq  C\int_{0}^t\|\nabla\mathbf{z}_{u_2,v}\|\|\nabla(\mathbf y_{u_1}-\mathbf y_{u_2})\|\|\nabla \mathbf e_{z}\|ds+  C\int_{0}^t\|\nabla\mathbf{z}_{u_1,v}\|\|\nabla(\mathbf y_{u_1}-\mathbf y_{u_2})\|\|\nabla \mathbf e_{{\mathbf z}}\|ds\\
&\leq  \frac{\nu}{6}\int_{0}^t\|\nabla \mathbf e_{\mathbf z}\|^2ds+C\Vert u_1-u_2\Vert^2_{L^2(I;L^2(\Gamma))}\Big(\|\nabla\mathbf{z}_{u_2,v}\|^2_{L^2(I;\mathbb L^2(\Omega))}+\|\nabla\mathbf{z}_{u_1,v}\|^2_{L^2(I;\mathbb L^2(\Omega))}\Big),
\end{aligned}
\end{equation}
where we have used Lemma \ref{lem:estimate:Y}. By Young's inequality, the term $J_4$ can be bounded as
\begin{align*}
J_4 \leq C \Big( \int_0^t \|\mathbf{e_{\mathbf z}}\|^2 \, ds + \int_0^t \|e_\xi\|^2 \, ds \Big).
\end{align*}
By the Poincar$\mathrm{\acute{e}}$ inequality, the  term $J_5$ can be bounded as
\begin{align*}
J_5&\leq C\int_{0}^t\|\mathbf e_{\mathbf z}\|^{\frac{1}{2}}\|\nabla\mathbf e_{\mathbf z}\|^{\frac{1}{2}}\|e_\xi\|^\frac{1}{2}\| e_\xi\|^{\frac{1}{2}}_{H^1(\Omega)}\|\nabla \theta_{u_1}\|ds\nonumber\\
&\leq C \int_0^t \|\mathbf e_{\mathbf z}\|\|\nabla\mathbf e_{\mathbf z}\|\|\nabla \theta_{u_1}\|ds+C\int_0^t \|e_\xi\|\| e_\xi\|_{H^1(\Omega)}\|\nabla \theta_{u_1}\|ds\nonumber\\
&\leq\frac{\nu}{6}\int_0^t\|\nabla\mathbf e_{\mathbf z}\|^2ds +\frac{C}{\nu}\int_0^t\|\mathbf e_{\mathbf z}\|^2\|\nabla \theta_{u_1}\|^2ds+\frac{C}{\min\{\chi,\gamma\eta\}}\int_0^t\|e_\xi\|^2\|\nabla \theta_{u_1}\|^2ds\nonumber\\
&+\frac{\min\{\chi,\gamma\eta\}}{4}\int_0^t\Big(\|\nabla e_\xi\|^2+\|e_\xi\|^2_\Gamma\Big)ds.
\end{align*}
Similar to \eqref{estimate:Z:4}, the term $J_6+J_7$ can be bounded as
\begin{align*}
J_6+J_7&\leq \frac{\min\{\chi,\gamma\eta\}}{4}\int_0^t\Big(\|\nabla e_\xi\|^2+\|e_\xi\|^2_\Gamma\Big)ds+\frac{C}{\min\{\chi,\gamma\eta\}}\Vert u_1-u_2\Vert^2_{L^2(I;L^2(\Gamma))}\Big(\|\nabla\mathbf{z}_{u_2,v}\|^2_{L^2(I;\mathbb L^2(\Omega))}\\
&+\|{\xi}_{u_1,v}\|^2_{L^2(I; H^1(\Omega))}\Big).
\end{align*}
Combining the estimates for terms $J_1-J_7$ with \eqref{estimate:Z:2} and using the continuity of the operator $\mathcal{S}'(u)$,  we obtain \eqref{estimate:Z:results} via Gronwall’s inequality, which completes the proof.
\end{proof}

\begin{lemma}\label{lem:Mu_Kappa}
There exists a constant $C_\rho>0$ such that for any $ u_1, u_2\in \mathcal B_\rho(\bar{u})\cap\mathcal{U}_{ad}$, one has
\begin{align}\label{lem:estimate:Mu_Kappa}
\|\bm \mu_{u_1}-\bm \mu_{u_2}\|_{L^2(I;\mathbb H^1(\Omega))}+\|\kappa_{u_1}-\kappa_{u_2}\|_{L^2(I;H^1(\Omega))}\leq C_\rho\|u_1-u_2\|_{L^2(I;L^2(\Gamma))}.
\end{align}
\end{lemma}
\begin{proof}
Let \( (\mathbf{e}_\mu, e_\kappa) := (\bm{\mu}_{u_1} - \bm {\mu}_{u_2}, \kappa_{u_1} - \kappa_{u_2}) \), from \eqref{adjoint:weak:solution} we know that \( (\mathbf{e}_\mu, e_\kappa) \) satisfies
\begin{equation}\label{estimae:Mu_Kappa:equations}
\left\{\begin{aligned}
&\langle - \partial_t \mathbf  e_\mu,\mathbf w\rangle+\nu \mathbf a(\mathbf e_\mu,\mathbf w)+\mathbf b(\mathbf w ,\mathbf y_{u_2},\mathbf e_\mu)+\mathbf b(\mathbf y_{u_2} ,\mathbf w,\mathbf e_\mu)+b(\mathbf w,\theta_{u_2},e_\kappa)=(\mathbf y_{u_1}-\mathbf y_{u_2},\mathbf w)\\
&-\mathbf b(\mathbf w, \mathbf y_{u_1}-\mathbf y_{u_2},\bm \mu_{u_1})-\mathbf b(\mathbf y_{u_1}-\mathbf y_{u_2},\mathbf w, \bm \mu_{u_1})-b(\mathbf w,\theta_{u_1}-\theta_{u_2},\kappa_{u_1}),\\
&\langle   -\partial_t e_\kappa, \zeta\rangle+\chi a(e_\kappa, \zeta)+b(\mathbf y_{u_2}, \zeta,e_\kappa)+\beta(\mathbf g\cdot\mathbf e_\mu,\zeta)+\eta \gamma\langle e_\kappa, \zeta\rangle_\Gamma=(\theta_{u_1}-\theta_{u_2},\zeta)-b(\mathbf y_{u_1}-\mathbf y_{u_2},\zeta,\kappa_{u_1}),\\
&\mathbf e_\mu(T)=\mathbf 0,\quad e_\kappa(T)=0
\end{aligned}\right.
\end{equation}
for any $ (\mathbf  w, \zeta)\in \mathbb W$. For any \(\mathbf{w} \in \mathbb{W}(I)\), we have
\begin{align}
\int_I\mathbf b(\mathbf w, \mathbf y_{u_1}-\mathbf y_{u_2},\bm \mu_{u_1})dt &\leq \int_I \|\mathbf w\|_{\mathbb L^4(\Omega)}\|\nabla(\mathbf y_{u_1}-\mathbf y_{u_2})\|\|\bm \mu_{u_1}\|_{\mathbb L^4(\Omega)}dt\nonumber\\
&\leq \int_I\|\mathbf w\|^\frac{1}{2}\|\nabla \mathbf w\|^\frac{1}{2}\|\nabla(\mathbf y_{u_1}-\mathbf y_{u_2})\|\|\bm \mu_{u_1}\|^\frac{1}{2}\|\nabla \bm \mu_{u_1}\|^\frac{1}{2}dt \label{estimae:Mu_Kappa:1}\\
&\leq \|\nabla(\mathbf y_{u_1}-\mathbf y_{u_2})\|_{L^2(I;\mathbb L^2(\Omega))} \|\mathbf w\|_{\mathbb W(I)}\|\bm \mu_{u_1}\|^{\frac{1}{2}}_{L^{\infty}(I;\mathbb L^2(\Omega))}\|\nabla \bm \mu_{u_1}\|^{\frac{1}{2}}_{L^{2}(I;\mathbb L^2(\Omega))},\nonumber
\end{align}
\begin{align}
&\int_I\mathbf b(\mathbf y_{u_1}-\mathbf y_{u_2},\mathbf w, \bm \mu_{u_1})dt \leq \int_I \|\mathbf y_{u_1}-\mathbf y_{u_2}\|_{\mathbb L^4(\Omega)}\|\nabla \mathbf w\|\|\bm \mu_{u_1}\|_{\mathbb L^4(\Omega)}dt\label{estimae:Mu_Kappa:2}\\
&\leq \|\mathbf y_{u_1}-\mathbf y_{u_2}\|^{\frac{1}{2}}_{L^\infty(I;\mathbb L^2(\Omega))}\|\nabla(\mathbf y_{u_1}-\mathbf y_{u_2})\|^{\frac{1}{2}}_{L^2(I;\mathbb L^2(\Omega))} \|\nabla\mathbf w\|_{L^2(I;\mathbb L^2(\Omega))}\|\bm \mu_{u_1}\|^{\frac{1}{2}}_{L^{\infty}(I;\mathbb L^2(\Omega))}\|\nabla \bm \mu_{u_1}\|^{\frac{1}{2}}_{L^{2}(I;\mathbb L^2(\Omega))}\nonumber
\end{align}
and
\begin{align}\label{estimae:Mu_Kappa:3}
\int_I b(\mathbf w,\theta_{u_1}-\theta_{u_2},\kappa_{u_1})ds\leq \|\nabla( \theta_{u_1}- \theta_{u_2})\|_{L^2(I;\mathbb L^2(\Omega))} \|\mathbf w\|_{\mathbb W(I)}\| \kappa_{u_1}\|^{\frac{1}{2}}_{L^{\infty}(I; L^2(\Omega))}\| \kappa_{u_1}\|^{\frac{1}{2}}_{L^{2}(I; H^1(\Omega))}.
\end{align}
Similarly, for any \(\psi \in W(I)\) we have
\begin{align}\label{estimae:Mu_Kappa:4}
  \int_I b(\mathbf y_{u_1}-\mathbf y_{u_2},\psi,\kappa_{u_1}) dt&\leq \|\mathbf y_{u_1}-\mathbf y_{u_2}\|^{\frac{1}{2}}_{L^\infty(I;\mathbb L^2(\Omega))}\|\nabla(\mathbf y_{u_1}-\mathbf y_{u_2})\|^{\frac{1}{2}}_{L^2(I;\mathbb L^2(\Omega))} \|\nabla\psi\|_{L^2(I;\mathbb L^2(\Omega))}\nonumber\\
  &\| \kappa_{u_1}\|^{\frac{1}{2}}_{L^{\infty}(I; L^2(\Omega))}\|\kappa_{u_1}\|^{\frac{1}{2}}_{L^{2}(I; H^1(\Omega))}.
\end{align}
Collecting the estimates  \eqref{estimae:Mu_Kappa:1}--\eqref{estimae:Mu_Kappa:4}, we draw the conclusion that the right-hand sides of \eqref{estimae:Mu_Kappa:equations} form  linear functionals satisfying $$(\mathbf y_{u_1}-\mathbf y_{u_2},\cdot)-\mathbf b(\cdot, \mathbf y_{u_1}-\mathbf y_{u_2},\bm \mu_{u_1})-\mathbf b(\mathbf y_{u_1}-\mathbf y_{u_2},\cdot, \bm \mu_{u_1})-b(\cdot,\theta_{u_1}-\theta_{u_2},\kappa_{u_1})\in {\mathbb W^*(I)},$$ $$(\theta_{u_1}-\theta_{u_2},\cdot)-b(\mathbf y_{u_1}-\mathbf y_{u_2},\cdot,\kappa_{u_1})\in W^*(I).$$
Thus, combining Lemma 1 (\cite[p. 5]{BarwolffHinze2007}) and Lemma \ref{lem:estimate:Y} we can obtain \eqref{lem:estimate:Mu_Kappa}. This finishes the proof.
\end{proof}

Combining the previous theorems and lemmas, the following lemma can be derived through standard inequality estimates.
\begin{lemma}\label{lem:second_first_stab_Lip}
There exist positive constants $C^1_\rho$ and $C^2_\rho$, such that for any $u\in \mathcal B_\rho(\bar{u})\cap\mathcal{U}_{ad}$ and $v\in L^2(I;L^2(\Gamma))$ there holds
\begin{align}\label{J:estimate:stab}
\big|J'(u)v\big|\leq C^1_\rho\|v\|_{L^2(I;L^2(\Gamma))}\quad\text{and}\quad\big|J''(u)v^2\big|\leq C^2_\rho\|v\|^2_{L^2(I;L^2(\Gamma))}.
\end{align}
Furthermore, there exist positive constants $\widetilde{C}^1_\rho$ and $\widetilde{C}^2_\rho$, such that for any $u_1,u_2\in \mathcal B_\rho(\bar{u})\cap\mathcal{U}_{ad}$ there hold
\begin{align}
\big|\big(J'(u_1)-J'(u_2)\big)v\big| &\leq \widetilde{C}^1_\rho\|u_1-u_2\|_{L^2(I;L^2(\Gamma))}\|v\|_{L^2(I;L^2(\Gamma))},\label{J:estimate:first:Lip}\\
\big|\big(J''(u_1)-J''(u_2)\big)v^2 \big| &\leq \widetilde{C}^2_\rho\|u_1-u_2\|_{L^2(I;L^2(\Gamma))}\|v\|^2_{L^2(I;L^2(\Gamma))}.\label{J:estimate:second:Lip}
\end{align}
\end{lemma}

The following lemma asserts that $J''(u)$ is a Legendre form (see \cite{BonnansShapiro2000}).
\begin{lemma}\label{lem:conv:scond} Let $v_n\rightharpoonup v$ in $L^2(I;L^2(\Gamma))$ as $n\to \infty$, then $J''(u)v^2\leq \underset{n\to \infty}{\lim \inf} J''(u)v^2_n$. If, in addition,  we assume $ \lim\limits_{n\to \infty} J''(u)v^2_n=J''(u)v^2$ for $v_n\rightharpoonup v$ in $L^2(I;L^2(\Gamma))$, then $\lim\limits_{n\to \infty} \|v_n-v\|_{L^2(I;L^2(\Gamma))}=0$.
\end{lemma}
\begin{proof}
A standard compactness argument suffices here; thus, we omit the proof.
\end{proof}

The second-order optimality condition plays an essential role in the following a priori error estimate. However, to derive second-order optimality conditions with minimal gap, it has to be written along the directions $v\in T_{\mathcal{U}_{ad}}(\bar u)$, where $T_{\mathcal{U}_{ad}}(\bar u)$ is the tangent cone at $\bar u$ to $\mathcal{U}_{ad}$ such that $ J^{\prime}(\bar u)v=(\eta\bar{\kappa}+\alpha \bar{u},v)_{L^2(I;L^2(\Gamma))}=0$ for any $v\in T_{\mathcal{U}_{ad}}(\bar u)$. 
In order to characterize these directions, we introduce the cone of critical directions as follows:
\begin{align}
&\mathcal{C}_{\bar{u}}=\big\{v\in L^2(I;L^2(\Gamma)): v \text{ satisfies \eqref{cone:cond:1}--\eqref{cone:cond:3}} \big\},\label{cone:cond:0}\\
&v\geq 0 \quad \text{ if } -\infty <u_a=\bar{u}\ \mbox{ and}\ \eta\bar{\kappa}+\alpha \bar{u}= 0,\label{cone:cond:1}\\
&v\leq 0 \quad \text{ if }  \bar{u}=u_b<+\infty \ \mbox{ and}\ \eta\bar{\kappa}+\alpha \bar{u}= 0,\label{cone:cond:2}\\
&v=0\quad \text{ if } \eta\bar{\kappa}+\alpha \bar{u}\neq 0,\label{cone:cond:3}
\end{align}
where $\bar\kappa$ satisfies the system \eqref{first:optimal:adjoint}.

\begin{theorem}  
If  $\bar{u}$ is a local solution to the problem $\mathrm{(P)}$, then $J''(\bar{u})v^2\geq 0$ for any $v\in \mathcal{C}_{\bar{u}}$.
Conversely, let $\bar{u}\in \mathcal{U}_{ad}$ satisfy the system \eqref{first:optimal:state}-\eqref{first_cond:varition_ineq}. Suppose that 
\begin{alignat}{2}
J''(\bar{u})v^2> 0 \label{second_con}
\end{alignat}
for any $v\in \mathcal{C}_{\bar{u}}\text{\textbackslash}\{0\}$. Then, there exist \( \rho > 0 \) and \( \vartheta> 0 \) such that
\begin{align}\label{scond:cond:ZengZhang}
J(\bar{u})+\frac{\vartheta}{2}\|u-\bar{u}\|^2_{L^2(I;L^2(\Gamma))}\leq J(u)\quad \forall u \in \mathcal B_\rho(\bar{u})\cap\mathcal{U}_{ad}.
\end{align}
\end{theorem}
\begin{proof} 
Applying Lemma \ref{lem:second_first_stab_Lip} and Lemma \ref{lem:conv:scond}, together  with the standard arguments detailed in Theorem 2.4 (p. 177) and Theorem 2.6 (p. 178) of \cite{Casas_Troltzsch_2012}, the theorem can be proved.
\end{proof}
\begin{remark}The optimality condition \eqref{second_con} is equivalent to the following condition:
\begin{align}\label{second:equi:cond}
J''(\bar{u})v^2\geq \vartheta\|v\|^2_{L^2(I;L^2(\Gamma))}\quad \forall v\in \mathcal C_{\bar{u}}.
\end{align}
Indeed, let us observe that  \eqref{scond:cond:ZengZhang}  implies that  $\bar{u}$ is a local solution of the following problem
\begin{align*}
\min\limits_{u\in \mathcal U_{ad}\cap \mathcal B_\rho(\bar{u})} J_\vartheta(u)=J(u)-\frac{\vartheta}{2}\|u-\bar{u}\|^2_{L^2(I;L^2(\Gamma))}.
\end{align*}
Therefore, from the second-order necessary condition we obtain $J''_\vartheta(\bar{u})v^2\geq0$ for every $v\in \mathcal{C}_{\bar{u}}$. It  suffices to note that $J''_\vartheta(\bar{u})v^2=J''(\bar{u})v^2-\vartheta \|v\|^2_{L^2(I;L^2(\Gamma))}\geq 0$ implies \eqref{second:equi:cond}.
\end{remark}

\subsection{Regularity of the optimal solutions}\label{subsec:Regular}
When the data of the state equation \eqref{intr:Rb_state} and adjoint equation \eqref{adjoint:weak:solution}  have a certain smoothness, we can show that the corresponding solution possesses better regularity.

\begin{theorem}\label{lem:conL2:Regualrit:YTheta}
Let $(\mathbf y, p)$ be a unique weak solution pair of the problem \eqref{intr:Rb_state}. Assume that $f\in L^2(I;L^2(\Omega))$, $\mathbf h\in L^2(I;\mathbb L^2(\Omega))$, $\mathbf y_0\in \mathbb X$, $u \in L^2(I;L^2(\Gamma))$ and $\theta_0\in H^{1}(\Omega)$. Then we have $\mathbf y\in \mathbb V(I)$, 
$\theta\in L^2(I; {H}^{\frac{3}{2}}(\Omega))\cap H^{\frac{3}{4}}(I; L^2(\Omega))$ and there exists  a  unique $p\in L^2(I; H^1(\Omega))\cap L^2(I;L^2_0(\Omega))$ such that
\begin{equation}\label{Regualrit:equa:reform}
\begin{aligned}
&(\partial_t  \mathbf y, \mathbf v)+\nu \mathbf a(\mathbf y,\mathbf  v)+\mathbf b(\mathbf y,\mathbf y,\mathbf v)-(p, \nabla\cdot\mathbf v)+\beta(\theta \mathbf g,\mathbf v)= (\mathbf h,\mathbf v)\quad \forall\,  \mathbf v \in \mathbb H^1_0(\Omega),~a.a.~ t\in (0,T)
\end{aligned}
\end{equation}
and $\mathbf y(0)=\mathbf y_0$,  and the following estimate holds:
\begin{align*}
&\|\mathbf y\|_{\mathbb V(I)}+\|p\|_{L^2(I;H^1(\Omega))}+\|\mathbf y\|_{C(\bar I;\mathbb X)}+\|\theta\|_{L^2(I; H^{\frac{3}{2}}(\Omega))}+\|\theta\|_{H^{\frac{3}{4}}(I; L^2(\Omega))}\nonumber\\
&\leq C\big({\mbox{$\|\mathbf h\|_{L^2(I;\mathbb L^2(\Omega))},\|\mathbf y_0\|_{\mathbb H^1(\Omega)},\|\theta_0\|_{H^1(\Omega)},\|f\|_{L^2(I;L^2(\Omega))},\|u\|_{L^2(I;L^2(\Gamma))}$}}\big).
\end{align*}
If, in addition, \( u \in H^{\frac{1}{4}}(I; L^2(\Gamma)) \cap L^2(I; H^{\frac{1}{2}}(\Gamma)) \), then \( \theta \in V(I) \) and
\begin{align*}
\|\theta\|_{V(I)}+\|\theta\|_{C(\bar{I};H^1(\Omega))}\leq C&\big({\mbox{$\|\mathbf y\|_{L^\infty(I;\mathbb H^1(\Omega))}, \|\theta\|_{L^2(I;H^\frac{3}{2}(\Omega))},\|f\|_{L^2(I;L^2(\Omega))},\|u\|_{H^\frac{1}{4}(I;L^2(\Gamma))},\|u\|_{ L^2(I;H^{\frac{1}{2}}(\Gamma))}$}}\big).
\end{align*}
In addition, let \( (\bm \mu, \kappa) \) be the unique solution pair of problem \eqref{adjoint:weak:solution}. Assuming \( \mathbf y_d \in L^2(I; \mathbb L^2(\Omega)) \), $\theta_d\in L^2(I;L^2(\Omega))$, then we have \( (\bm \mu, \kappa) \in \mathbb V \). Moreover, there exists a unique \( \lambda \in L^2(I; H^1(\Omega)) \cap L^2(I; L^2_0(\Omega)) \) such that for any $\mathbf w \in \mathbb H^1_0(\Omega)$ and $a.a. \,t \in (0,T)$
\begin{equation}\label{Regualrit:equa:reform:adjoint}
\begin{aligned}
( -\partial_t \bm \mu,\mathbf w)+\nu \mathbf a(\bm \mu,\mathbf w)+\mathbf b(\mathbf w ,\mathbf y,\bm \mu)+\mathbf b(\mathbf y ,\mathbf w,\bm \mu)+b(\mathbf w,\theta,\kappa)&-(\lambda, \nabla\cdot \mathbf w)=(\mathbf y-\mathbf y_d,\mathbf w)
\end{aligned}
\end{equation}
and $\bm \mu(T)=\mathbf 0$, and the following estimate holds:
\begin{align}\label{regular_MuKappa:ineq}
&\|\bm  \mu\|_{\mathbb V(I)}+\|\lambda\|_{L^2(I;H^1(\Omega))}+\|\kappa\|_{V(I)}+\|\bm \mu\|_{C(\bar{I};\mathbb X)}+\|\kappa\|_{C(\bar{I};H^1(\Omega))}\nonumber\\
&\leq C \big({\mbox{$\|\mathbf y\|_{\mathbb V(I)}, \|\theta\|_{V(I)}, \|\mathbf y_d\|_{L^2(I;\mathbb L^2(\Omega))},\|\theta_d\|_{L^2(I;L^2(\Omega))}$}}\big).
\end{align}
\end{theorem}
\begin{proof}
The desired regularity follows by applying a standard boot-strapping argument. Since $L^2(I;H^1(\Omega)) \hookrightarrow L^2(I; L^2(\Omega))$, it follows that $\beta \theta \mathbf{g} - \mathbf{h} \in L^2(I; \mathbb{L}^2(\Omega))$.  Consider the following system:
\begin{align*}
\langle \partial_t \mathbf{y}, \mathbf{v} \rangle + \nu \mathbf{a}(\mathbf{y}, \mathbf{v}) + \mathbf{b}(\mathbf{y}, \mathbf{y}, \mathbf{v}) = (\mathbf{h}, \mathbf{v}) - \beta (\theta \mathbf{g}, \mathbf{v}) \quad \forall \, \mathbf{v} \in \mathbb{X}, \quad \mathbf{y}(0) = \mathbf{y}_0.
\end{align*}
Standard estimates (cf. \cite[p. 465, Theorem 3.12]{vexler_wagner_2024}) imply that the velocity $\mathbf{y} \in \mathbb{V}(I)$ and the pressure $p\in L^2(I;H^1(\Omega))\cap L^2(I;L^2_0(\Omega))$. 

Since $H^{\frac{1}{2}}(\Omega) \hookrightarrow L^4(\Omega)$ (cf. \cite{Adam_2005}), for any $v \in L^2(I; H^{\frac{1}{2}}(\Omega))$ we have 
\begin{align*}
\int_{\Omega_T} (\mathbf{y} \cdot \nabla) \theta v \, dx \, dt 
\leq C\| \mathbf{y} \|_{L^\infty(I; L^4(\Omega))} \| \nabla \theta \|_{L^2(I; L^2(\Omega))} \| v \|_{L^2(I; H^{\frac{1}{2}}(\Omega))}.
\end{align*}
That is,  $(\mathbf{y} \cdot \nabla) \theta \in \left( L^2(I; H^{\frac{1}{2}}(\Omega)) \cap H^{\frac{1}{4}}(I; L^2(\Omega)) \right)'$. Furthermore, the trace theorem implies $u - \gamma \theta \in L^2(I; L^2(\Gamma))$. Consider the following parabolic equation:
\begin{equation}
\begin{aligned}\label{regularity:equation1}
\partial_t \theta - \chi \Delta \theta &= f - (\mathbf{y} \cdot \nabla) \theta\quad \mbox{in}\ \Omega_T, \\
\frac{\chi}{\eta} \frac{\partial \theta}{\partial \mathbf{n}} &= u - \gamma \theta\quad\mbox{on}\ \Gamma_T.
\end{aligned}
\end{equation}
The standard regularity results give $\theta\in L^2(I; {H}^{\frac{3}{2}}(\Omega))\cap H^{\frac{3}{4}}(I; L^2(\Omega))$ (cf. \cite[p. 81, Lemma 2.2]{Malanowski_1982}), which in turn implies $f - (\mathbf{y} \cdot \nabla) \theta \in L^2(I;L^2(\Omega))$. If $u \in H^{\frac{1}{4}}(I; L^2(\Gamma)) \cap L^2(I; H^{\frac{1}{2}}(\Gamma))$, then $u - \gamma \theta \in H^{\frac{1}{4}}(I; L^2(\Gamma)) \cap L^2(I; H^{\frac{1}{2}}(\Gamma))$ by the trace theorem (cf. \cite[p. 80, Lemma 2.1]{Malanowski_1982}). Finally, it follows that $\theta \in V(I)$ (cf. \cite[p. 81, Lemma 2.2]{Malanowski_1982}).

Using the \( H^2 \) regularity of the Stokes operator in convex polygonal domains (see  \cite[p. 400, Theorem 2]{Kellogg_Osborn_1976}) and the standard Galerkin method, we can  obtain the desired regularity and estimate for \( (\boldsymbol{\mu}, \kappa) \). Moreover,  applying Propositions 1.1 and 1.2 together with Remark 1.4  in \cite[p. 14–15]{Temam_1970} and Lemma 2.1 in \cite[p. 22]{Girault_1986}, it follows that there exists a unique 
$\lambda \in L^2(I; H^1(\Omega)) \cap L^2(I; L^2_0(\Omega))$ that satisfies \eqref{Regualrit:equa:reform:adjoint} and \eqref{regular_MuKappa:ineq}.
\end{proof}

\begin{remark}
In the formulations of equations \eqref{Exist:weak:solution:a} and \eqref{adjoint:Exist:weak:solution:a}, divergence-free test functions are used, and thus the pressure term is not included. Under suitable regularity assumptions on the data, Theorem \ref{lem:conL2:Regualrit:YTheta} provides an alternative and equivalent formulation given by \eqref{Regualrit:equa:reform} and \eqref{Regualrit:equa:reform:adjoint}, which includes the pressure term when the test functions are not divergence-free.
\end{remark}

\begin{theorem} 
Assume that \( \bar{u} \) is a local solution of problem \( \mathrm{(P)} \),  \( (\bar{\mathbf{y}}, \bar{\theta}) \) and \( (\bar{\bm{\mu}}, \bar{\kappa}) \) are the corresponding state pair and adjoint pair, respectively. Furthermore, assume that \( f \in L^2(I; L^2(\Omega)) \), \( \mathbf{h} \in L^2(I; \mathbb{L}^2(\Omega)) \), \( \mathbf{y}_0 \in \mathbb{X} \),  \( \theta_0 \in H^1(\Omega) \), $\theta_d\in L^2(I;L^2(\Omega))$ and \( \mathbf{y}_d \in L^2(I; \mathbb{L}^2(\Omega)) \). Then \( (\bar{\mathbf{y}}, \bar{\theta}) \in \mathbb{V} \), \((\bar{\mathbf{\mu}}, \bar{\kappa}) \in \mathbb{V} \) and \( \bar{u} \in L^2(I; W^{1,p}(\Gamma)) \cap H^{\frac{3}{4}}(I; L^2(\Gamma)) \)  for some $1\leq p<\infty $.
\end{theorem}
\begin{proof}Using \eqref{first:optimal:state}, \eqref{first:optimal:adjoint}, \eqref{Proj:equation}, Theorem \ref{lem:conL2:Regualrit:YTheta} and \cite[Lemma 3.3, p. 1735]{Kunisch_Vexler_2007}, and applying a standard bootstrapping argument, the desired result can be obtained.
\end{proof}

\section{Numerical approximation of the optimal control problem}\setcounter{equation}{0}
The purpose of this section is to introduce numerical approximations to solutions of the control problem $(\mathrm{P})$, and to analyze convergence and stability.  In the subsequent numerical analysis, we always assume that \( f \in L^2(I; L^2(\Omega)) \), \( \mathbf{h} \in L^2(I; \mathbb{L}^2(\Omega)) \), \( \mathbf{y}_0 \in \mathbb{X}\cap \mathbb H^2(\Omega)\), \( \theta_0 \in H^1(\Omega) \), $\theta_d\in L^2(I;L^2(\Omega))$ and \( \mathbf{y}_d \in L^2(I; \mathbb{L}^2(\Omega)) \).  

\subsection{Spatial discretization}
Let $\mathcal{T}_{h}$ be a family of shape-regular and quasi-uniform triangulations of $\Omega$ such that $\overline{{\Omega}} = \bigcup_{K \in \mathcal{T}_{h}} \overline{{K}}$. We denote by $h_{K}$  the diameter of the element $K$ and set $h: = \max_{K \in \mathcal{T}_{h}} h_{K}$. The collection of polynomials of degree less than or equal to $\ell$ ($\ell\geq1$) on the element $K$ is represented by $\mathbf{P}_{\ell}(K)$. Define
$$V_h:=\left\{\theta_h \in C(\bar{\Omega}): \theta_h|_{K}\in \mathbf{P}_{\ell}(K),\,\forall K\in \mathcal{T}_{h} \right\}$$
as the approximation space of the temperature. Let \( \mathbb{V}_h \subset \mathbb{H}_0^1(\Omega) \) and \( Q_h \subset L_0^2(\Omega) \) be the finite-dimensional approximation  spaces for the velocity and pressure, respectively. We assume that the spaces \( \mathbb V_h\) and \( Q_h \)  satisfy the following properties:

\textbf{(A1)} Inf-sup condition, i.e., 
\[
\inf_{q_h \in Q_h} \sup_{\mathbf v_h \in \mathbb V_h} \frac{\mathcal{B}(\mathbf v_h, q_h)}{\| \mathbf v_h \|_{\mathbb H^1(\Omega_h)} \| q_h \|_{L^2(\Omega)}} \geq C
\]
for some constant \( C > 0 \), where \( \mathcal{B} : \mathbb H^1(\Omega) \times L^2(\Omega) \to \mathbb{R} \) is defined by
$\mathcal{B}(\mathbf v, q) := -\int_{\Omega} q \, \text{div} \, \mathbf v \, dx$. 

\textbf{(A2)} Projection error estimates, i.e., for any $\mathbf v\in \mathbb H^2(\Omega)\cap \mathbb H^1_0(\Omega)$ and $p\in H^1(\Omega)\cap L^2_0(\Omega)$
\[\|\mathbf v-\mathbf I_h\mathbf v\|+h\|\mathbf v-\mathbf I_h \mathbf v\|_{\mathbb H^1(\Omega)}\leq C h^2\|v\|_{\mathbb H^2(\Omega)}, \quad \|p-I_h p\|\leq C h\|p\|_{H^1(\Omega)}
\]
for some constant $C>0$, where $\mathbf I_h: \mathbb H^2(\Omega)\cap \mathbb H^1_0(\Omega)\to \mathbb V_h$ and $I_h: L^2(\Omega)\to Q_h$ are some projection operators.

The above assumptions \textbf{(A1)}-\textbf{(A2)} are satisfied by several well-known finite elements, such as the Taylor-Hood and Mini elements (see \cite{Temam_1970}).  We define the (vector-valued) discrete  Laplacian \( \bm \Delta_h : \mathbb V_h \to \mathbb V_h \) 
\[
(-\bm \Delta_h \mathbf u_h,  \mathbf v_h)=(\nabla \mathbf u_h, \nabla \mathbf v_h) \quad \forall \mathbf v_h \in  \mathbb V_h,
\]
and the discrete divergence-free subspace of \( \mathbb{V}_h \):
\[
\mathbb X_h := \{ \mathbf y_h \in \mathbb V_h : \mathcal{B}(\mathbf y_h, q_h) = 0 \quad \forall q_h \in Q_h \}.
\]

In addition, we define the $\mathbb L^2$ projection  $\mathbf P_h:\mathbb L^2(\Omega) \to \mathbb X_h $ such that
$$(\mathbf P_h\mathbf  u, \mathbf v_h)=(\mathbf u,\mathbf v_h)\quad \forall \mathbf v_h \in \mathbb X_h.$$
Subsequently, the discrete Stokes operator is defined by \( \mathbf{A}_h : \mathbb{X}_h \to \mathbb{X}_h \) with \( \mathbf{A}_h = \mathbf P_h \bm \Delta_h|_{\mathbb{X}_h} \). The discrete Laplacian \( \bm \Delta_h \) can be bounded by the  discrete Stokes operator \( \mathbf A_h \) (cf. \cite[Corollary 4.4]{HeywoodRannacher}):
\[
\| \bm \Delta_h \mathbf v_h \|_{\mathbb L^2(\Omega)} \leq C \| \mathbf A_h \mathbf v_h \|_{\mathbb L^2(\Omega)} \quad \forall \, \mathbf v_h \in \mathbb X_h.
\]
There hold the following discrete interpolation inequality (cf. \cite{HeywoodRannacher, vexler_wagner_2024}):
\begin{align}
&\|\nabla \mathbf v_h\|_{\mathbb L^4(\Omega)}\leq C\|\nabla \mathbf v_h\|_{\mathbb L^2(\Omega)}^{\frac{1}{2}}\|\mathbf A_h \mathbf v_h\|_{\mathbb L^2(\Omega)}^{\frac{1}{2}}\quad \forall \mathbf v_h\in \mathbb X_h,\label{Interpolation:A:1}
\end{align}

The (scalar-valued)  discrete Laplacian \(  \Delta_h :  V_h \to  V_h \) is defined by
\begin{align}\label{definite_Delta_h}
(-\Delta_h u_h,v_h)=\chi(\nabla u_h, \nabla v_h )+\eta \gamma(u_h,v_h)_\Gamma \quad \forall v_h \in V_h.
\end{align}
It is easy to see that the inverse of $\Delta_h$ exists, and is denoted by $\Delta_h^{-1}$. There holds the following estimate:
\begin{align}\label{DeltaTheta_ineq1}
\|\nabla u_h\|_{L^4(\Omega)}\leq C\|u_h\|_{H^1(\Omega)}^\frac{1}{2}\|\Delta_h u_h\|_{ L^2(\Omega)}^\frac{1}{2}\quad \forall v_h \in V_h.
\end{align}

Let $R_h: H^1(\Omega) \to  V_h $ be the Ritz projection such that 
\begin{align}\label{definite_Rhhh}
\chi(\nabla R_h u,\nabla v_h)+\eta\gamma(R_h u, v_h)_\Gamma=\chi(\nabla u, \nabla v_h)+\eta\gamma(u, v_h)_\Gamma \quad \forall v_h\in V_h.
\end{align}
We denote by $P_h: L^2(\Omega)\to V_h$  the $L^2$ projection defined by 
\[(P_h u,v_h)=(u,v_h)\quad v_h\in V_h.\]
Furthermore, we define the Stokes-Ritz projection $(\mathbf R^S_h,R^{S,p}_h): \mathbb H^1_0(\Omega)\times L^2(\Omega)\to \mathbb V_h\times Q_h$ as follows:
\begin{equation}
\begin{aligned}\label{StokePro}
&\nu(\nabla(\mathbf u - \mathbf R^S_h(\mathbf u,p)), \nabla \mathbf v_h) - (p - R^{S,p}_h(\mathbf u, p), \nabla \cdot \mathbf v_h) = 0 \quad \forall \mathbf v_h \in \mathbb V_h,\\
&(\nabla \cdot (\mathbf u-\mathbf R^S_h(\mathbf u,p)), q_h) = 0 \quad \forall q_h \in Q_h.
\end{aligned}
\end{equation}
Some useful results on these projections can be found in the Appendix \ref{Appendix:A}.

In order to separate the estimate of the temporal and spatial errors, we introduce the spatial  semi-discrete state  equation: Given \( u \in L^2(I; L^2(\Gamma)) \),  find {$(\mathbf y_h(t), \theta_h(t))\in \mathbb X_h\times V_h$} satisfying 
\begin{equation}\label{semi_discrete:state}
\left \{\begin{aligned}
&(  \partial_t \mathbf y_{h},\mathbf v_h)+\nu \mathbf a(\mathbf y_h,\mathbf v_h)+\mathbf b(\mathbf y_h,\mathbf y_h,\mathbf v_h)+\beta(\theta_h \mathbf g,\mathbf v_h)= ( \mathbf h,\mathbf v_h),\\
&(  \partial_t \theta_{h}, \psi_h)+\chi a(\theta_h, \psi_h)+b(\mathbf y_h, \theta_h,\psi_h)+\eta \gamma(\theta_h, \psi_h)_\Gamma=( f,\psi_h)+\eta( u, \psi_h)_\Gamma,\\
&\mathbf y_h(0)=\mathbf y_{0h}:=\mathbf P_h \mathbf y_0,\quad\theta_h(0)=\theta_{0h}:=P_h \theta_0
\end{aligned}\right.
\end{equation}
for any $(v_h, \psi_h)\in \mathbb X_h\times V_h$ and a.a. $t\in (0,T]$, where $\mathbf y_{0h}\in \mathbb X_h$ and $\theta_{0h}\in V_h$ satisfy
\begin{align*}
   \|\mathbf  y_0 - \mathbf y_{0h} \|_{\mathbb L^2(\Omega)} \leq C h,\quad
\|  \theta_0 - \theta_{0h} \|_{ L^2(\Omega)} \leq C h \quad \text{ and }  \| \mathbf P_h \mathbf y_0 \|_{\mathbb H^1(\Omega)} \leq C ,\quad \| P_h \theta_0 \|_{ H^1(\Omega)} \leq C
\end{align*}
for some constant $C>0$. 

We also introduce the spatial semi-discrete adjoint problem:  Find {$(\bm \mu_h(t), \kappa_h(t))\in \mathbb X_h\times V_h$} satisfying
\begin{equation}\label{semi_discrete:adjoint}
\left\{\begin{aligned}
&( -\partial_t \bm  \mu_{h},\mathbf w_h)+\nu \mathbf a({\bm \mu_h},\mathbf w_h)+\mathbf b(\mathbf w_h ,{\mathbf y}_h,{\bm \mu}_h)+\mathbf b(\mathbf y_h ,\mathbf w_h,\bm \mu_h)+b(\mathbf w_h,\theta_h,\kappa_h)=(\mathbf y_h-\mathbf y_d,\mathbf w_h),\\
&(-\partial_t {\kappa}_{h}, \zeta_h)+\chi a({\kappa}_h, \zeta_h)+b(\mathbf y_h, \zeta_h,{\kappa}_h)+\beta(\mathbf g\cdot{\bm \mu}_h,\zeta_h)+\eta \gamma(\kappa_h, \zeta_h)_\Gamma=({\theta}_h-\theta_d,\zeta_h),\\
&\bm \mu_h(T)=\mathbf 0,\quad \kappa_h(T)=0
\end{aligned}\right.
\end{equation}
for any $(\mathbf w_h,\zeta_h) \in \mathbb X_h\times V_h$ and a.a. $t\in [0,T)$, where $\mathbf y_h$ and $\theta_h$ solve the scheme \eqref{semi_discrete:state}.

\begin{remark}
Obviously, \eqref{semi_discrete:state} and \eqref{semi_discrete:adjoint} are the semi-discretizations of the systems \eqref{Exist:weak:solution} and \eqref{adjoint:weak:solution}, respectively.
\end{remark}

\begin{lemma}\label{semi_discreYTheta:stab}
Given $u\in H^{\frac{1}{4}}(I;L^2(\Gamma))\cap L^2(I; H^{\frac{1}{2}}(\Gamma))$, let \((\mathbf{y}_h, \theta_h)\)  be the solutions of \eqref{semi_discrete:state}. Then there hold the following estimates:
\begin{align}
\|\mathbf y_h\|_{C(\bar I;\mathbb L^2(\Omega))}&+\|\theta_h\|_{C(\bar I;L^2(\Omega))}+\sqrt{\nu} \| \mathbf y_h\|_{L^2(I;\mathbb H^1(\Omega))}+\sqrt{\min\{\chi,\eta\gamma\}}\| \theta_h\|_{L^2(I;H^1(\Omega))}\nonumber\\
&\leq  C\big(\|\mathbf y_0\|+\|\theta_0\|+\|u\|_{L^2(I;L^2(\Gamma))}+\|\mathbf h\|_{L^2(I;\mathbb L^2(\Omega))}+\|f\|_{L^2(I;L^2(\Omega))}\big),\label{semi:Y_Theta:stab:ineq1}\\
\|\mathbf y_h\|_{L^\infty(I;\mathbb H^1(\Omega))}&+\sqrt{\nu}\|\mathbf A_h \mathbf y_h\|_{L^2(I;\mathbb L^2(\Omega))}\nonumber\\
&\leq C\big(\|\nabla\mathbf y_0\|+\|\theta_0\|+\|u\|_{L^2(I;L^2(\Gamma))}+\|\mathbf h\|_{L^2(I;\mathbb L^2(\Omega))}+\|f\|_{L^2(I;L^2(\Omega))}\big),\label{semi:Y_Theta:stab:ineq2}\\
\|\partial_t \mathbf y_h\|_{L^2(I;\mathbb L^2(\Omega))}&\leq C\big({\mbox{$ \|\nabla \mathbf y_{0h}\|, \|\theta_0\|, \|u\|_{L^2(I;L^2(\Gamma))}, \|\mathbf h\|_{L^2(I;\mathbb L^2(\Omega))}, \|f\|_{L^2(I;L^2(\Omega))}$}}\big).\label{semi:Y_Theta:stab:ineq3}
\end{align}
\end{lemma}

\begin{proof} The proof of the estimate \eqref{semi:Y_Theta:stab:ineq1} is standard, we omit it here.

$\bullet$ Estimate of \eqref{semi:Y_Theta:stab:ineq2}. Choosing $\mathbf{v}_h = -\mathbf{A}_h \mathbf y_h$ in \eqref{semi_discrete:state} and integrating from $0$ to $t$ give
\begin{align}
&\frac{1}{2}\|\nabla \mathbf y_h(t)\|^2-\frac{1}{2}\|\nabla \mathbf y_{0h}\|^2+\nu \int_0^t\|\mathbf A_h \mathbf y_h\|^2ds=\int_0^t [-(\mathbf h, \mathbf A_h\mathbf y_h)+(\beta\theta_h\mathbf{g}, \mathbf A_h\mathbf y_h)+\mathbf b(\mathbf y_h,\mathbf y_h, \mathbf A_h \mathbf y_h)]ds\nonumber\\
&\leq \frac{\nu}{4}\int_0^t \|\mathbf A_h \mathbf y_h\|^2ds +\frac{2}{\nu}\int_0^t \|\mathbf h\|^2ds+\frac{2|\mathbf{g}|^2\beta^2}{\nu}\int_0^t \|\theta_h\|^2 ds+
\frac{C}{\nu^3}\int_0^t\|\mathbf y_h\|^2\|\nabla \mathbf y_h\|^4 ds+\frac{\nu}{4}\int_0^t\|\mathbf A_h \mathbf y_h\|^2 ds,\nonumber
\end{align}
where we have used Young's inequality, \eqref{Interpolation:A:1} and \eqref{Interpolation:conti:1}. Absorbing the right-hand term by left and letting $J(t):=\frac{2C}{\nu^3}\|\mathbf y_h\|^2_{L^\infty(I;\mathbb L^2(\Omega))}\|\nabla \mathbf y_h(t)\|^2$, we obtain 
\begin{align*}
\|\nabla \mathbf y_h(t)\|^2+\nu \int_0^t\|\mathbf A_h \mathbf y_h\|^2ds&\leq \|\nabla \mathbf y_{0h}\|^2+\int_0^tJ \|\nabla \mathbf y_h\|^2 ds+\frac{4}{\nu}\int_I \|\mathbf h\|^2ds+\frac{4|\mathbf{g}|^2\beta^2}{\nu}\int_I \|\theta_h\|^2 ds.
\end{align*}
Note that $\int_0^T J(t)dt\leq C$, applying the Gronwall inequality and \eqref{semi:Y_Theta:stab:ineq1}, we obtain \eqref{semi:Y_Theta:stab:ineq2}.

$\bullet$ Estimate of \eqref{semi:Y_Theta:stab:ineq3}.  Taking  $\mathbf{v}_h = \partial_t \mathbf y_h$ in the first equation of \eqref{semi_discrete:state}, integrating from $0$ to $T$ and using \eqref{Interpolation:conti:1}, \eqref{Interpolation:A:1}, we have
\begin{align*}
&\int_I \|\partial_t \mathbf y_h \|^2dt=-\nu\int_I \mathbf a(\mathbf y_h, \partial_t \mathbf y_h)dt-\int_I\mathbf b(\mathbf y_h,\mathbf y_h,\partial_t \mathbf y_h) dt -\int_I\beta(\theta_h\mathbf{g},\partial_t \mathbf y_h)dt+\int_I(\mathbf h, \partial_t \mathbf y_h) dt\\
&\leq \nu\int_I \|\mathbf A_h \mathbf y_h\|\|\partial_t \mathbf y_h\| dt+ C\int_I\|\mathbf y_h\|^{\frac{1}{2}}\|\nabla \mathbf y_h\|\,\|\mathbf A_h \mathbf y_h\|^{\frac{1}{2}}\|\partial_t\mathbf y_h\|dt+\beta |\mathbf{g}|\int_I\|\partial_t \mathbf y_h\|\|\theta_h\|dt+\int_I \|\partial_t \mathbf y_h\|\,\|\mathbf h\| dt\\
&\leq \frac{1}{2}\int_I\|\partial_t \mathbf y_h\|^2dt+ C\Big(\int_I \Big(\|\mathbf A_h \mathbf y_h\|^2+\|\mathbf y_h\|^2\|\nabla \mathbf y_h\|^4\Big)dt+\int_I \|\theta_h\|^2 dt+ \int_I \|\mathbf h\|^2 dt\Big).
\end{align*}
Using equations \eqref{semi:Y_Theta:stab:ineq1} and \eqref{semi:Y_Theta:stab:ineq2}, we can obtain \eqref{semi:Y_Theta:stab:ineq3}. This finishes the proof.
\end{proof}

The next lemma specifies the error estimate for the spatial approximation of the state equation \eqref{Exist:weak:solution}.
\begin{lemma}\label{Lem:semi:state:err}
Given $u\in H^{\frac{1}{4}}(I;L^2(\Gamma))\cap L^2(I; H^{\frac{1}{2}}(\Gamma))$, let \((\mathbf y,\theta)\in \mathbb V\) be the solution of \eqref{Exist:weak:solution} and \((\mathbf y_h,\theta_h) \) be the solution of  the semi-discrete state equation \eqref{semi_discrete:state}. Then we have the following estimate:
\begin{equation}\label{Lem:semi:state:err:ineq}
\begin{aligned}
\|\mathbf y-\mathbf y_h\|_{L^\infty(I;\mathbb L^2(\Omega))}&+\|\mathbf y-\mathbf y_h\|_{L^2(I;\mathbb H^1(\Omega))}+\|\theta-\theta_h\|_{L^\infty(I; L^2(\Omega))}\\
&+\|\theta-\theta_h\|_{L^2(I; H^1(\Omega))}\leq  \widetilde{C}_1 h,
\end{aligned}
\end{equation}
where 
$\widetilde{C}_1=C\big({\mbox{$ \|\nabla \mathbf y_{0h}\|, \|\theta_0\|_{H^1(\Omega)}, \|u\|_{H^\frac{1}{4}(I;L^2(\Gamma))},\|u\|_{L^2(I;H^\frac{1}{2}(\Gamma))}, \|\mathbf h\|_{L^2(I;\mathbb L^2(\Omega))}, \|f\|_{L^2(I;L^2(\Omega))}$}}\big).$
\end{lemma}
\begin{proof}
We decompose the errors into the following parts:
\begin{equation}\label{decomp:YTheta_1}
\begin{aligned}
&\mathbf y-\mathbf y_h=\mathbf y-\mathbf P_h \mathbf y+\mathbf P_h \mathbf y-\mathbf y_h=\bm \zeta^{\mathbf y}_h+\bm \eta^{\mathbf y}_h,\quad \theta-\theta_h=\theta-P_h\theta+P_h\theta- \theta_h=\zeta^{\theta}_h+\eta^{\theta}_h.
\end{aligned}
\end{equation}
Subtracting \eqref{semi_discrete:state} from \eqref{Regualrit:equa:reform}, \eqref{Exist:weak:solution:c} and taking \((\mathbf v,\psi)=(\bm \eta^{\mathbf y}_h, \eta^\theta_h ) \), using the decomposition \eqref{decomp:YTheta_1}, and integrating from $0$ to $t$  yield
\begin{align}
&\frac{1}{2}\|\bm \eta^{\mathbf y}_h(t)\|^2+\frac{1}{2}\| \eta^{\theta}_h(t)\|^2+\nu \int_0^t\|\nabla \bm \eta^{\mathbf y}_h\|^2 ds+\chi\int_0^t\|\nabla \eta^\theta_h\|^2 ds+\eta\gamma\int_0^t\|\eta^\theta_h\|^2_{\Gamma}ds=\nonumber\\
&-\int_0^t\big[\nu \mathbf a(\bm \zeta^{\mathbf y}_h,\bm \eta^{\mathbf y}_h)-(p, \nabla \cdot \bm \eta_h^{\mathbf y})\big]ds-\int_0^t\big[ \chi a(\zeta^{\theta}_h, \eta^{\theta}_h)+\eta\gamma(\zeta^{\theta}_h, \eta^{\theta}_h)_\Gamma+\beta(\mathbf g \zeta^\theta_h, \bm \eta^{\mathbf y}_h)+\beta(\mathbf g \eta^\theta_h, \bm \eta^{\mathbf y}_h)\big]ds\nonumber\\
&+\int_0^t\Big[\mathbf b(\mathbf y_h, \mathbf y_h, \bm \eta_h^{\mathbf y})-\mathbf b(\mathbf y, \mathbf y, \bm \eta_h^{\mathbf y})\Big]ds+\int_0^t\Big[b(\mathbf y_h, \theta_h, \eta_h^{\theta})-b(\mathbf y, \theta, \eta_h^{\theta})\Big]ds\nonumber\\
&=J_1+J_2+J_3+J_4.\label{semi:Ytheta:err:13}
\end{align}
Using the  fact that $\int_0^t (I_h p, \nabla \cdot \bm{\eta}_h^{\mathbf{y}}) \, dt = 0$, we deduce the following estimate for $J_1$:
\begin{align*}
J_1&=-\int_0^t[\nu\mathbf a(\bm \zeta_h^{\mathbf y},\bm \eta_h^{\mathbf y})ds-(p-I_hp, \nabla \cdot\bm \eta^{\mathbf y}_h)]ds\leq \frac{\nu}{8}\int_0^t\|\nabla  \bm \eta_h^{\mathbf y}\|^2ds+ 4\nu\int_0^t\|\nabla \bm \zeta_h^{\mathbf y}\|^2ds+\frac{C}{\nu}\int_0^t\|p-I_hp\|^2ds.
\end{align*}
With the help of Young's inequality,  we can estimate $J_2$ as follows:
\begin{align*}
  J_2\leq& \frac{\nu}{8}\int_0^t\|\nabla \bm \eta_h^{\mathbf y}\|^2ds+\frac{\chi}{4}\int_0^t\|\nabla \eta_h^\theta\|^2ds+2\chi\int_0^t\|\nabla\zeta_h^\theta\|^2ds+2\eta\gamma\int_0^t\|\zeta_h^\theta\|^2_{\Gamma} ds+\frac{\eta\gamma}{4}\int_0^t\|\eta_h^\theta\|^2_{\Gamma} ds\\
  &+\frac{\beta^2|\mathbf g|^2}{2\nu}\int_0^t\|\eta_h^\theta\|^2 ds+\frac{\beta^2|\mathbf g|^2}{2\nu}\int_0^t\|\zeta_h^\theta\|^2ds.
\end{align*}
For the estimate of $J_3$, we have the decomposition
\begin{align*}
  J_3&=-\int_0^t\mathbf b(\bm \zeta_h^{\mathbf y},\mathbf y, \bm \eta^{\mathbf y}_h)ds-\int_0^t\mathbf b(\mathbf P_h\mathbf y,\bm \zeta_h^{\mathbf y}, \bm \eta^{\mathbf y}_h)ds-\int_0^t \mathbf b(\bm \eta^{\mathbf y}_h,\mathbf P_h \mathbf y, \bm \eta^{\mathbf y}_h)ds.
\end{align*}
We use \eqref{Interpolation:conti:1},  the fact $\mathbf y\in L^\infty(I;\mathbb H^1(\Omega))$ and the stability of $\mathbf P_h$ in $\mathbb H^1(\Omega)$ (see \cite[p. 15]{Vexler_Leykekhman_2024}) to obtain 
\begin{align*}
J_3\leq\frac{\nu}{8}\int_0^t\|\nabla \bm \eta^{\mathbf y}_h\|^2 ds+ C\Big(\frac{(\nu^2+1)}{\nu^3}\int_0^t\|\bm \eta_h^{\mathbf y}\|^2ds+\frac{1}{\nu}\int_0^t\|\nabla \bm\zeta^{\mathbf y}_h\|^2ds\Big)G
\end{align*}
with $G:=\|\mathbf y\|^2_{L^\infty(I; \mathbb H^1(\Omega))}+\|\mathbf y\|^4_{L^\infty(I; \mathbb H^1(\Omega))}$. Using integration by parts and similar ideas as above yields
\begin{align*}
&J_4=-\int_0^tb(\bm \zeta_h^{\mathbf y}, \theta, \eta_h^\theta)ds-\int_0^tb(\mathbf P_h\mathbf y, \zeta_h^\theta, \eta_h^\theta)ds-\int_0^tb(\bm \eta^{\mathbf y}_h, P_h\theta, \eta_h^\theta)\\
&\leq \frac{\nu}{8}\int_0^t\|\nabla \bm \eta_h^{\mathbf y}\|^2 ds+\frac{\chi}{4}\int_0^t\|\nabla \eta^\theta_h\|^2 ds+\frac{\eta\gamma}{4}\int_0^t\|\eta^\theta_h\|_{\Gamma}^2ds+C\Big(\frac{1}{\min\{\chi,\eta\gamma\}}\Big(\int_0^t\|\nabla \bm \zeta_h^{\mathbf y}\|^2ds+\int_0^t\|\zeta_h^\theta\|^2_{H^1(\Omega)}ds\Big)\\
&+\frac{1+\chi^2}{\chi^2\nu}\int_0^t\|\bm \eta_h^{\mathbf y}\|^2ds+\frac{1}{\min\{\chi, \eta\gamma\}}\int_0^t\|\eta_h^\theta\|^2 ds\Big)H
\end{align*}
with $H:=\|\theta\|^2_{L^\infty(I;H^1(\Omega))}+\|\theta\|^4_{L^\infty(I;H^1(\Omega))}+\|\mathbf y\|^2_{L^\infty(I;\mathbb H^1(\Omega))}$. Combining the estimates of the terms $J_1$, $J_2$, $J_3$,  $J_4$ and using  \eqref{semi:Ytheta:err:13}, we obtain
\begin{align*}
&\|\bm \eta^{\mathbf y}_h(t)\|^2+\| \eta^{\theta}_h(t)\|^2+\nu \int_0^t\|\nabla \bm \eta^{\mathbf y}_h\|^2 ds+\chi\int_0^t\|\nabla \eta^\theta_h\|^2 ds+\eta\gamma\int_0^t\|\eta^\theta_h\|^2_{\Gamma}ds\leq C\Big(J\Big(\int_0^t\|\bm \eta^{\mathbf y}_h\|^2ds\\
&+\int_0^t\|\eta_h^\theta\|^2 ds\Big)+K\Big(\int_0^t\|\nabla \bm \zeta_h^{\mathbf y}\|^2ds+\int_0^t\|\zeta_h^\theta\|^2_{H^1(\Omega)}ds+\int_0^t\|p-I_hp\|^2ds\Big)\Big)
\end{align*}
with $J:=\max\big\{G\frac{(\nu^2+1)}{\nu^3}
+H\frac{1+\chi^2}{\chi^2\nu}, \frac{\beta^2|\mathbf g|^2}{2\nu}+H\frac{1}{\min\{\chi, \eta\gamma\}}\big\}$,  $K:=\max\big\{2\nu+\frac{G}{\nu}+\frac{H}{\min\{\chi, \eta\gamma\}},\max\{\chi, \eta\chi\}+\frac{H}{\min\{\chi, \eta\gamma\}}+\frac{\beta^2|\mathbf g|^2}{2\nu},\frac{1}{\nu}\big\}$. Then, using the Gronwall inequality, assumption ($\mathbf A2$), Lemma \ref{Lem:L2:err}, and the triangle inequality, we can derive \eqref{Lem:semi:state:err:ineq}. This finishes the proof.
\end{proof}


\begin{lemma}\label{semi_discreYTheta:High_regular:stab}
Given $u\in H^{\frac{1}{4}}(I;L^2(\Gamma))\cap L^2(I; H^{\frac{1}{2}}(\Gamma))$, let \((\mathbf{y}_h, \theta_h)\) be the solutions of \eqref{semi_discrete:state}. Then the following estimates hold:
\begin{equation}
\begin{aligned}
\|\partial_t\theta_h\|_{L^2(I;L^2(\Omega))}+\|\theta_h\|_{L^\infty(I;H^1(\Omega))}+\|\theta_h\|_{H^\frac{1}{2}(I;H^1(\Omega))}
+\|\mathbf y_h\|_{H^\frac{1}{2}(I;\mathbb H^1(\Omega))}
+\|\theta_h\|_{L^4(I;W^{1,4}(\Omega))}\leq \widetilde{C}_1\label{semi:Y_Theta:stab:bu:ineq8}.
\end{aligned}
\end{equation}
\end{lemma}
\begin{proof}
We prove the expected estimates one by one. 

$\bullet$ Estimate of $\|\partial_t\theta_h\|_{L^2(I;L^2(\Omega))}$. For any \(v \in L^2(\Omega)\) and  \(\zeta\in C^{\infty}_0(I)\), setting \( v_h = P_h v \), taking $\psi_h = \zeta v_h$ and $\psi = \zeta v_h$ in the third equation of both \eqref{semi_discrete:state} and \eqref{Exist:weak:solution}, integrating from $0$ to $T$ yields
\begin{align*}
&\int_I(\partial_t \theta_{h}, \zeta v_h)dt=-\chi \int_I(\nabla\theta_h,\zeta \nabla v_h)dt-\eta \gamma\int_I(\theta_h, \zeta v_h)_\Gamma dt-\int_Ib(\mathbf y_h, \theta_h,\zeta v_h)dt+\int_I( f, \zeta v_h)dt+\int_I\eta( u, \zeta v_h)_\Gamma dt\nonumber\\
&=\int_I(\partial_t\theta, \zeta v_h)dt+\int_I \big[b(\mathbf y,\theta, \zeta v_h)-b(\mathbf y_h, \theta_h, \zeta v_h)\big]dt
+\int_I\big[\chi(\nabla(\theta-\theta_h), \zeta\nabla v_h) +\eta\gamma (\theta-\theta_h,\zeta v_h)_\Gamma\big]dt\nonumber\\
&\leq C\big(\|\partial_t \theta \|_{L^2(I;L^2(\Omega))}+\|\theta\|_{L^\infty(I;H^1(\Omega))}+\|\mathbf y_h\|_{L^\infty(0,T;\mathbb H^1(\Omega))}+ \max\{\chi, \eta\gamma\}\big)\|\zeta v\|_{L^2(I;L^2(\Omega))},
\end{align*}
where we have used the Cauchy-Schwarz inequality, \eqref{Lem:semi:state:err:ineq}  and the inverse estimate. Then a density argument gives the estimate.

$\bullet$ Estimate of $\|\theta_h\|_{L^\infty(I;H^1(\Omega))}$. By the inverse estimate, the $H^1$-stability of $L^2$-projection (see \cite{Crouzeix1987}), and Lemma \ref{Lem:semi:state:err}, we can derive
\begin{align*}
\|\theta_h\|_{L^\infty(I;H^1(\Omega))}&\leq \|\theta_h-P_h\theta\|_{L^\infty(I;H^1(\Omega))}+\|P_h \theta\|_{L^\infty(I;H^1(\Omega))}\\
&\leq Ch^{-1}\|\theta_h-P_h\theta\|_{L^\infty(I;L^2(\Omega))}+C\|\theta\|_{L^\infty(I;H^1(\Omega))}\\
&\leq  Ch^{-1}\|\theta_h-\theta\|_{L^\infty(I;L^2(\Omega))}+C\|\theta\|_{L^\infty(I;H^1(\Omega))}\\
&\leq C\big( 1+\|\theta\|_{L^\infty(I;H^1(\Omega))}\big).
\end{align*}

$\bullet$ Estimate of $\|\theta_h\|_{H^\frac{1}{2}(I;H^1(\Omega))}$.  By the inverse estimate and Lemma \ref{Lem:semi:state:err}, we have 
\begin{align*}
  \|\theta_h\|_{H^\frac{1}{2}(I;H^1(\Omega))}&\leq \|\theta_h-P_h\theta\|_{H^\frac{1}{2}(I;H^1(\Omega))}+\|P_h\theta\|_{H^\frac{1}{2}(I;H^1(\Omega))}\nonumber\\
  &\leq C\|\theta_h-P_h\theta\|^\frac{1}{2}_{L^2(I;H^1(\Omega))}\|\theta_h-P_h\theta\|^\frac{1}{2}_{H^1(I,H^1(\Omega))}+\|P_h\theta\|_{H^\frac{1}{2}(I;H^1(\Omega))}\nonumber\\
  &\leq C \big( h^{-\frac{1}{2}}\|\theta_h-\theta\|^\frac{1}{2}_{L^2(I;H^1(\Omega))}\|\theta_h-\theta\|^\frac{1}{2}_{H^1(I,L^2(\Omega))}+\|\theta\|_{H^\frac{1}{2}(I;H^1(\Omega))}\big)\nonumber\\
  &\leq C\big(\big(\|\theta_h\|_{H^1(I;L^2(\Omega))}+\|\theta\|_{H^1(I;L^2(\Omega))}\big)^\frac{1}{2}+\|\theta\|_{H^\frac{1}{2}(I;H^1(\Omega))}\big).
\end{align*}
Similarly to the above procedure, we can obtain the estimate for $\|\mathbf y_h\|_{H^\frac{1}{2}(I;\mathbb H^1(\Omega))}$.

$\bullet$ Estimate of $\|\theta_h\|_{L^4(I;W^{1,4}(\Omega))}$. Using the inverse estimate, the interpolation inequality, the $W^{1,p}$-stability of the $L^2$ projection (see \cite{Crouzeix1987}), and Lemma \ref{Lem:semi:state:err}, we deduce 
\begin{align*}
  \|\theta_h\|_{L^4(I;W^{1,4}(\Omega))}&\leq \|\theta_h-P_h\theta\|_{L^4(I;W^{1,4}(\Omega))}+\|P_h\theta\|_{L^4(I;W^{1,4}(\Omega))}\nonumber\\
  &\leq C\|\theta_h-P_h\theta\|^\frac{1}{2}_{L^2(I;W^{1,4}(\Omega))}\|\theta_h-P_h\theta\|^\frac{1}{2}_{L^\infty(I,W^{1,4}(\Omega))}+\|P_h\theta\|_{L^4(I;W^{1,4}(\Omega))}\nonumber\\
  &\leq C \big( h^{-\frac{1}{2}}\|\theta_h-\theta\|^\frac{1}{2}_{L^2(I;H^{1}(\Omega))}\|\theta_h-\theta\|^\frac{1}{2}_{L^\infty(I,H^{1}(\Omega))}+\|\theta\|_{L^4(I;W^{1,4}(\Omega))}\big)\nonumber\\
  &\leq C\big(\big(\|\theta_h\|_{L^\infty(I;H^1(\Omega))}+\|\theta\|_{L^\infty(I;H^1(\Omega))}\big)^\frac{1}{2}+\|\theta\|_{L^4(I;W^{1,4}(\Omega))}\big).
\end{align*}
This finishes the proof.
\end{proof}

The next lemma specifies the error estimate for the spatial  approximation of the adjoint equation \eqref{adjoint:weak:solution}. The proof is omitted because it is analogous to that of Lemma \ref{Lem:semi:state:err}.
\begin{lemma}\label{Lem:semi:adjoint:err}  Let \((\bm \mu,\kappa)\in \mathbb V\) be the solution of \eqref{adjoint:weak:solution} and \((\bm \mu_h,\kappa_h) \) be the solution of  the semi-discrete adjoint equation \eqref{semi_discrete:adjoint}. Then, under the assumptions of Lemma \ref{Lem:semi:state:err} there hold the following estimates:
\begin{equation*}
\begin{aligned}
\|\bm \mu-\bm \mu_h\|_{L^\infty(I;\mathbb L^2(\Omega))}+\|\bm \mu-\bm \mu_h\|_{L^2(I;\mathbb H^1(\Omega))}+\|\kappa-\kappa_h\|_{L^\infty(I; L^2(\Omega))}+\|\kappa-\kappa_h\|_{L^2(I; H^1(\Omega))}\leq  \widetilde{C}_2 h,
\end{aligned}
\end{equation*}
where 
\[ \widetilde{C}_2=C\big({\scriptsize\mbox{$ \|\nabla \mathbf y_{0h}\|, \|\theta_0\|_{H^1(\Omega)}, \|u\|_{H^\frac{1}{4}(I;L^2(\Gamma))\cap L^2(I;H^\frac{1}{2}(\Gamma))}, \|\mathbf h\|_{L^2(I;\mathbb L^2(\Omega))}, \|f\|_{L^2(I;L^2(\Omega))}, \|\mathbf y_d\|_{L^2(I;\mathbb L^2(\Omega))}, \|\theta_d\|_{L^2(I; L^2(\Omega))}$}}\big).\]
\end{lemma}

\begin{lemma}\label{semi_discreMukappa:stab}
Given $u\in H^{\frac{1}{4}}(I;L^2(\Gamma))\cap L^2(I; H^{\frac{1}{2}}(\Gamma))$ and let  \((\bm{\mu}_h, \kappa_h)\) be the solutions of  problem \eqref{semi_discrete:adjoint}. Then the following estimates hold:
\begin{align}
&\|\bm \mu_h\|_{C(\bar I;\mathbb L^2(\Omega))}+\|\kappa_h\|_{C(\bar I; L^2(\Omega))}+\sqrt{\nu}\|\bm \mu_h\|_{L^2(I;\mathbb H^1(\Omega))}+\sqrt{\min\{\chi,\eta\gamma\}}\| \kappa_h\|_{L^2(I; H^1(\Omega))}\nonumber\\
&\leq C\big({\mbox{$ \| \mathbf y_{0h}\|, \|\theta_0\|, \|u\|_{L^2(I;L^2(\Gamma))}, \|\mathbf h\|_{L^2(I;\mathbb L^2(\Omega))}, \|f\|_{L^2(I;L^2(\Omega))},\|\mathbf y_d\|_{L^2(I;\mathbb L^2(\Omega))}, \|\theta_d\|_{L^2(I; L^2(\Omega))}$}}\big),\label{semi:Y_Theta:stab:ineq6}\\
&\sqrt{\min\{\chi,\eta\gamma\}}\|\kappa_h\|_{L^\infty(I; H^1(\Omega))}+\|\Delta_h \kappa_h\|_{L^2(I;L^2(\Omega))}+\|\bm \mu_h\|_{L^\infty(I;\mathbb H^1(\Omega))}+\sqrt{\nu}\|\mathbf A_h \bm \mu_h\|_{L^2(I;\mathbb L^2(\Omega))}\leq \widetilde{C}_2,\nonumber\\
&\|\partial_t\kappa_h\|_{L^2(I; L^2(\Omega))}+\|\partial_t\bm\mu_h\|_{L^2(I;\mathbb L^2(\Omega))}+\|\bm \mu_h\|_{H^\frac{1}{2}(I;\mathbb H^1(\Omega))}+\|\kappa_h\|_{H^\frac{1}{2}(I; H^1(\Omega))}\leq \widetilde{C}_2.\nonumber
\end{align}
\end{lemma}
\begin{proof}
The proof is omitted because it is analogous to that of Lemmas \ref{semi_discreYTheta:stab} and \ref{semi_discreYTheta:High_regular:stab}.
\end{proof}

\begin{remark}
Given \( u \in L^2(I; L^2(\Gamma)) \), using the inequalities \eqref{semi:Y_Theta:stab:ineq1} and \eqref{semi:Y_Theta:stab:ineq6} along with the Picard-Lindel\"of theorem, it is well-known that the semi-discrete equations \eqref{semi_discrete:state} and \eqref{semi_discrete:adjoint} have at least one solution. Uniqueness can be obtained through standard methods.
\end{remark}

Using a suitable duality argument, we can deduce improved error estimates for the state and adjoint variables  under the norm $L^2(I;\mathbb L^2(\Omega))$ (or $L^2(I;L^2(\Omega))$).
\begin{lemma}\label{semi_err:YTheta:L2} 
Let $(\mathbf y, \theta)\in\mathbb{V}$ be the solution to \eqref{Exist:weak:solution:a}-\eqref{Exist:weak:solution:c} with $u\in  H^{\frac{1}{4}}(I;L^2(\Gamma))\cap L^2(I; H^{\frac{1}{2}}(\Gamma)) $, and \((\mathbf y_h,\theta_h)\in L^2({I};\mathbb X_h)\times L^2({I};V_h)\) be the solution to \eqref{semi_discrete:state}. Then there exists a constant $C>0$ such that
\begin{align}
&\|\mathbf y-\mathbf y_h\|_{L^2(I;\mathbb L^2(\Omega))}+\|\theta-\theta_h\|_{L^2(I; L^2(\Omega))}\leq C h^2.
\end{align}
\end{lemma}
\begin{proof} Using \eqref{decomp:YTheta_1} and  Lemma \ref{Lem:L2:err}, it follows that
\begin{align*}
\|\mathbf y-\mathbf y_h\|_{L^2(I;\mathbb L^2(\Omega))}+\|\theta-\theta_h\|_{L^2(I;L^2(\Omega))}\leq C h^2+\|\bm \eta^{\mathbf y}_h\|_{L^2(I;\mathbb L^2(\Omega))}+\|\eta^\theta_h\|_{L^2(I; L^2(\Omega))}.
\end{align*}
Now, we can estimate  $\|\bm \eta^{\mathbf y}_h\|_{L^2(I;\mathbb L^2(\Omega))} + \|\eta^\theta_h\|_{L^2(I; L^2(\Omega))}$ using a duality argument. Define the following dual problem: Find $(\bm  \varphi_{h},z_h)\in L^2(I;\mathbb X_h)\times L^2(I;V_h)$, such that for any $(\mathbf w_h,\zeta_h)\in \mathbb X_h\times V_h$ and a.a. $t\in (0,T)$,
\begin{equation}\label{dual_adjoint:semi_discre}
\left\{\begin{aligned}
&( -\partial_t \bm  \varphi_{h},\mathbf w_h)+\nu \mathbf a({\bm \varphi_h},\mathbf w_h)+\mathbf b(\mathbf w_h ,{\mathbf y}_h,{\bm \varphi}_h)+\mathbf b(\mathbf y,\mathbf w_h,\bm \varphi_h)+b(\mathbf w_h,\theta_h,z_h)=(\bm \eta_h^{\mathbf y},\mathbf w_h),\\
&(-\partial_t {z}_{h}, \zeta_h)+\chi a({z}_h, \zeta_h)+b(\mathbf y, \zeta_h,{z}_h)+\beta(\mathbf g\cdot{\bm \varphi}_h,\zeta_h)+\eta \gamma(z_h, \zeta_h)_\Gamma=(\eta_h^{\theta},\zeta_h),\\
&\bm \varphi_h(T)=\mathbf 0,\quad z_h(T)=0.
\end{aligned}\right.
\end{equation}

The above equation is similar to the semi-discrete adjoint equation \eqref{semi_discrete:adjoint}. Therefore, similar to the proof of Lemma \ref{semi_discreYTheta:stab}, we know that \eqref{dual_adjoint:semi_discre} has a unique solution $(\bm \varphi_h, z_h)\in L^\infty({I};\mathbb X_h)\times L^\infty({I};V_h)$, and the following stability estimates are valid:
\begin{align}\label{dual_adjoint_est:stab}
&\|\bm \varphi_h\|_{L^\infty(I;\mathbb H^1(\Omega))}+\|\mathbf A_h \bm \varphi_h\|_{L^2(I;\mathbb L^2(\Omega))}+\|z_h\|_{L^\infty(I; H^1(\Omega))}+\| \Delta_h z_h\|_{L^2(I; L^2(\Omega))}\nonumber\\
&\leq C\big(\|\bm \eta_h^{\mathbf y}\|_{L^2(I;\mathbb L^2(\Omega))}
+\|\eta_h^{\theta}\|_{L^2(I;L^2(\Omega))}\big).
\end{align}
Setting $(\mathbf w_h,\zeta_h)=(\bm \eta^{\mathbf y}_h,\eta_h^{\theta})$ in \eqref{dual_adjoint:semi_discre}, and integrating from $0$ to $T$, we obtain
\begin{align}\label{L2err:YTheta:1}
&\|\bm \eta^{\mathbf y}_h\|^2_{L^2(I;\mathbb L^2(\Omega))}+\|\eta^{\theta}_h\|^2_{L^2(I;L^2(\Omega))}=\int_I(-\partial_t\bm\varphi_h,\bm \eta_h^{\mathbf y})dt+\int_I(-\partial_t z_h,\eta_h^\theta)dt+\nu \int_I \mathbf a(\bm\varphi_h,\bm \eta^{\mathbf y}_h)dt\nonumber\\
&+\int_I\big[\mathbf b(\bm \eta_h^{\mathbf y},\mathbf y_h, \bm \varphi_h)+\mathbf b(\mathbf y, \bm\eta_h^{\mathbf y},\bm \varphi_h)\big]dt+\int_I\big[b(\bm \eta_h^{\mathbf y},\theta_h,z_h)+b(\mathbf y,\eta_h^\theta,z_h)\big]dt+\chi \int_I a(z_h, \eta_h^\theta)dt\nonumber\\
&+\beta\int_I(\mathbf{g}\cdot \bm \varphi_h,\eta_h^\theta)dt+\eta\gamma\int_I(z_h, \eta_h^{\theta})_\Gamma dt.
\end{align}
Subtracting \eqref{semi_discrete:state} from \eqref{Regualrit:equa:reform}, \eqref{Exist:weak:solution:c} with \((\mathbf v,\psi)=(\mathbf v_h,\psi_h)\equiv(\bm \varphi_h, z_h ) \), using the decomposition \eqref{decomp:YTheta_1}, and integrating from $0$ to $T$ yield
\begin{align}\label{L2err:YTheta:2}
&-\int_I(\bm \eta_h^{\mathbf y},\partial_t \bm \varphi_h)dt+\nu \int_I\mathbf a(\bm \eta_h^{\mathbf y},\bm \varphi_h)dt+\beta\int_I(\eta_h^\theta\mathbf{g},\bm \varphi_h)dt\nonumber=-\nu\int_I \mathbf a(\bm \zeta_h^{\mathbf y},\bm \varphi_h)dt+\int_I(p,\nabla\cdot \bm \varphi_h)dt-\beta\int_I(\zeta_h^\theta\mathbf{g}, \bm\varphi_h)dt\nonumber\\
&-\int_I\mathbf b(\bm \eta_h^{\mathbf y}, \mathbf y_h,\bm \varphi_h)dt-\int_I\mathbf b(\bm \zeta_h^{\mathbf y},\mathbf y_h, \bm \varphi_h)dt-\int_I\mathbf b(\mathbf y, \bm \eta_h^{\mathbf y}, \bm \varphi_h)dt-\int_I\mathbf b(\mathbf y, \bm \zeta_h^{\mathbf y}, \bm \varphi_h)dt
\end{align}
and 
\begin{align}\label{L2err:YTheta:3}
-\int_I(\eta_h^\theta, \partial_t z_h)dt &+\chi\int_Ia(\eta_h^\theta,z_h)dt+\eta\gamma\int_I(\eta^\theta_h,z_h)_\Gamma dt=-\chi\int_I a(\zeta^\theta_h,z_h)dt -\eta\gamma\int_I(\zeta_h^\theta,z_h)_\Gamma dt\nonumber\\
&-\int_I b(\bm \eta_h^{\mathbf y}, \theta_h, z_h) dt-\int_I b(\bm \zeta_h^{\mathbf y},\theta_h, z_h)dt -\int_Ib(\mathbf y, \eta^\theta_h, z_h)dt-\int_I b(\mathbf y, \zeta_h^\theta,z_h)dt.
\end{align}
Combining now \eqref{L2err:YTheta:1}, \eqref{L2err:YTheta:2} and \eqref{L2err:YTheta:3}, we obtain
\begin{align}\label{L2err:YTheta:4}
\|\bm \eta^{\mathbf y}_h\|^2_{L^2(I;\mathbb L^2(\Omega))}&+\|\eta^{\theta}_h\|^2_{L^2(I;L^2(\Omega))}= \int_I  \big[-\nu \mathbf a(\bm \zeta_h^{\mathbf y},\bm \varphi_h)+(p,\nabla\cdot \bm \varphi_h)\big]dt\nonumber\\
&-\int_I  \big[\chi a(\zeta^\theta_h,z_h)+\eta\gamma(\zeta_h^\theta,z_h)_\Gamma\big] dt-\int_I\big[\mathbf b(\bm \zeta_h^{\mathbf y},\mathbf y_h,\bm \varphi_h)+\mathbf b(\mathbf y,\bm \zeta_h^{\mathbf y},\bm \varphi_h)]dt\nonumber\\
&-\int_I\big[ b(\bm \zeta_h^{\mathbf y},\theta_h, z_h)+b(\mathbf y,\zeta_h^\theta,z_h)\big]dt-\beta\int_I(\zeta_h^\theta\mathbf{g}, \bm\varphi_h)dt=J_1+J_2+J_3+J_4+J_5.
\end{align}
We can argue using the projection $\mathbf R_h^S$ defined in \eqref{StokePro} and the definition of the discrete Stokes operator $\mathbf A_h$ to obtain
\begin{align}\label{L2err:YTheta:Bu:4}
J_1&=\int_I\big[-\nu\mathbf a(\mathbf R_h^S(\mathbf y,p),\bm \varphi_h)+(R_h^{S,p}({\mathbf y},p),\nabla\cdot \bm \varphi_h)+\nu\mathbf  a(\mathbf P_h\mathbf y, \bm \varphi_h)\big]dt=-\nu\int_I\mathbf a\big(\mathbf R_h^S(\mathbf y,p)-\mathbf P_h\mathbf y,\bm\varphi_h\big)dt\nonumber\\
&\leq C\nu\big(\| \bm \zeta_h^{\mathbf y}\|_{L^2(I;\mathbb L^2(\Omega))}+\|\mathbf y- \mathbf R_h^S(\mathbf y, p)\|_{L^2(I;L^2(\Omega))}\big)\,\|\mathbf A_h\bm \varphi_h\|_{L^2(I;\mathbb L^2(\Omega))},
\end{align}
where we have used the fact that $\bm \varphi_h$ and $\bm \zeta_h^{\mathbf y} $ are discretely divergence-free so that $\mathbf R_h^S(\bm \zeta_h^{\mathbf y},p)\in \mathbb X_h$.  To treat the term $J_2$, we use the projection $R_h$ defined in \eqref{definite_Rhhh} to obtain the following
\begin{align*}
J_2=-\int_I  \big[\chi a(R_h\zeta^\theta_h,z_h)&+\eta\gamma(R_h\zeta_h^\theta,z_h)_\Gamma\big] dt= \int_I(R_h\zeta^\theta_h,\Delta_h z_h)dt\nonumber\\
&\leq C{\big(\|\zeta^\theta_h\|_{L^2(I;L^2(\Omega))}+\|\theta- R_h \theta\|_{L^2(I; L^2(\Omega))}\big)}\|\Delta_h z_h\|_{L^2(I;L^2(\Omega))}.
\end{align*}
Applying Lemma \ref{lem:b:xingzhi} and \eqref{Interpolation:A:1}, we have
\begin{align}\label{dual:est:10}
J_3&\leq C\|\bm \zeta_h^{\mathbf y}\|_{L^2(I;\mathbb L^2(\Omega))}\Big(\|\mathbf y_h\|^{\frac{1}{2}}_{L^2(I;\mathbb H^1(\Omega))}\|\mathbf A_h\mathbf y_h\|^\frac{1}{2}_{L^2(I;\mathbb L^2(\Omega))}\|\bm \varphi_h\|_{L^\infty(I;\mathbb H^1(\Omega))}\nonumber\\
&+\big(\|\mathbf y_h\|_{L^\infty(I;\mathbb H^1(\Omega))}+\|\mathbf y\|_{L^\infty(I; \mathbb H^1(\Omega))}\big)\|\bm \varphi_h\|^\frac{1}{2}_{L^2(I;\mathbb H^1(\Omega))}\|\mathbf A_h\bm \varphi_h\|^\frac{1}{2}_{L^2(I;\mathbb L^2(\Omega))}\Big).
\end{align}
Using the same arguments as in \eqref{dual:est:10}, we conclude from \eqref{DeltaTheta_ineq1} that
\begin{align*}
J_4\leq C\|\bm \zeta_h^{\mathbf y}\|_{L^2(I;\mathbb L^2(\Omega))}&\big(\|\nabla\theta\|_{L^2(I;L^4(\Omega))}\| z_h\|_{L^\infty(I; H^1(\Omega))}+\| z_h\|^\frac{1}{2}_{L^2(I; H^1(\Omega))}\|\Delta_h z_h\|^\frac{1}{2}_{L^2(I;L^2(\Omega))}\|\theta_h\|_{L^\infty(I;H^1(\Omega))}\big)\nonumber\\
&+C\| \zeta_h^\theta\|_{L^2(I;L^2(\Omega))}\big(\|\mathbf y\|_{L^\infty(I;\mathbb H^1(\Omega))}\| z_h\|^\frac{1}{2}_{L^2(I; H^1(\Omega))}\|\Delta_h z_h\|^\frac{1}{2}_{L^2(I;L^2(\Omega))}\\
&+\|\mathbf y_h\|^{\frac{1}{2}}_{L^2(I;\mathbb H^1(\Omega))}\|\mathbf A_h\mathbf y_h\|^\frac{1}{2}_{L^2(I;\mathbb L^2(\Omega))}\|z_h\|_{L^\infty(I;H^1(\Omega))}\big).
\end{align*}
Applying the Cauchy-Schwarz inequality, we have
\begin{align*}
J_5\leq \beta|\mathbf{g}|\,\|\zeta_h^\theta\|_{L^2(I; L^2(\Omega))}\|\bm\varphi_h\|_{L^2(I;\mathbb L^2(\Omega))}.
\end{align*}
Combining the estimates for $J_1$, $J_2$, $J_3$, $J_4$ and $J_5$, using  \eqref{L2err:YTheta:4}, \eqref{dual_adjoint_est:stab} and Theorem \ref{semi_discreYTheta:stab} we obtain
\begin{align*}
\|\bm \eta^{\mathbf y}_h\|_{L^2(I;\mathbb L^2(\Omega))}+\|\eta^{\theta}_h\|_{L^2(I;L^2(\Omega))}\leq C&\big(\|\bm \zeta_h^{\mathbf y}\|_{L^2(I;\mathbb L^2(\Omega))}+\| \zeta_h^{\theta}\|_{L^2(I; L^2(\Omega))}+\|\theta- R_h \theta\|_{L^2(I; L^2(\Omega))}\\
&+ \|\mathbf y- \mathbf R_h^S(\mathbf y, p)\|_{L^2(I;L^2(\Omega))}\big).
\end{align*}
Finally,  applying Lemmas \ref{Lem:stokes:err} and \ref{Lem:L2:err},  the proof is complete.
\end{proof}

\begin{lemma}\label{semi_err:MuKappa:L2} Under the assumptions of Lemma \ref{Lem:semi:adjoint:err} , the following estimate holds:
\begin{align}\label{semi_err:MuKappa:L2:ineq}
&\|\bm \mu-\bm \mu_h\|_{L^2(I;\mathbb L^2(\Omega))}+\|\kappa-\kappa_h\|_{L^2(I; L^2(\Omega))}\leq C h^2.
\end{align}
\end{lemma}
\begin{proof}
We use the  decomposition
\begin{equation*}
\begin{aligned}
\bm  \mu-\bm \mu_h=\bm \mu-\mathbf P_h \bm \mu+\mathbf P_h \bm \mu-\bm \mu_h=\bm \zeta^{\bm \mu}_h+\bm \eta^{\bm \mu}_h,\quad \kappa-\kappa_h=\kappa- P_h\kappa+P_h\kappa- \kappa_h=\zeta^{\kappa}_h+\eta^{\kappa}_h.
\end{aligned}
\end{equation*}
From the error estimate of the $L^2$ projection given by Lemma \ref{Lem:L2:err}, we can derive
\begin{align*}
\|\bm \mu-\bm \mu_h\|_{L^2(I;\mathbb L^2(\Omega))}+\|\kappa-\kappa_h\|_{L^2(I;L^2(\Omega))}\leq C h^2+\|\bm \eta^{\bm \mu }_h|_{L^2(I;\mathbb L^2(\Omega))}+\|\eta^\kappa_h\|_{L^2(I; L^2(\Omega))}.
\end{align*}
We now use a duality argument to provide an estimate for $\|\bm \eta^{\bm \mu}_h\|_{L^2(I;\mathbb L^2(\Omega))} + \|\eta^\kappa_h\|_{L^2(I; L^2(\Omega))}$.  Define the following dual problem: Find $(\bm  z_{h},\xi_h)\in L^2(I;\mathbb X_h)\times L^2(I;V_h)$ such that for any $(\mathbf v_h,\psi_h)\in \mathbb X_h\times V_h$
\begin{equation}\label{dual_LinState:semi_discre}
\left\{\begin{aligned}
&(\partial_t  \mathbf z_h,\mathbf v_h)+\nu \mathbf a(\mathbf z_h ,\mathbf v_h)+\mathbf b(\mathbf z_h,\mathbf y_h,\mathbf v_h)+\mathbf b(\mathbf y_h,\mathbf z_h,\mathbf v_h)+\beta(\xi_h \mathbf g,\mathbf v_h)=(\bm \eta^{\bm \mu}_h,\mathbf v_h),\\
&( \partial_t  \xi_h, \psi_h)+\chi a(\xi_h, \psi_h)+b(\mathbf z_h, \theta_h,\psi_h)+b(\mathbf y_h, \xi_h,\psi_h)+\eta \gamma( \xi_h, \psi_h)_\Gamma=(\eta^\kappa_h, \psi_h),\\
&\mathbf  z_h(0)=0,\quad \xi_h(0)=0.
\end{aligned}\right.
\end{equation}
Note that the above equation is similar to the linearized state equation \eqref{Diff:first:deri}.  Similarly, as in the proof of Lemma \ref{semi_discreYTheta:stab}, it is known that \eqref{dual_LinState:semi_discre} has a unique solution \((\mathbf z_h, \xi_h) \in L^2({I};\mathbb X_h)\times L^2({I};V_h)\) with the following stability estimate:
\begin{equation*}
\begin{aligned}
&\|\mathbf  z_h\|_{L^\infty({I};\mathbb H^1(\Omega))}+\|\mathbf A_h\mathbf  z_h\|_{L^2(I;\mathbb L^2(\Omega))}+\|\xi_h\|_{L^\infty({I};H^1(\Omega))}+\| \Delta_h \xi_h|_{L^2(I; L^2(\Omega))}\leq C(\|\bm \eta_h^{\bm  \mu}\|_{L^2(I;\mathbb L^2(\Omega))}+\|\eta_h^{\kappa}\|_{L^2(I;L^2(\Omega))}).
\end{aligned}
\end{equation*}
Following an analysis similar to that of Lemma \ref{semi_err:YTheta:L2}, and using its result, we thus obtain \eqref{semi_err:MuKappa:L2:ineq}.
\end{proof}

\subsection{The fully-discrete  optimal control problem}
Furthermore, we consider the temporal discretization. Let \( 0 = t_0 < t_1 < \cdots < t_{N_\tau} = T \) represent a partition of the interval \( [0, T] \). We define \( \tau_n := t_n - t_{n-1} \),  $I^r_n:=(t_{n-1}, t_{n}]$ and $I^l_n:=[t_{n-1},t_n)$ for $n=1, \cdots, N_{\tau}$. We make the following assumption on the time step size
\[(\mathbf B)\qquad\exists\, \epsilon_0, \epsilon_1> 0  \text{ such that } \tau = \max_{1 \leq n \leq N_\tau}\tau_n < \epsilon_0  \tau_n {\text{ and } |\tau_n-\tau_{n-1}|\leq \epsilon_1\tau^2}   \quad \forall  1 \leq n \leq N_\tau \text{ and }\forall \tau > 0.\]
Moreover, we set $\sigma=(\tau, h)$.  Finally, we define the spaces
\begin{align*}
\mathbb X_\sigma&:=\big\{\mathbf v_\sigma \in L^2(I;\mathbb X_h): \mathbf v_\sigma=\sum\limits_{n=1}^{N_{\tau}}\mathbf v_{n,h}\mathds{1}_{I^r_n},~ \mathbf v_{n,h}\in \mathbb X_h, ~n=1,\cdots, N_\tau\big\},\\
V_\sigma&:=\big\{v_\sigma \in L^2(I; V_h):  v_\sigma=\sum\limits_{n=1}^{N_{\tau}} v_{n,h}\mathds{1}_{I^r_n},~ v_{n,h}\in V_h, ~n=1,\cdots, N_\tau\big\},
\end{align*}
where $\mathds{1}_{I^r_n}$ denotes the indicator function of the interval $I^r_n$.   

To treat the temporal error estimation for fully discrete state and adjoint equations, we need the following two interpolation operators in time. Define the interpolation operators $\Pi^r_\tau: C(\bar{I}; X) \to X^r_\tau$ and $\Pi^l_\tau: C(\bar{I}; X) \to X^l_\tau$ such that for any $v \in C(\bar{I}; X)$, $\Pi^r_\tau v$ and $\Pi^l_\tau v$ satisfy
\begin{align*}
\Pi^r_\tau v \big|_{I^r_n} = v(t_n), \quad \Pi^l_\tau v \big|_{I^l_n} = v(t_{n-1}), \quad n = 1, \dots, N_\tau,
\end{align*}
with $(\Pi^r_\tau v)(0) = v(0)$ and $(\Pi_\tau^l v)(T) = v(T)$, where $X$ is a given Banach space and $X^r_\tau := \big\{ v_\tau \in L^2(I; X) : v_\tau = \sum\limits_{n=1}^{N_{\tau}} v_n \mathds{1}_{I^r_n}, ~ v_n \in X, ~ n = 1, \dots, N_\tau \big\}$, and a similar definition for $X^l_\tau$ by replacing $I^r_n$ with $I^l_n$.  Note that the above interpolations are defined by taking the endpoint values on each subinterval, and the interpolation error estimate is provided in Appendix \ref{Appendix:A}, Lemma \ref{int:L2:time:err}.   In addition, define the $L^2$-projection $P_\tau: L^2(I;X)\to X^r_\tau$, such that $P_\tau w$ satisfies
\begin{align*}
P_\tau w\big|_{I^r_n}=\frac{1}{\tau_n}\int_{t_{n-1}}^{t_n}w \,dt\quad n=1,\cdots,N_\tau,\quad \forall w\in L^2(I;X).
\end{align*}
For simplicity we write $P^n_\tau w=P_\tau w\big|_{I^r_n}$, and the $L^2$-projection error estimate is provided in Appendix \ref{Appendix:A}, Lemma \ref{Proj:L2:time:err}.

To define the fully discrete optimal control problem, we need to consider the temporal discretization of the state equation \eqref{semi_discrete:state}. To achieve this goal, we use an implicit-explicit time discretization scheme. Given $u\in L^2(I;L^2(\Gamma))$, we seek a discrete solution pair $(\mathbf y_{n,h}, \theta_{n,h})\in \mathbb X_h\times V_h$, satisfying, for $n=1, \cdots, N_{\tau}$
\begin{equation}
\begin{aligned}\label{Full_discrete:state}
&\Big(\frac{\mathbf y_{n,h}-\mathbf y_{n-1,h}}{\tau_n}, \mathbf v_h\Big)+\nu \mathbf a(\mathbf y_{n,h}, \mathbf v_h)+\mathbf b(\mathbf y_{n-1,h},\mathbf y_{n,h},\mathbf v_h)+\beta(\theta_{n-1,h}\mathbf g, \mathbf v_h )=({\mathbf h}^n, \mathbf v_h),\\
&\Big(\frac{\theta_{n,h}-\theta_{n-1,h}}{\tau_n}, \psi_h\Big)+\chi a(\theta_{n,h},\psi_h)+b(\mathbf y_{n-1,h}, \theta_{n,h},\psi_{h})+\eta\gamma(\theta_{n,h}, \psi_h)_\Gamma=({f}^n, \psi_h)+\eta({u}^n, \psi_h)_\Gamma,\\
&\mathbf y_{0,h}=\mathbf y_{0h},\quad \theta_{0,h}=\theta_{0h}
\end{aligned}
\end{equation}
for any $(\mathbf v_h,\psi_h)\in \mathbb X_h\times V_h$, where $\mathbf{h}^n:=\frac{1}{\tau_n}\int_{t_{n-1}}^{t_{n}}  \mathbf h(t) dt$.  Moreover, for $(\mathbf y_\sigma, \theta_\sigma)\in \mathbb X_\sigma\times V_\sigma$ we set 
\begin{align}\label{Define:YthetaSigma}
(\mathbf y_\sigma, \theta_\sigma)\big|_{(0,T]}=\Big(\sum\limits_{n=1}^{N_{\tau}}\mathbf y_{n,h}\mathds{1}_{I^r_n},\sum\limits_{n=1}^{N_{\tau}} \theta_{n,h}\mathds{1}_{I^r_n}\Big) \text{ with } \mathbf y_\sigma(0)=\mathbf y_{0h},\   \theta_\sigma(0)=\theta_{0h}.
\end{align}

Since the state equation is a coupled system, we use the implicit-explicit scheme to decouple the computation of the velocity and temperature fields. Specifically, from the discrete scheme \eqref{Full_discrete:state}, at each time step, we separately solve the Navier-Stokes equations to obtain the velocity field \( \mathbf{y}_{n,h} \), and the heat equation to determine the temperature field \( \theta_{n,h} \). This approach significantly reduces computational cost compared to the discontinuous Galerkin time-stepping  method (cf. \cite{Thomee_1997}). Moreover, it is easy to show that the equation \eqref{Full_discrete:state} has a unique solution. 

\subsubsection{Analysis of the fully-discrete state equation}
In the following, we present the stability analysis of the numerical scheme \eqref{Full_discrete:state}.
\begin{lemma}\label{lem:Full:YTheta:bu} 
Given \( u \in L^2(I; L^2(\Gamma)) \), let $\big\{(\mathbf y_{n,h}, \theta_{n,h})\big\}_{n=1}^{N_\tau}$  be the solutions of \eqref{Full_discrete:state}. Then the following estimates hold:
\begin{align}
&\|\mathbf y_\sigma\|_{L^\infty(I;\mathbb L^2(\Omega))}+\nu^{\frac{1}{2}}\|\mathbf y_\sigma\|_{L^2(I;\mathbb H^1(\Omega))}+\|\theta_\sigma\|_{L^\infty(I;L^2(\Omega))}+\min\{\chi,\eta\gamma\}^{\frac{1}{2}}\|\theta_\sigma\|_{L^2(I;H^1(\Omega))}\leq \mathcal C_1\big(\|\mathbf y_{0h}\|\nonumber\\
&+\big(1+\beta|\mathbf g|\nu^{-\frac{1}{2}}\big)\|\theta_{0h}\|+\nu^{-\frac{1}{2}}\|\mathbf h\|_{L^2(I;\mathbb L^2(\Omega))}+\min\{\chi,\eta\gamma\}^{-\frac{1}{2}}\|f\|_{L^2(I;L^2(\Omega))}\nonumber\\
&+\gamma^{-\frac{1}{2}}\eta^{\frac{1}{2}}\|u\|_{L^2(I;L^2(\Gamma))}\big)=: \widetilde G_1,\label{Full:stab:state:ineq1}\\
&\|\mathbf y_\sigma\|_{L^\infty(I;\mathbb H^1(\Omega))}+\nu^\frac{1}{2}\|\mathbf A_h\mathbf y_\sigma\|_{L^2(I;\mathbb L^2(\Omega))}\leq \mathcal C_2\big(\big(1+\nu^{-2}\widetilde G_1^2\big)\|\nabla \mathbf y_{0h}\|+\tau_1^{\frac{1}{2}}\|\Delta \mathbf y_0\|+\tau_1^\frac{1}{2}\nu^{-\frac{1}{2}}\|\theta_{0h}\|\nonumber\\
&+\nu^{-\frac{1}{2}}\widetilde G_1+\|\mathbf h\|_{L^2(I;\mathbb L^2(\Omega))}\big)=:\widetilde G_2,\label{Full:stab:state:ineq2}\\
&\Big(\sum_{n=1}^{N_\tau}\frac{\|\mathbf y_{n,h}-\mathbf y_{n-1,h}\|^2}{\tau_n}\Big)^\frac{1}{2}\leq C\big(\big(\nu^{-\frac{1}{2}}+1\big)\widetilde G_1^2+\big(\nu^{-\frac{1}{2}}+1\big)\widetilde G_2^2+ \widetilde G_1\big),\label{Full:stab:state:ineq3}
\end{align}
where
$\mathcal C_1:=C\exp\big(C\nu^{-\frac{1}{2}}\beta|\mathbf g|\sqrt{T}\big),\quad \mathcal C_2:=C\exp\big(C\nu^{-2}\widetilde G_1^2\big)$.
\end{lemma}

\begin{proof}
We prove the estimate one by one.

$\bullet$ Estimate of \eqref{Full:stab:state:ineq1}.  Choosing $(\mathbf{v}_h,\psi_h) = \tau_n(\mathbf y_{n,h}, \theta_{n,h})$ in \eqref{Full_discrete:state}, we deduce that
\begin{align}
&\frac{1}{2}\|\mathbf y_{n,h}\|^2-\frac{1}{2}\|\mathbf y_{n-1,h}\|^2+\frac{1}{2}\|\mathbf y_{n,h}-\mathbf y_{n-1,h}\|^2+\nu\tau_n\|\nabla \mathbf y_{n,h}\|^2
+\frac{1}{2}\| \theta_{n,h}\|^2-\frac{1}{2}\|\theta_{n-1,h}\|^2\nonumber\\
&+\frac{1}{2}\|\theta_{n,h}-\theta_{n-1,h}\|^2+\chi\tau_n\|\nabla \theta_{n,h}\|^2+\eta\gamma\tau_n\|\theta_{n,h}\|^2_\Gamma\nonumber=-\beta\tau_n(\theta_{n-1,h}\mathbf g, \mathbf y_{n,h})+\tau_n({\mathbf h}^n, \mathbf y_{n,h})\nonumber\\
&+\eta\tau_n({u}^n, \theta_{n,h})_\Gamma+\tau_n({f}^n,\theta_{n,h})\nonumber\\
&\leq \frac{\nu\tau_n}{2}\|\nabla \mathbf y_{n,h}\|+\frac{\eta\gamma\tau_n}{2}\|\theta_{n,h}\|_\Gamma+\frac{\chi \tau_n}{2}\|\nabla \theta_{n,h}\|+\frac{C\beta^2|\mathbf g|^2\tau_n}{\nu}\|\theta_{n-1,h}\|^2+\frac{C}{\nu}\int_{t_{n-1}}^{t_n}\|\mathbf h\|^2dt\nonumber\\
&+\frac{C}{\min\{\chi,\eta\gamma\}}\int_{t_{n-1}}^{t_{n}}\|f\|^2 dt+\frac{\eta}{\gamma}\int_{t_{n-1}}^{t_{n}} \|u\|^2_\Gamma dt.\nonumber
\end{align}
Hence, summing over $n$ from $1$ to $k$ for $1 \leq k \leq N_{\tau}$, we have
\begin{align*}
&\|\mathbf y_{k,h}\|^2+\|\theta_{k,h}\|^2+\nu\sum_{n=1}^{k}\tau_n\|\nabla \mathbf y_{n,h}\|^2+\chi\sum_{n=1}^{k}\tau_n\|\nabla \theta_{n,h}\|+\eta\gamma\sum_{n=1}^k\tau_n\|\theta_{n,h}\|^2_\Gamma\nonumber\\
&\leq C\big(\|\mathbf y_{0h}\|^2+\|\theta_{0h}\|^2+\sum_{n=1}^{k} \frac{\beta^2|\mathbf g|^2\tau_{n}}{\nu}\|\theta_{n-1,h}\|^2+\frac{1}{\nu}\int_I\|\mathbf h\|^2dt+\frac{1}{\min\{\chi,\eta\gamma\}}\int_I\|f\|^2 dt+\frac{\eta}{\gamma}\int_I \|u\|^2_\Gamma dt\big).
\end{align*}
Then, using the discrete Gronwall inequality we can obtain the result \eqref{Full:stab:state:ineq1}.

$\bullet$ Estimate of \eqref{Full:stab:state:ineq2}. Taking $\mathbf{v}_h = -\tau_n\mathbf A_h\mathbf y_{n,h}$ in the first equation of \eqref{Full_discrete:state},  using \eqref{Interpolation:conti:1} and \eqref{Interpolation:A:1} we have
\begin{align}\label{Full:stab:process:1}
&\frac{1}{2}\|\nabla \mathbf y_{n,h}\|^2-\frac{1}{2}\|\nabla \mathbf y_{n-1,h}\|^2+\frac{1}{2}\|\nabla(\mathbf y_{n,h}-\mathbf y_{n-1,h})\|^2+\nu\tau_n\|\mathbf A_h\mathbf y_{n,h}\|^2=\tau_n\mathbf b(\mathbf y_{n-1,h},\mathbf y_{n,h},\mathbf A_h\mathbf y_{n,h})\nonumber\\
&+\tau_n\beta(\theta_{n-1,h}\mathbf g,\mathbf A_h\mathbf y_{n,h})-\tau_n({\mathbf h}^n, \mathbf A_h \mathbf y_{n,h})\nonumber\\
&\leq C\Big(\tau_n\|\mathbf y_{n-1,h}\|^\frac{1}{2}\|\nabla \mathbf y_{n-1,h}\|^\frac{1}{2}\|\nabla \mathbf y_{n,h}\|^\frac{1}{2}\|\mathbf A_h\mathbf  y_{n,h}\|^\frac{3}{2}+\tau_n\|\nabla \mathbf y_{n-1,h}\|^\frac{1}{2}\|\mathbf A_h\mathbf y_{n-1,h}\|^\frac{1}{2}\|\mathbf y_{n,h}\|^\frac{1}{2}\|\nabla \mathbf y_{n,h}\|^\frac{1}{2}\|\mathbf A_h\mathbf y_{n,h}\|\nonumber\\
&+\tau_n \|\theta_{n-1,h}\|\,\|\mathbf A_h\mathbf y_{n,h}\|+\big(\int_{t_{n-1}}^{t_n}\|\mathbf h\|^2 \,dt\big)^\frac{1}{2}\|\mathbf A_h\mathbf y_{n,h}\|\tau_n^\frac{1}{2}\Big)\nonumber\\
&\leq\frac{\nu\tau_n}{2}\|\mathbf A_h \mathbf y_{n,h}\|^2+\frac{\nu\tau_{n-1}}{4}\|\mathbf A_h \mathbf y_{n-1,h}\|^2+\big(G_n+H_n\big)\|\nabla \mathbf y_{n-1,h}\|^2+C\Big(\frac{\tau_n}{\nu}\|\theta_{n-1,h}\|^2+\frac{\tau_n}{\nu}\int_{t_{n-1}}^{t_n}\|\mathbf h\|^2 \,dt\Big),
\end{align}
with $\tau_0=\tau_1$, $G_n=\frac{C\tau_n}{\nu^3}\|\mathbf y_{n-1,h}\|^2\|\nabla \mathbf y_{n,h}\|^2$, $H_n=\frac{C\tau_n}{\nu^3}\|\nabla \mathbf y_{n,h}\|^2\|\mathbf y_{n,h}\|^2$. Hence, using the equality
\begin{align*}
&\frac{\nu\tau_n}{2}\|\mathbf A_h \mathbf y_{n,h}\|^2 = \frac{\nu\tau_n}{4}\|\mathbf A_h\mathbf y_{n,h}\|^2 + \frac{\nu\tau_n}{4}\|\mathbf A_h\mathbf y_{n,h}\|^2 - \frac{\nu\tau_{n-1}}{4}\|\mathbf A_h\mathbf y_{n-1,h}\|^2 + \frac{\nu\tau_{n-1}}{4}\|\mathbf A_h\mathbf y_{n-1,h}\|^2,
\end{align*}
and rearranging \eqref{Full:stab:process:1}, summing over $n$ from $1$ to $k$ for $1 \leq k \leq N_{\tau}$, we obtain
\begin{align}\label{Full:stab:process:bu:1}
&\|\nabla \mathbf y_{k,h}\|^2+\nu\sum_{n=1}^k\tau_n \|\mathbf A_h\mathbf y_{n,h}\|^2\leq \|\nabla \mathbf y_{0h}\|^2+\tau_1C\|\Delta \mathbf y_{0}\|^2+\big(G_1+H_1\big)\|\nabla \mathbf y_{0h}\|^2\nonumber\\
&+\sum_{n=1}^{k-1}\big(G_{n+1}+H_{n+1}\big)\|\nabla \mathbf y_{n,h}\|^2+C\big(\sum_{n=1}^{k}\frac{\tau_n}{\nu}\|\theta_{n-1,h}\|^2+\int_I\,\|\mathbf h\|^2dt\big),
\end{align}
where we have used the inequality $\|\mathbf A_h\mathbf y_{0,h}\|\leq C\|\Delta \mathbf y_0\|$. Then, using the discrete Gronwall inequality  we can obtain  \eqref{Full:stab:state:ineq2}.

$\bullet$ Estimate of \eqref{Full:stab:state:ineq3}. Taking  $\mathbf{v}_h = {\mathbf y_{n,h}-\mathbf y_{n-1,h}}$ in the first equation of \eqref{Full_discrete:state},  using \eqref{Interpolation:conti:1} and \eqref{Interpolation:A:1} we have
\begin{align*}
&\frac{1}{\tau_n}\|\mathbf y_{n,h}-\mathbf y_{n-1,h}\|^2+\frac{\nu}{2}\|\nabla \mathbf y_{n,h}\|^2-\frac{\nu}{2}\|\nabla \mathbf y_{n-1,h}\|^2+\frac{\nu}{2}\|\nabla(\mathbf y_{n-1,h}-\mathbf y_{n,h})\|^2=\\
&-\mathbf b(\mathbf y_{n-1,h},\mathbf y_{n,h},\mathbf y_{n,h}-\mathbf y_{n-1,h})-\beta(\theta_{n-1,h}\mathbf g, \mathbf y_{n,h}-\mathbf y_{n-1,h})+({\mathbf h}^n, \mathbf y_{n,h}-\mathbf y_{n-1,h})\\
&\leq \frac{\|\mathbf y_{n,h}-\mathbf y_{n-1,h}\|^2}{2\tau_n}+C\big(\tau_n\|\mathbf y_{n-1,h}\|\,\|\nabla \mathbf y_{n-1,h}\|\,\|\nabla \mathbf y_{n,h}\|\,\|\mathbf A_h \mathbf y_{n,h}\|\\
&+\tau_n\|\nabla \mathbf y_{n-1,h}\|\,\|\mathbf A_h \mathbf y_{n-1,h}\|\,\|\mathbf y_{n,h}\|\,\|\nabla \mathbf y_{n,h}\|+\tau_n\|\theta_{n-1,h}\|^2+\int_{t_{n-1}}^{t_n}\|\mathbf h\|^2 dt\big)\\
&\leq \frac{\|\mathbf y_{n,h}-\mathbf y_{n-1,h}\|^2}{2\tau_n}+C\big(\tau_n\|\mathbf y_{n-1,h}\|^2\|\nabla \mathbf y_{n,h}\|^2+\tau_n\|\nabla \mathbf y_{n-1,h}\|^2\|\mathbf A_h \mathbf y_{n,h}\|^2\\
&+\tau_n\|\nabla \mathbf y_{n-1,h}\|^2\|\mathbf A_h \mathbf y_{n-1,h}\|^2+\tau_n\|\mathbf y_{n,h}\|^2\|\nabla \mathbf y_{n,h}\|^2+\tau_n\|\theta_{n-1,h}\|^2+\int_{t_{n-1}}^{t_n}\|\mathbf h\|^2 dt\big).
\end{align*}
Summing over $n$ from $1$ to $N_\tau$ we obtain
\begin{align*}
&\sum_{n=1}^{N_\tau}\frac{1}{\tau_n}\|\mathbf y_{n,h}-\mathbf y_{n-1,h}\|^2\leq C\big(\sum_{n=1}^{N_\tau}\tau_n\|\mathbf y_{n-1,h}\|^2\|\nabla \mathbf y_{n,h}\|^2+\sum_{n=1}^{N_\tau}\tau_n\|\nabla \mathbf y_{n-1,h}\|^2\|\mathbf A_h \mathbf y_{n,h}\|^2\\
&+\sum_{n=1}^{N_\tau}\tau_n\|\nabla \mathbf y_{n-1,h}\|^2\|\mathbf A_h \mathbf y_{n-1,h}\|^2+\sum_{n=1}^{N_\tau}\tau_n\|\mathbf y_{n,h}\|^2\|\nabla \mathbf y_{n,h}\|^2+\sum_{n=1}^{N_\tau} \tau_n\|\theta_{n-1,h}\|^2+\int_I\|\mathbf h\|^2dt+\nu\|\nabla y_0\|^2\big).
\end{align*}
Finally, combining \eqref{Full:stab:state:ineq1} and \eqref{Full:stab:state:ineq2} yields \eqref{Full:stab:state:ineq3}.
\end{proof}
\begin{remark} 
If \( \mathbf{y}_0 \in \mathbb{X} \), we conclude from an inverse estimate that \( \|\mathbf{A}_h \mathbf{y}_{0 h}\| \leq Ch^{-1} \|\nabla \mathbf{y}_0\| \). Therefore, in view of \eqref{Full:stab:process:bu:1} we need to assume that \( \tau \leq C h^2 \) to obtain a higher regularity of  \((\mathbf y_\sigma,\theta_\sigma)\).
\end{remark}

\begin{lemma}\label{lem:Full:YTheta:time:err}  Given $u\in H^{\frac{1}{4}}(I;L^2(\Gamma))\cap L^2(I; H^{\frac{1}{2}}(\Gamma))$, let \((\mathbf y_h,\theta_h)\in L^2({I};\mathbb X_h)\times L^2({I};V_h)\) be the solution of \eqref{semi_discrete:state} and let $\big\{(\mathbf y_{n,h}, \theta_{n,h})\big\}_{n=1}^{N_\tau}$  be the solution of  the fully-discrete state equation \eqref{Full_discrete:state}. Then the following estimates hold:
\begin{equation}\label{Full:YTheta:time:err}
\begin{aligned}
\|\mathbf y_h-\mathbf y_\sigma\|_{L^\infty(I;\mathbb L^2(\Omega))}+\|\mathbf y_h-\mathbf y_\sigma\|_{L^2(I;\mathbb H^1(\Omega))}+\|\theta_h-\theta_\sigma\|_{L^\infty(I; L^2(\Omega))}+\|\theta_h-\theta_\sigma\|_{L^2(I; H^1(\Omega))}\leq C \tau^\frac{1}{2}.
\end{aligned}
\end{equation}
\end{lemma}
\begin{proof} 
We begin with the following decomposition:
\begin{equation}\label{FUll:decomp:YTheta_1}
\begin{aligned}
&\mathbf y_h-\mathbf y_{\sigma}=\mathbf y_h -\Pi^r_\tau \mathbf y_h+\Pi^r_\tau \mathbf y_h-\mathbf y_\sigma=\bm \zeta^{\mathbf y}_\tau+\bm \eta^{\mathbf y}_\tau,\ \ \theta_h-\theta_{\sigma}=\theta_h- \Pi^r_\tau \theta_h +\Pi^r_\tau\theta_h- \theta_\sigma=\zeta^{\theta}_\tau+\eta^{\theta}_\tau
\end{aligned}
\end{equation}
and set $\bm \eta^{\mathbf y }_{n, \tau}=\bm \eta^{\mathbf y}_{\tau}(t_n)$. Integrating from $t_{n-1}$ to $t_n$ for \eqref{semi_discrete:state}, subtracting \eqref{Full_discrete:state} from it with $(\mathbf v_h,\psi_h)= (\bm \eta^{\mathbf y}_{n, \tau}, \eta^\theta_{n, \tau})$ and using the decomposition \eqref{FUll:decomp:YTheta_1},  we obtain
\begin{align}
&\frac{1}{2}\|\bm \eta_{n,\tau}^{\mathbf y}\|^2-\frac{1}{2}\|\bm \eta_{n-1,\tau}^{\mathbf y}\|^2+\frac{1}{2}\| \eta_{n,\tau}^{\theta}\|^2-\frac{1}{2}\|\eta_{n-1,\tau}^{\theta}\|^2+\nu\tau_n\|\nabla \bm \eta_{n,\tau}^{\mathbf y}\|^2+\chi\tau_n\|\nabla \eta_{n, \tau}^\theta\|^2+\eta\gamma\tau_n\|\eta_{n,\tau}^\theta\|^2_\Gamma\nonumber\\
&\leq \int_{t_{n-1}}^{t_n}\big[-\nu \mathbf a(\bm \zeta^{\mathbf y}_\tau, \bm \eta_{n,\tau}^{\mathbf y})-\chi a( \zeta^{\theta}_\tau,  \eta_{n,\tau}^{\theta})-\eta\gamma(\zeta^{\theta}_\tau,\eta_{n,\tau}^{\theta})_\Gamma-\beta(\zeta_\tau^\theta+\theta_h(t_n)-\theta_h(t_{n-1})+\eta_{n-1,\tau}^\theta, \bm \eta_{n,\tau}^{\mathbf y}\cdot \mathbf g)\big] dt\nonumber\\
&+\int_{t_{n-1}}^{t_n}\big[\mathbf b(\mathbf y_{n-1,h},\mathbf y_{n,h}, \bm \eta_{n,\tau}^{\mathbf y})-\mathbf b(\mathbf y_h,\mathbf y_h, \bm\eta_{n,\tau}^{\mathbf y})\big]dt+\int_{t_{n-1}}^{t_n}\big[ b(\mathbf y_{n-1,h}, \theta_{n,h},  \eta_{n,\tau}^{\theta})- b(\mathbf y_h, \theta_h, \eta_{n,\tau}^{\theta})\big]dt\nonumber\\
&=J_1+J_2+J_3.\label{Full:err:YTheta:1}
\end{align}
By the Cauchy-Schwarz inequality, Young's inequality, and the trace inequality,  the term $J_1$ satisfies
\begin{align*}
J_1&\leq \frac{\nu}{4}\tau_n\|\nabla \bm \eta_{n,\tau}^{\mathbf y} \|+ \frac{\chi}{4}\tau_n\|\nabla\eta_{n,\tau}^{\theta} \|^2+ \frac{\eta\gamma}{4}\tau_n\|\eta_{n,\tau}^{\theta} \|^2_\Gamma+\frac{3\nu}{2}\int_{t_{n-1}}^{t_n}\|\nabla \bm \zeta^{\mathbf y}_\tau \|^2dt
+C\Big(\max\{\chi,\eta\gamma\}\int_{t_{n-1}}^{t_n}\|\zeta_{\tau}^\theta\|^2_{H^1(\Omega)}dt\\
&+\frac{\beta^2|\mathbf g|^2}{\nu}\big(\int_{t_{n-1}}^{t_n}\|\zeta_{\tau}^\theta\|^2dt+\tau_n\|\theta_h(t_n)-\theta_h(t_{n-1})\|^2+\tau_n\|\eta_{n-1,\tau}^\theta\|^2\big)\Big).
\end{align*}
Using Lemma \ref{semi_discreYTheta:stab}, \eqref{Interpolation:conti:1} and Young's inequality, we can estimate the trilinear term $J_2$ as follows:
\begin{align}
J_2&=\int_{t_{n-1}}^{t_n}-\mathbf b(\bm \zeta_\tau^{\mathbf y}, \mathbf y_h, \bm \eta_{n,\tau}^{\mathbf y})dt+\int_{t_{n-1}}^{t_n}-\mathbf b(\Pi^r_\tau \mathbf y_h, \bm \zeta_\tau^{\mathbf y}, \bm \eta_{n,\tau}^{\mathbf y})dt+\int_{t_{n-1}}^{t_n}-\mathbf b(\bm \eta_{n-1 ,\tau}^{\mathbf y},\Pi^r_\tau \mathbf y_h, \bm \eta_{n,\tau}^{\mathbf y})dt\nonumber\\
&+\int_{t_{n-1}}^{t_n}-\mathbf b(\mathbf y_h(t_n)-\mathbf y_{h}(t_{n-1}),\Pi^r_\tau \mathbf y_h, \bm \eta_{n,\tau}^{\mathbf y})dt\nonumber\\
&\leq\frac{\nu}{4}\tau_n\|\nabla \bm \eta^{\mathbf y}_{n,\tau}\|^2+\frac{\nu}{8}\tau_{n-1}\|\nabla \bm \eta^{\mathbf y}_{n-1,\tau}\|^2+C\Big(\frac{1}{\nu}\int_{t_{n-1}}^{t_n}\|\nabla \bm \zeta^{\mathbf y}_{\tau}\|^2dt+\frac{1}{\nu^3}\tau_{n-1}\|\bm \eta_{n-1,\tau}^{\mathbf y}\|^2\nonumber\\
&+\tau_n\frac{1}{\nu}\int_{t_{n-1}}^{t_{n}}\|\partial_t\mathbf y_{h}\|dt \,\|\nabla (\mathbf y_h(t_n)-\mathbf y_h(t_{n-1}))\|\Big).\label{Full:err:YTheta:bu:1}
\end{align}
Using a similar argument, we have the estimate for $J_3$
\begin{align*}
J_3&\leq\frac{\nu}{8}\tau_{n-1}\|\nabla \bm \eta_{n-1,\tau}^{\mathbf y}\|^2+ \frac{\chi}{4}\tau_n\|\nabla\eta_{n,\tau}^{\theta} \|^2+ \frac{\eta\gamma}{4}\tau_n\|\eta_{n,\tau}^{\theta} \|^2_\Gamma+C\Big(\frac{1}{\min\{\chi,\eta\gamma\}}\int_{t_{n-1}}^{t_n}\|\nabla\bm \zeta_\tau^{\mathbf y}\|^2 dt\\
&+\int_{t_{n-1}}^{t_n}\| \zeta_\tau^{\theta}\|^2_{H^1(\Omega)} dt+\frac{1}{\nu\min\{\chi,\eta\gamma\}^2}\tau_{n-1}\| \bm \eta_{n-1,\tau}^{\mathbf y}\|^2+\tau_n\frac{1}{\min\{\chi,\eta\gamma\}}\int_{t_{n-1}}^{t_{n}}\|\partial_t\mathbf y_{h}\|dt \,\|\nabla (\mathbf y_h(t_n)-\mathbf y_h(t_{n-1}))\|\Big).
\end{align*}
Thus, combining the above estimates with \eqref{Full:err:YTheta:1}  and using the equality $\frac{1}{2}a^2=\frac{1}{4}a^2+\frac{1}{4}a^2-\frac{1}{4}b^2+\frac{1}{4}b^2$,  we reorganize and sum over $n$ from $1$ to $k$ for $1 \leq k \leq N_{\tau}$ to \eqref{Full:err:YTheta:1}, obtaining
\begin{align*}
&\|\bm \eta_{k,\tau}^{\mathbf y}\|^2+\| \eta_{k,\tau}^{\theta}\|^2+\nu\sum_{n=1}^k\tau_n\|\nabla \bm \eta_{n,\tau}^{\mathbf y}\|^2+\chi\sum_{n=1}^k\tau_n\|\nabla \eta_{n, \tau}^\theta\|^2+\eta\gamma\sum_{n=1}^k\tau_n\|\eta_{n,\tau}^\theta\|^2_\Gamma\nonumber\\
&\leq C  \Big(\int_{0}^{t_k}\big[\|\nabla \bm \zeta^{\mathbf y}_\tau \|^2
+\|\zeta_{\tau}^\theta\|^2_{H^1(\Omega)}\big]dt+\tau\int_0^{t_k}\big[\|\partial_t \theta_h\|^2+\|\partial_t \mathbf y_h\|^2\big]dt+\tau\sum_{n=1}^k\tau_n\|\nabla (\mathbf y_h(t_n)-\mathbf y_h(t_{n-1}))\|^2\\
&+\sum_{n=1}^{k-1}\tau_{n}\big(\|\eta_{n,\tau}^\theta\|^2+\|\bm \eta_{n,\tau}^{\mathbf y}\|^2\big)\Big),
\end{align*}
where we set  $\tau_0:=\tau_1$ and used the fact $\bm \eta_{0,h}^{\mathbf y}=\eta_{0,h}^{\theta}=0$. The discrete Gronwall inequality implies
\begin{align*}
&\max_{1\leq n\leq N_\tau}\Big(\|\bm \eta_{n,\tau}^{\mathbf y}\|^2+\| \eta_{n,\tau}^{\theta}\|^2\Big)+\sum_{n=1}^{N_\tau}\tau_n\Big(\nu\|\nabla \bm \eta_{n,\tau}^{\mathbf y}\|^2+\min\{\chi,\eta\gamma\}\|\eta_{n, \tau}^\theta\|^2_{H^1(\Omega)}\Big)\\
&\leq C\Big(\|\nabla \bm \zeta^{\mathbf y}_\tau \|^2_{L^2(I;\mathbb L^2(\Omega))}+\|\zeta_{\tau}^\theta\|^2_{L^2(I;H^1(\Omega))}+\tau\big(\|\partial_t \theta_h\|^2_{L^2(I;L^2(\Omega))}+\|\partial_t \mathbf y_h\|_{L^2(I;\mathbf L^2(\Omega))}^2\big)\Big).
\end{align*}
This inequality, together with Lemmas \ref{int:L2:time:err}, \ref{semi_discreYTheta:stab} and \ref{semi_discreYTheta:High_regular:stab}, and the triangle inequality, proves \eqref{Full:YTheta:time:err}.
\end{proof}
\begin{remark} 
A similar error estimate for \( u \in L^2(I; L^2(\Gamma)) \) that is needed for the subsequent stability analysis of the adjoint equation (see Lemma \ref{lem:stab:Full:adjoint:bu}), is provided in Appendix \ref{Appendix:B}.
\end{remark}
\begin{remark}\label{remark:conv:Theta}
Using Lemmas 4.17 and 4.18 from \cite[p. 476]{vexler_wagner_2024} and the above lemma, we have
\begin{align}\label{err:conv:theta:bu}
\sup_{1\leq n\leq N_\tau}\tau_n \|\theta_{n,h}\|^2_{H^1(\Omega)}\overset{\tau\rightarrow 0}{\xrightarrow{\hspace{1cm}}} 0.
\end{align}
\end{remark}

We now present the continuous dependence of \( (\mathbf{y}_\sigma, \theta_\sigma) \) on the control $u$.

\begin{lemma}\label{Full:lem:Lip:YTheta} Given  \( \max\big\{ \|u\|_{H^{\frac{1}{4}}(I; L^2(\Gamma))}, \|u\|_{L^2(I; H^{\frac{1}{2}}(\Gamma))}, \|v\|_{L^2(I; L^2(\Gamma))} \big\} \leq M \), where \( M \) is a positive constant. Then for a sufficiently small $\tau$,  it follows that
\begin{align*}
\|\mathbf{y}_\sigma(u) - \mathbf{y}_\sigma(v)\|_{L^\infty(I; \mathbb{L}^2(\Omega))} &+ \|\theta_\sigma(u) - \theta_\sigma(v)\|_{L^\infty(I; L^2(\Omega))}+\|\mathbf{y}_\sigma(u) - \mathbf{y}_\sigma(v)\|_{L^2(I;\mathbb H^1(\Omega))}\\
&+\|\theta_\sigma(u)-\theta_\sigma(v)\|_{L^2(I;H^1(\Omega))} \leq C_{M} \|u - v\|_{L^2(I; L^2(\Gamma))},
\end{align*}
where \(  C_{M}\) is a positive constant depending on \( M \).
\end{lemma}
\begin{proof}
Let $\big(\mathbf e^{\mathbf y}_{n,h},e^{\theta}_{n,h}\big)=\big(\mathbf y_{n,h}(u)-\mathbf y_{n,h}(v),\theta_{n,h}(u)-\theta_{n,h}(v)\big)$. For any $n=1,\cdots, N_\tau$ it follows from \eqref{Full_discrete:state} that
\begin{equation}
\begin{aligned}
& \big(\mathbf e^{\mathbf y}_{n,h}-\mathbf e^{\mathbf y}_{n-1,h}, \mathbf v_h\big)+\tau_n\nu\nu \mathbf a(\mathbf e^{\mathbf y}_{n,h}, \mathbf v_h)+\tau_n\mathbf b(\mathbf e^{\mathbf y}_{n-1,h}, \mathbf y_{n,h}(u), \mathbf v_h)
+\tau_n\mathbf b(\mathbf y_{n-1,h}(v), \mathbf e^{\mathbf y}_{n,h},\mathbf v_h)\\
&+\tau_n\beta(e^{\theta}_{n-1,h}\mathbf g, \mathbf v_h)=0,\\
&(e^{\theta}_{n,h}-e^{\theta}_{n-1,h}, \psi_h)+\chi\tau_n a(e^{\theta}_{n-1,h}, \psi_h)+\tau_n b(\mathbf e^{\mathbf y}_{n-1,h}, \theta_{n,h}(u), \psi)+\tau_n b(\mathbf y_{n-1,h}(v), e^{\theta}_{n,h}, \psi_h)\\
&+\eta\gamma(e^{\theta}_{n,h}, \psi_h)= \tau_n \eta ({u}^n-{v}^n, \psi_h)_\Gamma,\\
&\mathbf e_{0,h}^{\mathbf y}=0,\quad \ e_{0,h}^\theta=0.
\end{aligned}
\end{equation}
Taking $(\mathbf v_h,\psi_h)=(\mathbf e_{n,h}^{\mathbf y}, e_{n,h}^\theta)$, we obtain
\begin{align}
&\frac{1}{2}\|\mathbf e_{n,h}^{\mathbf y}\|^2-\frac{1}{2}\|\mathbf e^{\mathbf y}_{n-1,h}\|^2+\frac{1}{2}\| e_{n,h}^{ \theta}\|^2-\frac{1}{2}\| e_{n-1,h}^{\theta}\|^2+\nu\tau_n\|\nabla \mathbf e_{n,h}^{\mathbf y}\|^2+\chi\tau_n\|\nabla e^{\theta}_{n,h}\|^2+\eta\gamma \tau_n\|e^{\theta}_{n,h}\|_\Gamma^2\nonumber\\
&\leq -\tau_n\mathbf b(\mathbf e_{n-1,h}^{\mathbf y}, \mathbf y_{n,h}(u), \mathbf e_{n,h}^{\mathbf y})-\tau_n\mathbf b(\mathbf e^{\mathbf y}_{n-1,h}, \theta_{n,h}(u),e_{n,h}^{\theta})+\big[\tau_n\eta(u^n-v^n, e_{n,h}^\theta)_\Gamma-\tau_n\beta ( e_{n-1,h}^{\theta}\mathbf g, \mathbf e_{n,h}^{\mathbf y})\big]\nonumber\\
&=J_1+J_2+J_3.\label{Full:YTheta:Lip}
\end{align}
Using \eqref{Interpolation:conti:1} and \eqref{Interpolation:A:1}, we can deduce that
\begin{align*}
&J_1\leq G_n\|\mathbf e_{n-1,h}^{\mathbf y}\|^2+\frac{\nu}{4}\tau_n\|\nabla \mathbf e_{n,h}^{\mathbf y}\|^2+\frac{\nu}{8}\tau_{n-1}\|\nabla \mathbf e^{\mathbf y}_{n-1,h}\|^2
\end{align*}
with $G_n=\tau_nC\big(\frac{1}{\nu} \big(\|\nabla \mathbf y_{n-1,h}\|^2+\|\mathbf A_h\mathbf y_{n-1,h}\|^2\big)+\frac{1}{\nu^3}\|\nabla\mathbf y_{n,h}(u)\|^4\big)$ and $\tau_0:=\tau_1$. From inequality \eqref{Interpolation:conti:1} we deduce that
\begin{align*}
J_2&\leq C\tau_n\|\mathbf e^{\mathbf y}_{n-1,h}\|^\frac{1}{2}\|\nabla \mathbf e^{\mathbf y}_{n-1,h}\|^\frac{1}{2}\|e_{n,h}^{\theta}\|^\frac{1}{2}\| e_{n,h}^{\theta}\|^\frac{1}{2}_{H^1(\Omega)}\|\theta_{n,h}(u)\|_{H^1(\Omega)}\nonumber\\
&+C\tau_n\|\mathbf e^{\mathbf y}_{n-1,h}\|^\frac{1}{2}\|\nabla \mathbf e^{\mathbf y}_{n-1,h}\|^\frac{1}{2}\|\nabla e_{n,h}^\theta\|\,\|\theta_{n,h}(u)\|^\frac{1}{2}\|\theta_{n,h}(u)\|^\frac{1}{2}_{H^1(\Omega)}.
\end{align*}
We can apply Young's inequality to obtain
\begin{align*}
J_2\leq \frac{\nu}{8}\tau_{n-1}\|\nabla \mathbf e_{n-1,h}^{\mathbf y}\|^2
&+\frac{\chi}{2}\tau_n\|\nabla e_{n,h}^\theta\|^2+\frac{\eta\gamma}{4}\tau_n\| e_{n,h}^\theta\|^2_\Gamma+C\Big(H_n\|\mathbf e^{\mathbf y}_{n-1,h}\|^2+\tau_n \|\theta_{n,h}(u)\|_{H^1(\Omega)}^2\|e^\theta_{n,h}\|^2\Big)
\end{align*}
with $H_n=\tau_n\|\theta_{n,h}(u)\|^2_{H^1(\Omega)}+\tau_n\|\theta_{n,h}(u)\|^2\|\theta_{n,h}(u)\|^2_{H^1(\Omega)} $. By the Cauchy-Schwarz, Young's and Sobolev inequalities we have
\begin{align*}
J_3\leq \frac{\nu}{4}\tau_n\|\nabla \mathbf e_{n,h}^{\mathbf y}\|^2+\frac{\eta\gamma}{4}\tau_n\| e_{n,h}^\theta\|^2_\Gamma+C\tau_n \Big(\|e_{n-1,h}^{\theta}\|^2+\|u^n-v^n\|^2_\Gamma\Big).
\end{align*}
Thus, combining the above estimates with \eqref{Full:YTheta:Lip} and summing over $n = 1, \dots, k$ for any $1 \leq k \leq N_\tau$, we obtain the following bound
\begin{align*}
\|\mathbf e_{k,h}^{\mathbf y}\|^2&+\| e_{k,h}^{ \theta}\|^2+\nu\tau_k\|\nabla \mathbf e_{k,h}^{\mathbf y}\|^2+\chi\tau_k\|\nabla e^{\theta}_{k,h}\|^2+\eta\gamma \tau_k\|e^{\theta}_{k,h}\|_\Gamma^2\leq C\Big(\sum_{n=1}^k\Big( (G_n+H_n)\|\mathbf e_{n-1,h}^{\mathbf y}\|^2\\
&+\tau_n\|e_{n-1,h}^{\theta}\|^2\Big)+\sum_{n=1}^k\tau_n\|\theta_{n,h}(u)\|^2_{H^1(\Omega)}\|
e_{n,h}^\theta\|^2+\int_0^{t_k}\|u-v\|_\Gamma^2dt\Big).
\end{align*}
Remark \ref{remark:conv:Theta} implies that $\tau_n\|\theta_{n,h}(u)\|^2_{H^1(\Omega)} \to 0$ uniformly in $n$ as $\tau \to 0$. Hence, we can choose the temporal mesh size fine enough so that $C  \tau_n \|\theta_{n,h}(u)\|^2_{H^1(\Omega)} \leq \frac{1}{2}$ holds for all $n$. This together with Lemma \ref{lem:Full:YTheta:bu} and the Gronwall inequality finishes the proof.
\end{proof}

\begin{remark} 
In Lemma \ref{Full:lem:Lip:YTheta}, one can use \( \max\big\{ \|u\|_{L^2(I; L^2(\Gamma))}, \|v\|_{L^2(I; L^2(\Gamma))} \big\} \leq M \) to replace the original hypothesis \( \max\big\{ \|u\|_{H^{\frac{1}{4}}(I; L^2(\Gamma))}, \|u\|_{L^2(I; H^{\frac{1}{2}}(\Gamma))}, \|v\|_{L^2(I; L^2(\Gamma))} \big\} \leq M \). Then, for a sufficiently small \( \sigma \), the conclusion of Lemma \ref{Full:lem:Lip:YTheta} can be derived similarly from Lemma \ref{appB:Lem:full:state:err}.
\end{remark}

Similarly, using a suitable duality argument we can deduce the improved temporal convergence order for the fully discrete state variable under the norm $L^2(I;\mathbb L^2(\Omega))$ ( or $L^2(I;L^2(\Omega))$).
\begin{lemma}\label{Full_err:YTetha:time:L2} Under the assumptions of Lemma \ref{lem:Full:YTheta:time:err} , the following estimate holds:
\begin{equation}\label{Full:YThetaL2:time:err}
\begin{aligned}
\|\mathbf y_h-\mathbf y_\sigma\|_{L^2(I;\mathbb L^2(\Omega))}+\|\theta_h-\theta_\sigma\|_{L^2(I; L^2(\Omega))}\leq C \tau.
\end{aligned}
\end{equation}
\end{lemma}
\begin{proof}
We use the same decomposition as in \eqref{FUll:decomp:YTheta_1}. We can derive from the error estimate of the interpolation $\Pi^r_\tau$ given by Lemma \ref{int:L2:time:err} that
\begin{align*}
\|\mathbf y_h-\mathbf y_\sigma\|_{L^2(I;\mathbb L^2(\Omega))}+\|\theta_h-\theta_\sigma\|_{L^2(I;L^2(\Omega))}\leq C \tau+\|\bm \eta^{\mathbf y}_\tau\|_{L^2(I;\mathbb L^2(\Omega))}+\|\eta^\theta_\tau\|_{L^2(I; L^2(\Omega))}.
\end{align*}
We now use a duality argument to provide an estimate for $\|\bm \eta^{\mathbf y}_\tau\|_{L^2(I;\mathbb L^2(\Omega))} + \|\eta^\theta_\tau\|_{L^2(I; L^2(\Omega))}$.  We begin by defining a discrete dual problem: seek a discrete solution pair $(\bm \varphi _{n,h}, z_{n,h})\in \mathbb X_h\times V_h$, satisfying, for $n=N_{\tau}, \cdots, 1$ and any $(\mathbf w_h,\zeta_h)\in \mathbb X_h\times V_h$,
\begin{align}
&\Big(\frac{\bm \varphi_{n,h}-\bm \varphi_{n+1,h}}{\tau_n}, \mathbf w_h\Big)+\nu \mathbf a(\bm \varphi_{n,h}, \mathbf w_h)+\frac{\tau_{n+1}}{\tau_{n}}\mathbf b(\mathbf w_h,\mathbf y_{n+1,h},\bm \varphi_{n+1,h})+\mathbf b(\mathbf y_h(t_{n-1}),\mathbf w_h, \bm \varphi_{n,h})\nonumber\\
&+ b(\mathbf w_h, \theta_h(t_{n}), z_{n,h})=({\bm \eta}^{\mathbf y}_{n,\tau}, \mathbf w_h),\nonumber\\
&\Big(\frac{z_{n,h}-z_{n+1,h}}{\tau_n}, \zeta_h\Big)+\chi a(z_{n,h},\zeta_h)+b(\mathbf y_{n-1,h}, \zeta_h,z_{n,h})+\eta\gamma(z_{n,h}, \zeta_h)_\Gamma+\frac{\tau_{n+1}}{\tau_{n}}\beta(\bm \varphi_{n+1,h}, \zeta_h \mathbf g)=({\eta_{n,\tau}^{\theta}}, \zeta_h),\nonumber\\
&\varphi_{N_{\tau}+1,h}=0,\ \mathbf y_{N_{\tau}+1, h}:=\mathbf y_{N_{\tau}, h}, \quad z_{N_{\tau}+1,h}=0,\ \tau_{N_\tau+1}:=\tau_{N_\tau}.\label{dual:Ytheta:Full_discrete:state}
\end{align}
Similar to the proof of Lemmas \ref{lem:Full:YTheta:bu} and \ref{lem:stab:Full:adjoint:bu}, we know that \eqref{dual:Ytheta:Full_discrete:state} has a unique solution $\big\{(\bm \varphi_{n,h},z_{n,h})\big\}_{n=1}^{N_\tau}$ and the following stability estimate holds:
 \begin{align}\label{Full_err:YTetha:time:L2:bu:3}
\max_{1\leq n\leq N_\tau}\Big(\|z_{n,h}\|^2_{H^1(\Omega)}+\|\nabla \bm \varphi_{n,h}\|^2\Big)+&\sum_{n=1}^{N_\tau}\tau_n\Big(\|\nabla \bm \varphi_{n,h}\|^2+\|z_{n,h}\|^2_{H^1(\Omega)}\Big)+\sum_{n=1}^{N_\tau}\tau_n\Big(\|\mathbf A_h \bm \varphi_{n,h}\|^2+\|\Delta_hz_{n,h}\|^2\Big)\nonumber\\
&\leq C\Big(\|\bm \eta_\tau^{\mathbf y}\|^2_{L^2(I;\mathbb L^2(\Omega))}+\|\eta_\tau^\theta\|^2_{L^2(I;L^2(\Omega))}\Big).
\end{align}
Setting $(\mathbf w_h,\zeta_h)=\tau_n(\bm \eta^{\mathbf y}_{n,\tau},\eta_{n,\tau}^{\theta})$ in \eqref{dual:Ytheta:Full_discrete:state},  and summing  over $N_\tau$ to $1$, we obtain
\begin{align}\label{Full_err:YTetha:time:L2:1}
&\|\bm \eta^{\mathbf y}_\tau\|^2_{L^2(I;\mathbb L^2(\Omega))}+\| \eta^\theta_\tau\|^2_{L^2(I;L^2(\Omega))}=\sum_{n=1}^{N_\tau}\big(\bm \varphi_{n,h}-\bm \varphi_{n+1,h},\bm \eta_{n,\tau}^{\mathbf y}\big)+\nu\sum_{n=1}^{N_\tau} \tau_n\mathbf a(\bm \varphi_{n,h}, \bm \eta_{n,\tau}^{\mathbf y})+\sum_{n=1}^{N_\tau}\tau_{n+1}\mathbf b(\bm \eta_{n,\tau}^{\mathbf y},\mathbf y_{n+1,h},\bm \varphi_{n+1,h})\nonumber\\
&+\sum_{n=1}^{N_\tau}\tau_n\mathbf b(\mathbf y_h(t_{n-1}),\bm \eta_{n,\tau}^{\mathbf y}, \bm \varphi_{n,h})+\sum_{n=1}^{N_\tau}\tau_n b(\bm \eta_{n,\tau}^{\mathbf y}, \theta_{h}(t_n), z_{n,h})+\sum_{n=1}^{N_\tau}\Big(z_{n,h}-z_{n+1,h}, \eta^\theta_{n,\tau}\Big)\nonumber\\
&+\chi\sum_{n=1}^{N_\tau}\tau_n a(z_{n,h},\eta^\theta_{n,\tau})+\sum_{n=1}^{N_\tau}\tau_nb(\mathbf y_{n-1,h}, \eta^\theta_{n,\tau},z_{n,h})+\eta\gamma\sum_{n=1}^{N_\tau}\tau_n(z_{n,h}, \eta^\theta_{n,\tau})_\Gamma+\beta\sum_{n=1}^{N_\tau}\tau_{n+1}(\bm \varphi_{n+1,h}, \eta^\theta_{n,\tau} \mathbf g).
\end{align}

For \eqref{semi_discrete:state} integrating from $t_{n-1}$ to $t_n$, and  subtracting in \eqref{Full_discrete:state} multiple $\tau_n$, from it with \((\mathbf v_h,\psi_h)=(\bm \varphi_{n,h}, z_{n,h}) \)  yields
\begin{align*}
&\nu\tau_n\mathbf a(\bm \eta_{n,\tau}^{\mathbf y},\bm \varphi_{n,h})=-(\bm \eta_{n,\tau}^{\mathbf y}-\bm \eta_{n-1,\tau}^{\mathbf y},\bm \varphi_{n,h})-\nu\int_{t_{n-1}}^{t_n}\mathbf a(\bm \zeta^{\mathbf y}_\tau, \bm \varphi_{n,h})dt-\tau_n \beta(\eta_{n-1,\tau}^\theta\mathbf g, \bm\varphi_{n,h})\\
&-\tau_n\beta (\theta_h(t_n)-\theta_h(t_{n-1}),\bm \varphi_{n,h}\cdot \mathbf g)-\beta\int_{t_{n-1}}^{t_n}(\zeta^\theta_{\tau}\mathbf g, \bm \varphi_{n,h})dt-\tau_n \mathbf b(\bm \eta^{\mathbf y}_{n-1,\tau}, \mathbf y_{n,h}, \bm \varphi_{n,h})-\tau_n \mathbf b(\mathbf y_h(t_{n-1}),\bm \eta^{\mathbf y}_{n,\tau}, \bm \varphi_{n,h})\\
&-\tau_n\mathbf b(\mathbf y_h(t_n)-\mathbf y_h(t_{n-1}),\mathbf y_h(t_n),\bm \varphi_{n,h})- \int_{t_{n-1}}^{t_n}\mathbf b(\mathbf y_h(t_n), \bm \zeta_\tau^{\mathbf y},\bm \varphi_{n,h})-\int_{t_{n-1}}^{t_n}\mathbf b(\bm \zeta_\tau^{\mathbf y}, \mathbf y_h, \bm \varphi_{n,h})dt,\\
&\chi\tau_n a(\eta_{n,\tau}^\theta, z_{n,h})+\tau_n\eta\gamma(\eta^\theta_{n,\tau}, z_{n,h})_\Gamma=-(\eta_{n,\tau}^\theta-\eta_{n-1, \tau}^\theta, z_{n,h})-\chi \int_{t_{n-1}}^{t_n} a(\zeta_\tau^\theta, z_{n,h})dt -\eta\gamma\int_{t_{n-1}}^{t_n}(\zeta_\tau^\theta, z_{n,h})_\Gamma dt\\
&-\tau_n b(\mathbf y_{n-1,h}, \eta_{n,\tau}^\theta,z_{n,h})-\tau_n b(\mathbf y_{n,h}-\mathbf y_{n-1,h},\theta_h(t_n), z_{n,h})-\tau_n b(\bm \eta^{\mathbf y}_{n,\tau}, \theta_{h}(t_n),z_{n,h})- \int_{t_{n-1}}^{t_n}b(\mathbf y_h(t_n),\zeta^\theta_\tau, z_{n,h})dt\\
&- \int_{t_{n-1}}^{t_n}b(\bm \zeta^{\mathbf y }_\tau,\theta_h, z_{n,h})dt.
\end{align*}
We substitute the above equation into \eqref{Full_err:YTetha:time:L2:1} and perform the necessary simplifications to obtain
\begin{align}
&\|\bm \eta^{\mathbf y}_\tau\|^2_{L^2(I;\mathbb L^2(\Omega))}+\| \eta^\theta_\tau\|^2_{L^2(I;L^2(\Omega))}=-\nu\sum_{n=1}^{N_\tau}\int_{t_{n-1}}^{t_n}\mathbf a(\bm \zeta^{\mathbf y}_\tau, \bm \varphi_{n,h})dt+\sum_{n=1}^{N_\tau} \Big(-\tau_n(\theta_h(t_n)-\theta_h(t_{n-1}),\bm \varphi_{n,h}\cdot \mathbf g)\nonumber\\
&-\beta\int_{t_{n-1}}^{t_n}(\zeta^\theta_{\tau}\mathbf g, \bm \varphi_{n,h})dt\Big)+\sum_{n=1}^{N_\tau}\Big(-\tau_n\mathbf b(\mathbf y_h(t_n)-\mathbf y_h(t_{n-1}),\mathbf y_h(t_n),\bm \varphi_{n,h})- \int_{t_{n-1}}^{t_n}\mathbf b(\mathbf y_h(t_n), \bm \zeta_\tau^{\mathbf y},\bm \varphi_{n,h})\nonumber\\
&-\int_{t_{n-1}}^{t_n}\mathbf b(\bm \zeta_\tau^{\mathbf y}, \mathbf y_h, \bm \varphi_{n,h})dt\Big)\nonumber+\sum_{n=1}^{N_\tau}\Big(-\chi \int_{t_{n-1}}^{t_n} a(\zeta_\tau^\theta, z_{n,h})dt -\eta\gamma\int_{t_{n-1}}^{t_n}(\zeta_\tau^\theta, z_{n,h})_\Gamma dt\Big)\nonumber\\
&+\sum_{n=1}^{N_\tau}\Big(-\tau_n b(\mathbf y_{n,h}-\mathbf y_{n-1,h},\theta_h(t_n), z_{n,h})- \int_{t_{n-1}}^{t_n}b(\mathbf y_h(t_n),\zeta^\theta_\tau, z_{n,h})dt- \int_{t_{n-1}}^{t_n}b(\bm \zeta^{\mathbf y }_\tau,\theta_h, z_{n,h})dt\Big)\nonumber\\
&=J_1+J_2+J_3+J_4+J_5.\label{Full_err:YTetha:time:L2:2}
\end{align}
For $J_1$ and $J_2$,  using the definition of the operator $\mathbf A_h$ and the Cauchy-Schwarz inequality, we have
\begin{align*}
J_1&\leq \nu \|\bm \zeta^{\mathbf y}_\tau\|_{L^2(I;\mathbb L^2(\Omega))}\big(\sum_{k=1}^{N_\tau}\tau_n\|\mathbf A_h\bm \varphi_{n,h}\|^2\big)^\frac{1}{2},\\
J_2&\leq C|\mathbf g|\big(\tau\|\partial_t\theta_{h}\|_{L^2(I;L^2(\Omega))}+\|\zeta^\theta_\tau\|_{L^2(I;L^2(\Omega))}\big)\big(\sum_{n=1}^{N_\tau}\tau_n\|\bm \varphi_{n,h}\|^2\big)^\frac{1}{2}.
\end{align*}
Using \eqref{Interpolation:conti:1}, \eqref{Interpolation:A:1} and Lemma \ref{semi_discreYTheta:stab}, we deduce
\begin{align*}
J_3\leq C\Big(\tau\|\partial_t\mathbf y_h\|_{L^2(I;\mathbb L^2(\Omega))}+\|\bm \zeta^{\mathbf y}_\tau\|_{L^2(I;\mathbb L^2(\Omega))}\Big)&\Big(\Big(\max_{1\leq n\leq N_\tau}\|\bm \varphi_{n,h}\|\Big)^\frac{1}{2}\Big(\sum_{n=1}^{N_\tau}\tau_n\|\nabla \bm \varphi_{n,h}\|^2\Big)^\frac{1}{4}\\
&+\Big(\max_{1\leq n\leq N_\tau}\|\nabla \bm \varphi_{n,h}\|\Big)^\frac{1}{2}\Big(\sum_{n=1}^{N_\tau}\tau_n\|\mathbf A_h \bm \varphi_{n,h}\|^2\Big)^\frac{1}{4}\Big).
\end{align*}
To treat the term $J_4$, we use the operator $\Delta_h$ defined in \eqref{definite_Delta_h} to obtain
\begin{align*}
  J_4=\sum_{n=1}^{N_\tau}\int_{t_{n-1}}^{t_n}(\zeta^{\theta}_\tau, \Delta_hz_{n,h})\leq \|\zeta^{\theta}_\tau\|_{L^2(I; L^2(\Omega))}\Big(\sum_{k=1}^{N_\tau}\tau_n\| \Delta_h z_{n,h}\|^2\Big)^\frac{1}{2}.
\end{align*}
Using \eqref{DeltaTheta_ineq1}, \eqref{semi:Y_Theta:stab:bu:ineq8}, Lemmas \ref{semi_discreYTheta:stab} and \ref{lem:Full:YTheta:bu} to $J_5$, we can derive
\begin{align*}
  J_5&\leq C\Big(\tau+\|\zeta^\theta_\tau\|_{L^2(I;L^2(\Omega))}+\|\bm \zeta^{\mathbf y}_\tau\|_{L^2(I;\mathbb L^2(\Omega))}\Big)\Big(\big(\sum_{k=1}^{N_\tau}\tau_n\|z_{n,h}\|^2_{H^1(\Omega)}\big)^\frac{1}{4}\big(\sum_{k=1}^{N_\tau}\tau_n\| \Delta_h z_{n,h}\|^2\big)^\frac{1}{4}\\
  &+\max_{1\leq n\leq N_\tau}\| z_{n,h}\|_{H^1(\Omega)}\Big).
\end{align*}
Hence, combining the estimates for $J_1$--$J_5$ with \eqref{Full_err:YTetha:time:L2:2} and \eqref{Full_err:YTetha:time:L2:bu:3}, we have the bound
\begin{align*}
\|\bm \eta^{\mathbf y}_\tau\|_{L^2(I;\mathbb L^2(\Omega))}+\| \eta^\theta_\tau\|_{L^2(I;L^2(\Omega))}\leq C\big(\tau+\|\bm \zeta^{\mathbf y}_\tau\|_{L^2(I;\mathbb L^2(\Omega))}+\|\zeta^\theta_\tau\|_{L^2(I;L^2(\Omega))}\big).
\end{align*}
Applying Lemma \ref{int:L2:time:err}, Lemmas \ref{semi_discreYTheta:stab} and \ref{semi_discreYTheta:High_regular:stab} to above estimates yields the final estimate \eqref{Full:YThetaL2:time:err}.
\end{proof}

\subsubsection{Analysis of the discrete adjoint equation}
Let us define the discrete control-to-state operator $S_\sigma: L^2(I;L^2(\Gamma))\to \mathbb X_\sigma\times V_\sigma $ by $(\mathbf y_\sigma(u), \theta_\sigma(u))=\mathcal S_\sigma(u)$, where $(\mathbf y_\sigma(u), \theta_\sigma(u))$ is given in \eqref{Define:YthetaSigma} with the control $u\in L^2(I;L^2(\Gamma))$. Then, associated to the discrete state equation \eqref{Full_discrete:state}, the discrete control problem reads as follows:
\begin{equation*}
\mathrm{(P_\sigma)}\quad\begin{aligned}
		\min\limits_{u\in \mathcal{U}_{ad}}\quad J_\sigma(u)=\frac{1}{2}\int_I\int_{\Omega} |\mathbf{y}_\sigma(u)-\mathbf{y}_d|^2dx dt&+\frac{1}{2}\int_I\int_{\Omega}|\theta_{\sigma}(u)- \theta_d|^2dx dt+\frac{\alpha}{2}\int_I\int_{\Gamma}| u|^2ds dt.
	\end{aligned}
\end{equation*}
Here, the control variables are not explicitly discretized. Instead, a variational discretization method is used (see \cite{Hinze_2005}). It is easy to verify that the discrete problem $\mathrm{(P_\sigma)}$ has at least one solution.

We can  see that the map $\mathcal S_\sigma$ is infinitely often Fre\'chet-differentiable. Hence,  applying the chain rule, we see that $J_\sigma :L^2(I;L^2(\Gamma))\to \mathbb R$ is of  class $C^\infty$. Moreover, given a direction \( v \in L^2(I; L^2(\Gamma)) \), the first order derivative is expressed as follows:
\begin{align*}
J'_\sigma(u)v = \int_I (\mathbf{y}_\sigma(u) - \mathbf{y}_d, \mathbf{z}_\sigma) \, dt + \int_I (\theta_\sigma(u) - \theta_d, \xi_\sigma) \, dt + \alpha\int_I (u, v)_\Gamma \, dt,
\end{align*}
where \( (\mathbf{z}_\sigma, \xi_\sigma) = S'_\sigma(u)v \) and $(\mathbf{z}_\sigma, \xi_\sigma)\big|_{(0,T]} = \big( \sum\limits_{n=1}^{N_{\tau}} \mathbf{z}_{n,h} \mathds{1}_{I^r_n}, \sum\limits_{n=1}^{N_{\tau}} \xi_{n,h} \mathds{1}_{I^r_n} \big)  \text{ with } \mathbf{z}_\sigma(0) =0, \, \theta_\sigma(0) = 0,$ where \( \{ (\mathbf{z}_{n,h}, \xi_{n,h}) \}_{n=1}^{N_\tau} \) is the solution of the following system of equations:
\begin{align}\label{Full_discrete:Linzied:state}
&\Big(\frac{\mathbf z_{n,h}-\mathbf z_{n-1,h}}{\tau_n}, \mathbf v_h\Big)+\nu \mathbf a(\mathbf z_{n,h}, \mathbf v_h)+\mathbf b(\mathbf z_{n-1,h},\mathbf y_{n,h},\mathbf v_h)+\mathbf b(\mathbf y_{n-1,h},\mathbf z_{n,h},\mathbf v_h)+\beta(\xi_{n-1,h}\mathbf g, \mathbf v_h )=\mathbf 0,\nonumber\\
&\Big(\frac{\xi_{n,h}-\xi_{n-1,h}}{\tau_n}, \psi_h\Big)+\chi a(\xi_{n,h},\psi_h)+b(\mathbf z_{n-1,h}, \theta_{n,h},\psi_{h})+b(\mathbf y_{n-1,h}, \xi_{n,h},\psi_{h})+\eta\gamma(\xi_{n,h}, \psi_h)_\Gamma=\eta({v}^n, \psi_h)_\Gamma,\nonumber\\
&\mathbf z_{0,h}=0,\quad \xi_{0,h}=0
\end{align}
for $n=1, \cdots, N_{\tau}$ and any $(\mathbf v_h, \psi_h)\in \mathbb X_h\times V_h$, where $\big\{(\mathbf y_{n,h}, \theta_{n,h})\big\}_{n=1}^{N_\tau}$ is the solution of the fully discrete state equation \eqref{Full_discrete:state}. As is common in control theory (see \cite{Troltzsh_2010}), the adjoint state is usually introduced to simplify the expression of this derivative. To this end, we consider the discrete adjoint equation: seek a discrete solution pair $(\bm \mu_{n,h}, \kappa_{n,h})$, satisfying, for $n=N_\tau, \cdots, 1$ and any $(\mathbf w_h,\zeta_h)\in \mathbb X_h\times V_h$,
\begin{align}\label{MuKappa:Full_discrete:adjoint}
&\Big(\frac{\bm \mu_{n,h}-\bm \mu_{n+1,h}}{\tau_n}, \mathbf w_h\Big)+\nu \mathbf a(\bm \mu_{n,h}, \mathbf w_h)+\frac{\tau_{n+1}}{\tau_n}\mathbf b(\mathbf w_h,\mathbf y_{n+1,h},\bm \mu_{n+1,h})+\mathbf b(\mathbf y_{n-1,h},\mathbf w_h, \bm \mu_{n,h})\nonumber\\
&+\frac{\tau_{n+1}}{\tau_n}b(\mathbf w_h, \theta_{n+1,h}, \kappa_{n+1,h})=(\mathbf y_{n,h}-\mathbf y_d ^n, \mathbf w_h),\nonumber\\
&\Big(\frac{\kappa_{n,h}-\kappa_{n+1,h}}{\tau_n}, \zeta_h\Big)+\chi a(\kappa_{n,h},\zeta_h)+b(\mathbf y_{n-1,h}, \zeta_h,\kappa_{n,h})+\eta\gamma(\kappa_{n,h}, \zeta_h)_\Gamma+\frac{\tau_{n+1}}{\tau_n}\beta(\bm \mu_{n+1,h}, \zeta_h \mathbf g)=(\theta_{n,h}-\theta_d^n, \zeta_h),\nonumber\\
&\mu_{N_{\tau}+1,h}:=0,\ \mathbf y_{N_{\tau}+1, h}:=\mathbf y_{N_{\tau}, h},\ \theta_{N_{\tau}+1}:=\theta_{N_\tau}, \quad\kappa_{N_{\tau}+1,h}:=0,
\end{align}
where $\big\{(\mathbf y_{n,h}, \theta_{n,h})\big\}_{n=1}^{N_\tau}$ is the solution of the fully discrete state equation \eqref{Full_discrete:state} and $\tau_{N_\tau+1}:=\tau_{N_\tau}$. 
It is easy to show that the equation \eqref{MuKappa:Full_discrete:adjoint} has a unique solution. Using \eqref{Full_discrete:Linzied:state} and \eqref{MuKappa:Full_discrete:adjoint}, we have
\begin{align*}
&\int_I (\mathbf{y}_\sigma(u) - \mathbf{y}_d, \mathbf{z}_\sigma) \, dt + \int_I (\theta_\sigma(u) - \theta_d, \xi_\sigma) \, dt=\sum_{n=1}^{N_\tau}\int_{t_{n-1}}^{t_n}\big[(\mathbf y_{n,h}-\mathbf y_d, \mathbf z_{n,h})+(\theta_{n,h}-\theta_d, \xi_{n,h})\big]dt\\
&=\sum_{n=1}^{N_\tau}\big(\bm \mu_{n,h}-\bm \mu_{n+1,h}, \mathbf z_{n,h} \big)+\sum_{n=1}^{N_\tau}\big[ \tau_n \nu \mathbf a(\bm \mu_{n,h}, \mathbf z_{n,h})+ \tau_{n+1}\mathbf b(\mathbf z_{n,h},\mathbf y_{n+1,h},\bm \mu_{n+1,h})+\tau_n\mathbf b(\mathbf y_{n-1,h},\mathbf z_{n,h}, \bm \mu_{n,h})\big]\\
&+\sum_{n=1}^{N_\tau}\big(\kappa_{n,h}-\kappa_{n+1,h}, \xi_{n,h})+\sum_{n=1}^{N_\tau}\big[\tau_n \chi a(\kappa_{n,h},\xi_{n,h})+\tau_{n+1} b(\mathbf z_{n,h}, \theta_{n+1,h}, \kappa_{n+1,h})+\tau_n b(\mathbf y_{n-1,h}, \xi_{n,h},\kappa_{n,h})\\
&+\eta\gamma\tau_n(\kappa_{n,h}, \xi_{n,h})_\Gamma+\tau_{n+1}\beta(\bm \mu_{n+1,h}, \xi_{n,h} \mathbf g)\big]\\
&=\sum_{n=1}^{N_\tau}\big(\mathbf z_{n,h}-\mathbf z_{n-1,h}, \bm  \mu_{n,h} \big)+\sum_{n=1}^{N_\tau}\tau_{n}\big[  \nu \mathbf a(\bm \mu_{n,h}, \mathbf z_{n,h})+ \mathbf b(\mathbf z_{n-1,h},\mathbf y_{n,h},\bm \mu_{n,h})+\mathbf b(\mathbf y_{n-1,h},\mathbf z_{n,h}, \bm \mu_{n,h})+\beta( \xi_{n-1,h} \mathbf g, \bm \mu_{n,h})\big]\\
&+\sum_{n=1}^{N_\tau}\big(\xi_{n,h}-\xi_{n-1,h}, \kappa_{n,h})+\sum_{n=1}^{N_\tau}\tau_n \big[\chi a(\kappa_{n,h},\xi_{n,h})+ b(\mathbf z_{n-1,h}, \theta_{n,h}, \kappa_{n,h})+ b(\mathbf y_{n-1,h}, \xi_{n,h},\kappa_{n,h})+\eta\gamma(\kappa_{n,h}, \xi_{n,h})_\Gamma\big]\\
&=\sum_{n=1}^{N_\tau}\tau_n \eta({v}^n, \kappa_{n,h})_\Gamma=\int_I\eta(\kappa_\sigma, v)_\Gamma dt,
\end{align*}
where $(\bm {\mu}_\sigma, \kappa_\sigma)\big|_{[0,T)}= \big( \sum\limits_{n=1}^{N_{\tau}} \bm \mu_{n,h} \mathds{1}_{I^l_n}, \sum\limits_{n=1}^{N_{\tau}} \kappa_{n,h} \mathds{1}_{I^l_n} \big)  \text{ with } \mu_\sigma(T) = 0, \, \kappa_\sigma(T) = 0$. Hence, from the obtained identity,  we conclude that
\begin{align}\label{discrete:J:diff1}
J'_\sigma(u)v = \int_I (\eta\kappa_\sigma+\alpha u, v)_\Gamma \, dt.
\end{align}

Similar to Theorem \ref{Theorem:fisrt_order_condtion}, the next theorem establishes the first-order necessary optimality condition satisfied for any local minimum of $(\mathrm{P_\sigma})$.
\begin{theorem} \label{Theorem:dicrete:fisrt_order_condtion}Let \( \bar{u}_\sigma \) be a local solution of problem $\mathrm{(P_\sigma)}$. Then there exist \( \big\{(\bar{\mathbf y}_{n,h},\bar{\theta}_{n,h})\big\}_{n=1}^{N_\tau}\) and \( \big\{(\bar{\bm \mu}_{n,h},\bar{\kappa}_{n,h})\big\}_{n=1}^{N_\tau}\),  such that
\begin{align}\label{Full_discrete:first_order:state}
&\Big(\frac{\bar{\mathbf y}_{n,h}-\bar{\mathbf y}_{n-1,h}}{\tau_n}, \mathbf v_h\Big)+\nu \mathbf a(\bar{\mathbf y}_{n,h}, \mathbf v_h)+\mathbf b(\bar{\mathbf y}_{n-1,h},\bar{\mathbf y}_{n,h},\mathbf v_h)+\beta(\bar{\theta}_{n-1,h}\mathbf g, \mathbf v_h )=({\mathbf h}^n, \mathbf v_h) \quad \forall \mathbf v_h\in \mathbb X_h,\nonumber\\
&\Big(\frac{\bar{\theta}_{n,h}-\bar{\theta}_{n-1,h}}{\tau_n}, \psi_h\Big)+\chi a(\bar{\theta}_{n,h},\psi_h)+b(\bar{\mathbf y}_{n-1,h}, \bar{\theta}_{n,h},\psi_{h})+\eta\gamma(\bar{\theta}_{n,h}, \psi_h)_\Gamma=({f}^n, \psi_h)+\eta({\bar{u}_\sigma}^n, \psi_h)_\Gamma\quad \forall\psi_h\in V_h,\nonumber\\
&\bar{\mathbf y}_{0,h}=\mathbf y_{0h},\quad \bar{\theta}_{0,h}=\theta_{0h}, 
\end{align}
\begin{align}\label{MuKappa:fisrt_order:Full_discrete:adjoint}
&\Big(\frac{\bar{\bm \mu}_{n,h}-\bar{\bm \mu}_{n+1,h}}{\tau_n}, \mathbf w_h\Big)+\nu \mathbf a(\bar{\bm \mu}_{n,h}, \mathbf w_h)+\frac{\tau_{n+1}}{\tau_n}\mathbf b(\mathbf w_h,\bar{\mathbf y}_{n+1,h},\bar{\bm \mu}_{n+1,h})+\mathbf b(\bar{\mathbf y}_{n-1,h},\mathbf w_h, \bar{\bm \mu}_{n,h})\nonumber\\
&+\frac{\tau_{n+1}}{\tau_n}b(\mathbf w_h, \bar{\theta}_{n+1,h}, \bar{\kappa}_{n+1,h})=(\bar{\mathbf y}_{n,h}-\mathbf y_d ^n, \mathbf w_h)\quad \forall \mathbf w_h\in \mathbb X_h,\nonumber\\
&\Big(\frac{\bar{\kappa}_{n,h}-\bar{\kappa}_{n+1,h}}{\tau_n}, \zeta_h\Big)+\chi a(\bar{\kappa}_{n,h},\zeta_h)+b(\bar{\mathbf y}_{n-1,h}, \zeta_h,\bar{\kappa}_{n,h})+\eta\gamma(\bar{\kappa}_{n,h}, \zeta_h)_\Gamma+\frac{\tau_{n+1}}{\tau_n}\beta(\bar{\bm \mu}_{n+1,h}, \zeta_h \mathbf g)\nonumber\\
&=(\bar{\theta}_{n,h}-\theta_d^n, \zeta_h) \quad \forall\zeta_h\in V_h,\nonumber\\
&\bar{\mu}_{N_{\tau}+1,h}:=0,\ \bar{\mathbf y}_{N_{\tau}+1, h}:=\bar{\mathbf y}_{N_{\tau}, h},\ \bar{\theta}_{N_{\tau}+1}:=\bar{\theta}_{N_\tau}, \quad\bar{\kappa}_{N_{\tau}+1,h}:=0, 
\end{align}
\begin{align}
\int_I (\eta\bar{\kappa}_\sigma+\alpha \bar{u}_\sigma, v-\bar{u}_\sigma)_\Gamma \geq 0\quad\forall v\in\mathcal{U}_{ad}.\label{first_cond:discrete:varition_ineq}
\end{align}
\end{theorem}

\begin{remark} We simply observe that \eqref{first_cond:discrete:varition_ineq} implies that
\begin{align}\label{Pro:discrete:adj}
\bar{u}_\sigma(t,x)=\mathrm{Proj}_{[u_a,u_b]}\big(-\frac{\eta}{\alpha}\bar{\kappa}_\sigma(t,x)\big)\quad \text{for a.e. } (t, x)\in \Sigma_T.
\end{align}
\end{remark}

Below, we present the stability analysis of the discrete adjoint equation \eqref{MuKappa:Full_discrete:adjoint}.
\begin{lemma}\label{lem:stab:Full:adjoint:bu} Given that \( u \in L^2(I; L^2(\Gamma)) \), let \(\big\{(\mathbf{y}_{n,h}, \theta_{n,h})\big\}_{n=1}^{N_\tau}\) be the solution of the fully discrete state equation \eqref{Full_discrete:state}, while \(\big\{(\bm \mu_{n,h}, \kappa_{n,h})\big\}_{n=1}^{N_\tau}\) represents the solution of the fully discrete adjoint equation \eqref{MuKappa:Full_discrete:adjoint}. Then for $\sigma$ small enough, the following estimates hold:
\begin{align}
&\|\bm \mu_{\sigma}\|_{L^\infty(I;\mathbb L^2(\Omega))}+\|\kappa_{\sigma}\|_{L^\infty(I;L^2(\Omega))}+\nu^\frac{1}{2}\|\bm \mu_\sigma\|_{L^2(I;\mathbb H^1(\Omega))}+\min\{\chi, \eta\gamma\}^\frac{1}{2}\|\kappa_{\sigma}\|_{L^2(I;H^1(\Omega))}\leq \mathcal C_3\big(\nu^{-\frac{1}{2}}\|\mathbf y_d\|_{L^2(I;L^2(\Omega))}\nonumber\\
&+\min\{\chi,\eta\gamma\}^{-\frac{1}{2}}\|\theta_d\|_{L^2(I;L^2(\Omega))}+\widetilde{G}_1\big)=:\widetilde G_3,\label{Full:stab:adjoint:ineq1}\\
&\min\{\chi,\eta\gamma\}^{\frac{1}{2}}\Big(\| \kappa_\sigma\|_{L^\infty(I;H^1(\Omega))}+\Big(\sum_{n=1}^{N_\tau}\|\kappa_{n,h}-\kappa_{n+1,h}\|^2_{H^1(\Omega)}\Big)^\frac{1}{2}\Big)+\|\Delta_h\kappa_\sigma\|_{L^2(I;L^2(\Omega))}\nonumber\\
&\leq \mathcal C_4\Big((\beta|\mathbf g|+1)\widetilde G_3+\nu^{-\frac{1}{2}}\widetilde G_2^2\Big)=:\widetilde G_4,\label{Full:stab:adjoint:ineq2}\\
&\Big(\sum_{n=1}^{N_\tau}\frac{\|\kappa_{n,h}-\kappa_{n+1,h}\|^2}{\tau_n}\Big)^\frac{1}{2}\leq C\Big(\widetilde G_2 \widetilde G_4\min\{\chi,\eta\gamma\}^{-\frac{1}{4}}+\widetilde G_2^2\widetilde G_3\min\{\chi,\eta\gamma\}^{-\frac{1}{2}}+\widetilde G_4\Big),\label{Full:stab:adjoint:ineq3}
\end{align}
where
\begin{align*}
&\mathcal C_3:=C\exp{\Big(C\big((\nu^{-\frac{1}{2}}+\min\{\chi,\eta\gamma\}^{-\frac{1}{2}})\widetilde{G}_1+\min\{\chi,\eta\gamma\}^{-\frac{1}{2}}\beta|\mathbf g|\sqrt{T}+\nu^{-\frac{3}{2}}\widetilde{G}_1\widetilde{G}_2+\nu^{-1}\min\{\chi,\eta\gamma\}^{-\frac{1}{2}}\widetilde{G}_2^2\big)\Big)},\\
&\mathcal C_4:=C\exp{\Big(C\big(\nu^{-\frac{1}{2}}\widetilde G^2_1+\widetilde G_3\big)\Big)}.
\end{align*}
\end{lemma}
\begin{proof}We prove the estimate one by one. 

$\bullet$ Estimate of \eqref{Full:stab:adjoint:ineq1}.  Choosing $(\mathbf{w}_h,\zeta_h) =\tau_n (\bm \mu_{n,h}, \kappa_{n,h})$ in \eqref{MuKappa:Full_discrete:adjoint}, we deduce
\begin{align}
&\frac{1}{2}\big(\|\bm \mu_{n,h}\|^2+\|\kappa_{n,h}\|^2\big)-\frac{1}{2}\big(\|\bm \mu_{n+1,h}\|^2+\|\kappa_{n+1,h}\|^2\big)+\nu\tau_n\|\nabla \bm \mu_{n,h}\|^2+\chi\tau_n\|\nabla \kappa_{n,h}\|^2+\eta\gamma\tau_n\|\kappa_{n,h}\|^2_\Gamma\nonumber\\
&\leq \Big[\int_{t_{n-1}}^{t_n}\big[(\mathbf y_\sigma-\mathbf y_d, \bm \mu_{n,h})+(\theta_\sigma-{\theta_d}, \kappa_{n,h})\big]dt-\tau_{n+1} \beta(\kappa_{n,h}\mathbf g, \bm \mu_{n+1,h})\Big]-\tau_{n+1}\mathbf b(\bm \mu_{n,h}, \mathbf y_{n+1, h}, \bm \mu_{n+1,h})\nonumber\\
&-\tau_{n+1}\mathbf b(\bm \mu_{n,h},\theta_{n+1,h}, \kappa_{n+1,h})=J_1+J_2+J_3.\label{stab:Full_discrete:adjoint:bu:1}
\end{align}
By the Cauchy-Schwarz and Young's inequalities,  the  term $J_1$ satisfies
\begin{align*}
J_1\leq  \frac{\tau_n\nu}{6}\|\nabla \bm \mu_{n,h}\|^2&+\frac{\tau_n}{2}\big(\chi\|\nabla \kappa_{n,h}\|^2+\eta\gamma\|\kappa_{n,h}\|^2_\Gamma\big)+C\Big(\nu^{-1}\int_{t_{n-1}}^{t_n}\|\mathbf y_\sigma-\mathbf y_d\|^2 dt\\
&+\min\{\chi, \eta\gamma\}^{-1}\int_{t_{n-1}}^{t_n}\|\theta_\sigma-\theta_d\|^2 dt+F_n\|\bm \mu_{n+1,h}\|^2\Big)
\end{align*}
with $F_n=\min\{\chi,\eta\gamma\}^{-1}\beta^2|\mathbf g|^2\tau_n$. Using \eqref{Interpolation:conti:1}, we deduce
\begin{align*}
J_2&\leq C \tau_{n+1} \|\bm \mu_{n,h}\|^\frac{1}{2}\|\nabla\bm \mu_{n,h}\|^\frac{1}{2}\,\|\nabla \mathbf y_{n+1,h}\|\,\|\bm \mu_{n+1,h}\|^\frac{1}{2}\|\nabla \bm \mu_{n+1,h}\|^\frac{1}{2}\\
&+C\tau_{n+1} \|\nabla \bm \mu_{n,h}\|\,\|\bm \mu_{n+1,h}\|^\frac{1}{2}\|\nabla \bm \mu_{n+1,h}\|^\frac{1}{2}\|\nabla \mathbf y_{n+1,h}\|\\
&\leq \frac{\nu\tau_n}{6}\|\nabla \bm \mu_{n,h}\|^2+\frac{\nu\tau_{n+1}}{4}\|\nabla \bm \mu_{n+1,h}\|^2+C\big(G_n\|\bm \mu_{n,h}\|^2+H_n\|\bm \mu_{n+1,h}\|^2\big)
\end{align*}
with $G_n=\nu^{-1}\tau_{n+1}\|\nabla \mathbf y_{n+1,h}\|^2$, $H_n=\tau_{n+1}\big(\nu^{-1}\|\nabla \mathbf y_{n+1,h}\|^2+\nu^{-3}\|\nabla \mathbf y_{n+1,h}\|^4\big)$. Using a similar argument, we have the estimate for $J_3$
\begin{align*}
J_3&\leq \frac{\nu\tau_n}{6}\|\nabla\bm \mu_{n,h}\|^2+\frac{\tau_{n+1}}{4}\big(\chi\|\nabla \kappa_{n+1,h}\|^2+\eta\gamma\|\kappa_{n+1,h}\|^2_\Gamma\big)+C\big(K_n\| \bm \mu_{n,h} \|^2+M_n\|\kappa_{n+1,h}\|^2\big)
\end{align*}
with $K_n=\tau_{n+1}\nu^{-1}\|\nabla\theta_{n+1,h}\|^2$, $M_n=\tau_{n+1}\min\{\chi,\eta\gamma\}^{-1}\|\nabla\theta_{n+1,h}\|^2+\tau_{n+1}\nu^{-2}\min\{\chi,\eta\gamma\}^{-1}\|\theta_{n+1,h}\|^2\|\theta_{n+1,h}\|^2_{H^1(\Omega)}$. Hence, using the equality $\frac{1}{2}a^2=\frac{1}{4}a^2+\frac{1}{4}a^2-\frac{1}{4}b^2+\frac{1}{4}b^2$,  and combining the above estimates with \eqref{stab:Full_discrete:adjoint:bu:1} and summing over $n = N_\tau, \dots, k$ for any $1 \leq k \leq N_\tau$, we obtain the following bound
\begin{align*}
&\|\bm \mu_{n,h}\|^2+\|\kappa_{n,h}\|^2+\sum_{n=k}^{N_\tau}\tau_n\big[\nu\|\nabla \bm\mu_{n,h}\|^2+\min\{\chi,\eta\gamma\}\big(\|\nabla \kappa_{n,h}\|^2+\|\kappa_{n,h}\|^2_{\Gamma}\big)\big]\\
&\leq C\Big(\sum_{n=k}^{N_\tau}\Big((G_n+K_n)\|\bm \mu_{n,h}\|^2+(F_n+H_n+M_n)\big(\|\bm \mu_{n+1,h}\|^2+\|\kappa_{n+1,h}\|^2\big)\Big)+\int_{t_{k-1}}^T\nu^{-1}\|\mathbf y_\sigma-\mathbf y_d\|^2dt\\
&+\min\{\chi,\eta\gamma\}^{-1}\int_{t_{k-1}}^T\|\theta_\sigma-\mathbf \theta_d\|^2dt\Big).
\end{align*}
Similar to Remark \ref{remark:conv:Theta}, Lemma \ref{appB:Lem:full:state:err} implies that $\tau_n\|\theta_{n,h}(u)\|^2_{H^1(\Omega)} \to 0$ uniformly in $n$ as $\sigma \to 0$. 
Hence, together with Lemma \ref{lem:Full:YTheta:bu}, bounding $\|\mathbf y_\sigma\|_{L^\infty(I;\mathbb H^1(\Omega))}$,  we can choose the temporal and space mesh sizes fine enough so that $C(G_n+K_n)\leq \frac{1}{2}$ holds for all $n$. Together with Lemma \ref{lem:Full:YTheta:bu} and the Gronwall inequality, the proof follows.

The estimates for \eqref{Full:stab:adjoint:ineq2} and \eqref{Full:stab:adjoint:ineq3} are analogous to those for \eqref{Full:stab:state:ineq2} and \eqref{Full:stab:state:ineq3}.
\end{proof}

The next theorem states the temporal error estimates for the approximation of the adjoint equation.  The proof is omitted because it is analogous to that of Lemma \ref{lem:Full:YTheta:time:err}.
\begin{lemma}\label{lem:err:Full:adjoint}Let \((\bm  \mu_h,\kappa_h)\in L^2({I};\mathbb X_h)\times L^2({I};V_h)\) be the solution of \eqref{semi_discrete:adjoint} and $\big\{(\bm \mu_{n,h}, \kappa_{n,h})\big\}_{n=1}^{N_\tau}$  be the solution of  the discrete adjoint equation \eqref{MuKappa:Full_discrete:adjoint}. Then, under the assumptions of Lemma \ref{lem:Full:YTheta:time:err}  the following estimates hold:
\begin{equation*}
\begin{aligned}
\|\bm \mu_h-\bm \mu_\sigma\|_{L^\infty(I;\mathbb L^2(\Omega))}+\|\bm \mu_h- \bm \mu_\sigma\|_{L^2(I;\mathbb H^1(\Omega))}
+\|\kappa_h-\kappa_\sigma\|_{L^\infty(I; L^2(\Omega))}+\|\kappa_h-\kappa_\sigma\|_{L^2(I; H^1(\Omega))}\leq C \tau^\frac{1}{2}.
\end{aligned}
\end{equation*}
\end{lemma}

We now present the continuous dependence of the solution \( (\bm \mu_\sigma, \kappa_\sigma) \) on the control $u$.
\begin{lemma}\label{Full:lem:Lip:MuKappa}
Given that  \( \max\big\{ \|u\|_{H^{\frac{1}{4}}(I; L^2(\Gamma))}, \|u\|_{L^2(I; H^{\frac{1}{2}}(\Gamma))}, \|v\|_{L^2(I; L^2(\Gamma))} \big\} \leq M \), where \(M\) is a positive constant. Then for a sufficiently small $\sigma$,  it follows that
\begin{align*}
\|\bm \mu_\sigma(u) - \bm \mu_\sigma(v)\|_{L^\infty(I; \mathbb{L}^2(\Omega))} &+ \|\kappa_\sigma(u) - \kappa_\sigma(v)\|_{L^\infty(I; L^2(\Omega))}+\|\bm \mu_\sigma(u) - \bm \mu_\sigma(v)\|_{L^2(I;\mathbb H^1(\Omega))}\\
&+\|\kappa_\sigma(u) - \kappa_\sigma(v)\|_{L^2(I; H^1(\Omega))}\leq C_{M} \|u - v\|_{L^2(I; L^2(\Gamma))},
\end{align*}
where \(  C_{ M}\) is a positive constant depending on \( M \).
\end{lemma}
\begin{proof}
Let $\big( \mathbf{e}^{\bm{\mu}}_{n,h}, e^{\kappa}_{n,h} \big) = \big( \bm{\mu}_{n,h}(u) - \bm{\mu}_{n,h}(v), \kappa_{n,h}(u) - \kappa_{n,h}(v) \big)$. From \eqref{MuKappa:Full_discrete:adjoint} it follows that
\begin{equation*}
\begin{aligned}
& \big(\mathbf e^{\bm \mu}_{n,h}-\mathbf e^{\bm  \mu}_{n+1,h}, \mathbf w_h\big)+\tau_n \nu\mathbf a(\mathbf e^{\bm \mu}_{n,h}, \mathbf w_h)+\tau_{n+1}\mathbf b(\mathbf w_h, \mathbf e^{\mathbf y}_{n+1,h}, \bm \mu_{n+1,h}(u))
+\tau_{n+1}\mathbf b(\mathbf w_h, \mathbf y_{n+1,h}(v), \mathbf e^{\bm \mu}_{n+1,h})\\
&+\tau_n\mathbf b(\mathbf e^{\mathbf y}_{n-1,h},\mathbf w_h, \bm \mu_{n,h}(u))+\tau_n\mathbf b(\mathbf y_{n-1,h}(v),\mathbf w_h, \mathbf e^{\bm \mu}_{n,h})+\tau_{n+1}b(\mathbf w_h,\theta_{n+1,h}(u), e^\kappa_{n+1,h})\\
&+\tau_{n+1}b(\mathbf w_h, e^\theta_{n+1,h}, \kappa_{n+1,h}(v))=\tau_n(\mathbf e_{n,h}^{\mathbf y}, \mathbf w_h),\\
&(e^{\kappa}_{n,h}-e^{\kappa}_{n+1,h}, \zeta_h)+\chi\tau_n a(e^{\kappa}_{n,h}, \zeta_h)+\tau_n b(\mathbf e^{\mathbf y}_{n-1,h}, \zeta_h, \kappa_{n,h}(u))+\tau_n b(\mathbf y_{n-1,h}(v), \zeta_h, e^\kappa_{n,h})\\
&+\eta\gamma\tau_n(e^{\kappa}_{n,h}, \zeta_h)_\Gamma+\tau_{n+1}\beta(\mathbf e^{\bm \mu}_{n+1,h}, \zeta_h\mathbf g)= \tau_n (e^\theta_{n,h}, \zeta_h),\\
&\mathbf e^{\bm \mu}_{N_\tau+1,h}=0,\quad  e_{N_\tau+1,h}^\kappa=0,
\end{aligned}
\end{equation*}
where $\big( \mathbf{e}^{\mathbf{y}}_{n,h}, e^{\theta}_{n,h} \big) := \big( \mathbf{y}_{n,h}(u) - \mathbf{y}_{n,h}(v), \theta_{n,h}(u) - \theta_{n,h}(v) \big)$. Taking $(\mathbf w_h,\zeta_h)=(\mathbf e_{n,h}^{\bm \mu}, e_{n,h}^\kappa)$, we obtain
\begin{align}\label{Full:adjoint:Lip}
&\frac{1}{2}\big(\|\mathbf e^{\bm \mu}_{n,h}\|^2+\|e^\kappa_{n,h}\|^2\big)-\frac{1}{2}\big(\|\mathbf e^{\bm \mu}_{n+1,h}\|^2+\|e^\kappa_{n+1,h}\|^2\big)+\tau_n\nu\|\nabla \mathbf e^{\bm \mu}_{n,h}\|^2+\tau_n\chi\|\nabla e^\kappa_{n,h}\|^2 +\tau_n\eta\gamma\|e^\kappa_{n,h}\|^2_\Gamma\nonumber\\
&\leq -\big[\tau_{n+1}\mathbf b(\mathbf e^{\bm \mu}_{n,h}, \mathbf e^{\mathbf y }_{n+1,h}, \bm \mu_{n+1,h}(u))+\tau_{n+1}\mathbf b(\mathbf e^{\bm \mu}_{n,h}, \mathbf y_{n+1,h}(v), \mathbf e^{\bm \mu}_{n+1,h})+\tau_n\mathbf b(\mathbf e^{\mathbf y}_{n-1,h}, \mathbf e^{\bm \mu}_{n,h}, \bm \mu_{n,h}(u))\big]\nonumber\\
&-\big[\tau_{n+1} b(\mathbf e^{\bm \mu}_{n,h}, \theta_{n+1,h}(u), e^\kappa_{n+1,h})+\tau_{n+1}b(\mathbf e^{\bm \mu}_{n,h},e^\theta_{n+1,h}, \kappa_{n+1}(v))+\tau_nb(\mathbf e^{\mathbf y}_{n-1,h}, e^\kappa_{n,h}, \kappa_{n,h}(u))\big]\nonumber\\
&-\big[\tau_{n+1}\beta(\mathbf e^{\bm \mu}_{n+1,h}, e^\kappa_{n,h}\mathbf g)
-\tau_n(\mathbf e^{\mathbf y}_{n,h}, \mathbf e^{\bm \mu}_{n,h})-\tau_n(e^\theta_{n,h}, e^\kappa_{n,h})\big]=J_1+J_2+J_3
\end{align}
with $ \mathbf{e}^{\mathbf{y}}_{N_\tau+1,h}: = \mathbf{0} $ and $ e^\theta_{N_\tau+1,h} := 0 $. Based on Lemmas \ref{lem:Full:YTheta:bu} and \ref{lem:stab:Full:adjoint:bu}, \eqref{Full:adjoint:Lip} can be estimated as follows:
\begin{align*}
J_1&\leq C \tau_{n+1}\|\mathbf e^{\bm \mu}_{n,h}\|_{\mathbb L^4(\Omega)}\|\mathbf e^{\mathbf y}_{n+1,h}\|_{\mathbb L^4(\Omega)}\|\nabla \bm \mu_{n+1,h}(u)\|+ C\tau_{n+1} \|\mathbf e^{\bm \mu}_{n,h}\|_{\mathbb L^4(\Omega)}\|\bm \mu_{n+1,h}(u)\|_{\mathbb L^4(\Omega)}\|\nabla \mathbf e^{\mathbf y}_{n+1,h}\|\\
&+C\tau_{n+1}\|\mathbf e^{\bm \mu}_{n,h}\|_{\mathbb L^4(\Omega)}\|\nabla \mathbf y_{n+1,h}(v)\|\|\mathbf e^{\bm \mu}_{n+1,h}\|_{\mathbb L^4(\Omega)}+C\tau_{n+1} \|\mathbf e^{\bm \mu}_{n+1,h}\|_{\mathbb L^4(\Omega)}\|\mathbf  y_{n+1,h}(v)\|_{\mathbb L^4(\Omega)}\|\nabla \mathbf e^{\bm \mu}_{n,h}\|\\
&+ C \tau_{n}\|\mathbf e^{\bm \mu}_{n,h}\|_{\mathbb L^4(\Omega)}\|\mathbf e^{\mathbf y}_{n-1,h}\|_{\mathbb L^4(\Omega)}\|\nabla \bm \mu_{n,h}(u)\|+ C \tau_{n}\|\mathbf e^{\bm \mu}_{n,h}\|_{\mathbb L^4(\Omega)}\|\bm \mu_{n,h}(u)\|_{\mathbb L^4(\Omega)}\|\nabla \mathbf e^{\mathbf y}_{n-1,h}\|\\
&\leq \frac{\nu}{6}\tau_{n}\|\nabla \mathbf e^{\bm \mu}_{n,h} \|^2+ \frac{\nu}{4}\tau_{n+1}\|\nabla \mathbf e^{\bm \mu}_{n+1,h}\|^2+C\Big(G_n\|\mathbf e^{\bm \mu}_{n,h}\|^2+H_n\|\mathbf e^{\bm \mu}_{n+1,h}\|^2+ \tau_{n+1}\|\nabla \bm \mu_{n+1,h}(u)\|\|\mathbf e^{\mathbf y}_{n+1,h}\|\|\nabla \mathbf e^{\mathbf y}_{n+1,h}\|\\
&+\tau_{n+1}\|\nabla \mathbf e^{\mathbf y}_{n+1,h}\|^2+\tau_{n}\|\nabla \bm \mu_{n,h}(u)\|\|\mathbf e^{\mathbf y}_{n-1,h}\|\|\nabla \mathbf e^{\mathbf y}_{n-1,h}\|+\tau_n\|\nabla \mathbf e^{\mathbf y}_{n-1,h}\|^2\Big)
\end{align*}
with $G_n=\tau_{n+1}\|\nabla \bm \mu_{n+1,h}(u)\|^2+\tau_{n+1}\|\bm \mu_{n+1,h}(u)\|^2\|\nabla \bm \mu_{n+1,h}(u) \|^2+\tau_{n}\|\bm \mu_{n,h}(u)\|^2\|\nabla \bm \mu_{n,h}(u) \|^2+\tau_n\|\nabla \mathbf y_{n+1,h}(v)\|^2$, $H_n=\tau_n \|\nabla \mathbf y_{n+1,h}(v)\|^2+\tau_n\|\mathbf y_{n+1,h}(v)\|^2\|\nabla \mathbf y_{n+1,h}(v)\|^2$. Furthermore, we have
\begin{align*}
J_2&\leq C\tau_{n+1}\|\mathbf e^{\bm \mu}_{n,h}\|_{\mathbb L^4(\Omega)}\|\nabla  \theta_{n+1,h}(u)\|\|\ e^{\kappa}_{n+1,h}\|_{L^4(\Omega)}+C\tau_{n+1} \| e^{\kappa}_{n+1,h}\|_{L^4(\Omega)}\|\theta_{n+1,h}(u)\|_{L^4(\Omega)}\|\nabla \mathbf e^{\bm \mu}_{n,h}\|\\
&+C\tau_{n+1}\|\mathbf e^{\bm \mu}_{n,h}\|_{\mathbb L^4(\Omega)}\|e^{\theta}_{n+1,h}\|_{L^4(\Omega)}\|\nabla \kappa_{n+1,h}(v)\|+ C\tau_{n+1} \|\mathbf e^{\bm \mu}_{n,h}\|_{\mathbb L^4(\Omega)}\|\kappa_{n+1,h}(v)\|_{L^4(\Omega)}\|\nabla  e^{\theta}_{n+1,h}\|\\
&+ C \tau_{n}\| e^{\kappa}_{n,h}\|_{L^4(\Omega)}\|\mathbf e^{\mathbf y}_{n-1,h}\|_{\mathbb L^4(\Omega)}\|\nabla \kappa_{n,h}(u)\|+ C \tau_{n}\| e^{\kappa}_{n,h}\|_{L^4(\Omega)}\|\kappa_{n,h}(u)\|_{\mathbb L^4(\Omega)}\|\nabla \mathbf e^{\mathbf y}_{n-1,h}\|\\
&\leq \frac{\nu}{6}\tau_{n}\|\nabla \mathbf e^{\bm \mu}_{n,h} \|^2+\frac{1}{4}\big(\tau_n\chi\|\nabla e^\kappa_{n,h}\|^2 +\tau_n\eta\gamma\|e^\kappa_{n,h}\|^2_\Gamma\big)+\frac{1}{4}\big(\tau_{n+1}\chi\|\nabla e^\kappa_{n+1,h}\|^2 +\tau_{n+1}\eta\gamma\|e^\kappa_{n+1,h}\|^2_\Gamma\big)\\
&+C\Big(I_n\|\mathbf e^{\bm \mu}_{n,h}\|^2+K_n\| e^{ \kappa}_{n+1,h}\|^2+\tau_{n+1}\|e^\theta_{n+1,h}\|^2_{H^1(\Omega)}+\tau_n\|\nabla \mathbf e^{\mathbf y}_{n-1,h}\|^2\Big)
\end{align*}
with $I_n=\tau_{n+1}\|\nabla \theta_{n+1,h}(u)\|^2$, $K_n=\tau_{n+1}\|\nabla \theta_{n+1,h}(u)\|^2+\tau_{n+1}\|\theta_{n+1,h}(u)\|^2\|\theta_{n+1,h}(u)\|^2_{H^1(\Omega)}$ and
\begin{align*}
J_3\leq \frac{\nu}{6}\tau_{n}\|\nabla \mathbf e^{\bm \mu}_{n,h} \|^2+\frac{1}{4}\big(\tau_n\chi\|\nabla e^\kappa_{n,h}\|^2 +\tau_n\eta\gamma\|e^\kappa_{n,h}\|^2_\Gamma\big)+C\Big(\tau_{n+1}\beta^2|\mathbf g|^2\|\mathbf e^{\bm \mu}_{n+1,h}\|^2+\tau_n\|\mathbf e_{n,h}^{\mathbf y}\|^2+\tau_n\|e^\theta_{n,h}\|^2\Big),
\end{align*}
where we used \eqref{Full:stab:adjoint:ineq2}. Thus, combining the above estimates with \eqref{Full:adjoint:Lip} and summing over $n = N_\tau, \dots, k$ for any $1 \leq k \leq N_\tau$, we obtain the following bound
\begin{align*}
&\|\mathbf e^{\bm \mu}_{k,h}\|^2+\|e^\kappa_{k,h}\|^2+\sum_{n=k}^{N_\tau}\big[\tau_n\nu\|\nabla \mathbf e^{\bm \mu}_{n,h}\|^2+\tau_n\chi\|\nabla e^\kappa_{n,h}\|^2 +\tau_n\eta\gamma\|e^\kappa_{n,h}\|^2_\Gamma\big]\leq C\Big(\sum_{n=k}^{N_\tau}(G_n+I_n)\|\mathbf e^{\bm \mu}_{n,h}\|^2\\
&+\sum_{n=k}^{N_\tau}(H_n+\tau_{n+1}\beta^2|\mathbf g|^2)\|\mathbf e^{\bm \mu}_{n+1,h}\|^2+\sum_{n=k}^{N_\tau} K_n\| e^{ \kappa}_{n+1,h}\|^2+\sum_{n=k}^{N_\tau}\tau_{n}\|\nabla \bm \mu_{n,h}(u)\|\|\mathbf e^{\mathbf y}_{n-1,h}\|\|\nabla \mathbf e^{\mathbf y}_{n-1,h}\|\\
&\sum_{n=k}^{N_\tau}\tau_{n+1}\|\nabla \bm \mu_{n+1,h}(u)\|\|\mathbf e^{\mathbf y}_{n+1,h}\|\|\nabla \mathbf e^{\mathbf y}_{n+1,h}\|+\sum_{n=k}^{N_\tau}\tau_{n}\big[\|\nabla \mathbf e^{\mathbf y}_{n-1,h}\|^2+\|e^\theta_{n,h}\|^2_{H^1(\Omega)}\big]\Big).
\end{align*}
Using Lemmas \ref{lem:Full:YTheta:time:err}, \ref{lem:Full:YTheta:bu} and 
\ref{lem:err:Full:adjoint}, we can choose the temporal and spatial  mesh sizes fine enough so that $C (G_n+I_n) \leq \frac{1}{2}$ holds for all $n$. Then Lemma \ref{Full:lem:Lip:YTheta} and Gronwall's inequality imply the proof.
\end{proof}

We can deduce the improved temporal convergence order for the fully discrete adjoint variable under the norm $L^2(I;\mathbb L^2(\Omega))$ or $L^2(I;L^2(\Omega))$.
\begin{lemma}\label{Full_err:Mukappa:time:L2} Under the assumptions of Lemma \ref{lem:err:Full:adjoint} , the following estimate holds:
\begin{equation}\label{Full:MukappL2:time:err}
\begin{aligned}
\|\bm  \mu_h-\bm \mu_\sigma\|_{L^2(I;\mathbb L^2(\Omega))}+\|\kappa_h-\kappa_\sigma\|_{L^2(I; L^2(\Omega))}\leq C \tau.
\end{aligned}
\end{equation}
\end{lemma}
\begin{proof}
We use the  decomposition 
\begin{equation*}
\begin{aligned}
 &\bm \mu_h-\bm \mu_\sigma=\bm \mu_h -\Pi^l_\tau \bm \mu_h+\Pi^l_\tau \bm \mu_h-\bm \mu_\sigma=\bm \zeta^{\bm \mu}_\tau+\bm \eta^{\bm \mu}_\tau, \quad \kappa_h-\kappa_{\sigma}=\kappa_h- \Pi^l_\tau \kappa_h +\Pi^l_\tau\kappa_h- \kappa_\sigma=\zeta^{\kappa}_\tau+\eta^{\kappa}_\tau.
 \end{aligned}
 \end{equation*}
We can derive from the error estimate of the interpolation $\Pi^l_\tau$ given by Lemma \ref{int:L2:time:err} that
\begin{align*}
\|\bm \mu_h-\bm \mu_\sigma\|_{L^2(I;\mathbb L^2(\Omega))}+\|\kappa_h-\kappa_\sigma\|_{L^2(I;L^2(\Omega))}\leq C \tau+\|\bm \eta^{\bm  \mu}_\tau\|_{L^2(I;\mathbb L^2(\Omega))}+\|\eta^\kappa_\tau\|_{L^2(I; L^2(\Omega))}.
\end{align*}
We now use a duality argument to provide an estimate for $\|\bm \eta^{\bm \mu}_\tau\|_{L^2(I;\mathbb L^2(\Omega))} + \|\eta^\kappa_\tau\|_{L^2(I; L^2(\Omega))}$.  We begin by defining a discrete dual problem: seek a discrete solution pair $(\mathbf z_{n,h}, \xi_{n,h})\in \mathbb X_h\times V_h$, satisfying, for $n=1, \cdots, N_{\tau}$ and any $(\mathbf v_h,\psi_h)\in \mathbb X_h\times V_h$,
\begin{align}\label{dual:Ytheta:Full_discrete:adjoint}
&\Big(\frac{\mathbf z_{n,h}-\mathbf z_{n-1,h}}{\tau_n}, \mathbf v_h\Big)+\nu \mathbf a(\mathbf z_{n,h}, \mathbf v_h)+\mathbf b(\mathbf z_{n-1,h},\mathbf y_{n,h},\mathbf v_h)+\mathbf b(\mathbf y_{n-1,h},\mathbf z_{n,h},\mathbf v_h)+\beta(\xi_{n-1,h}\mathbf g, \mathbf v_h )\nonumber\\
&=(\big(\bm \eta^{\bm \mu}_\tau\big)^n, \mathbf v_h),\nonumber\\
&\Big(\frac{\xi_{n,h}-\xi_{n-1,h}}{\tau_n}, \psi_h\Big)+\chi a(\xi_{n,h},\psi_h)+b(\mathbf z_{n-1,h}, \theta_h(t_n),\psi_{h})+b(\mathbf y_{n-1,h}, \xi_{n,h},\psi_{h})+\eta\gamma(\xi_{n,h}, \psi_h)_\Gamma=(\big(\eta^\kappa_\tau\big)^n, \psi_h),\nonumber\\
&\mathbf z_{0,h}=0,\quad \xi_{0,h}=0.
\end{align}
Similar to the proof of Lemma \ref{lem:Full:YTheta:bu}, we know that \eqref{dual:Ytheta:Full_discrete:adjoint} has a unique solution $\big\{(\mathbf z_{n,h},\xi_{n,h})\big\}_{n=1}^{N_\tau}$ and the following stability estimate holds:
 \begin{align}\label{Full_err:MuKappa:time:L2:bu:3}
\max_{1\leq n\leq N_\tau}\Big(\|\xi_{n,h}\|^2_{H^1(\Omega)}+\|\nabla \mathbf z_{n,h}\|^2\Big)+&\sum_{n=1}^{N_\tau}\tau_n\Big(\|\nabla \mathbf z_{n,h}\|^2+\|\xi_{n,h}\|^2_{H^1(\Omega)}\Big)+\sum_{n=1}^{N_\tau}\tau_n\Big(\|\mathbf A_h \mathbf z_{n,h}\|^2+\|\Delta_h \xi_{n,h}\|^2\Big)\nonumber\\
&\leq C\Big(\|\bm \eta_\tau^{\bm \mu}\|^2_{L^2(I;\mathbb L^2(\Omega))}+\|\eta_\tau^\kappa\|^2_{L^2(I;L^2(\Omega))}\Big).
\end{align}
Setting $(\mathbf v_h,\psi_h)=\tau_n(\bm \eta^{\bm \mu}_{n,\tau},\eta_{n,\tau}^{\kappa})$ in \eqref{dual:Ytheta:Full_discrete:adjoint},  and summing  over $1$ to $N_\tau$, we obtain
\begin{align}\label{Full_err:MUkappa:time:L2:1}
&\|\bm \eta^{\bm  \mu}_\tau\|^2_{L^2(I;\mathbb L^2(\Omega))}+\| \eta^\kappa_\tau\|^2_{L^2(I;L^2(\Omega))}=\sum_{n=1}^{N_\tau}\big({\mathbf z_{n,h}-\mathbf z_{n-1,h}}, \bm \eta^{\bm \mu}_{n,\tau}\big)+\sum_{n=1}^{N_\tau}\tau_n\nu \mathbf a(\mathbf z_{n,h},\bm \eta^{\bm \mu}_{n,\tau})+\sum_{n=1}^{N_\tau}\tau_n \mathbf b(\mathbf z_{n-1,h},\mathbf y_{n,h},\bm \eta^{\bm \mu}_{n,\tau})\nonumber\\
&+\sum_{n=1}^{N_\tau}\tau_n \mathbf b(\mathbf y_{n-1,h},\mathbf z_{n,h},\bm \eta^{\bm \mu}_{n,\tau})+\sum_{n=1}^{N_\tau}\tau_n \beta(\xi_{n-1,h}\mathbf g, \bm \eta^{\bm \mu}_{n,\tau})+\sum_{n=1}^{N_\tau} \big(\xi_{n,h}-\xi_{n-1,h}, \eta^\kappa_{n,h}\big)+\sum_{n=1}^{N_\tau}\tau_n\big[\chi a(\xi_{n,h},\eta^\kappa_{n,h})\nonumber\\
&+\eta\gamma(\xi_{n,h}, \eta^\kappa_{n,h})_\Gamma\big]+\sum_{n=1}^{N_\tau} \tau_n b(\mathbf z_{n-1,h}, \theta_{h}(t_n),\eta^\kappa_{n,h})+\sum_{n=1}^{N_\tau}\tau_n b(\mathbf y_{n-1,h}, \xi_{n,h},\eta^\kappa_{n,h}).
\end{align}
For \eqref{semi_discrete:adjoint} integrating from $t_{n-1}$ to $t_n$, and  subtracting in \eqref{MuKappa:Full_discrete:adjoint} multiple $\tau_n$, from it with \((\mathbf w_h,\zeta_h)=(\mathbf z_{n,h}, \xi_{n,h}) \)  yields
\begin{align*}
&\tau_n\nu\mathbf a(\bm \eta_{n,\tau}^{\bm \mu},\mathbf z_{n,h})=-(\bm \eta_{n,\tau}^{\bm \mu}-\bm \eta_{n+1,\tau}^{\bm \mu},\mathbf z_{n,h})-\int_{t_{n-1}}^{t_n}\nu\mathbf a(\bm \zeta^{\bm \mu}_\tau, \mathbf z_{n,h})dt-\int_{t_{n-1}}^{t_n}\mathbf b(\mathbf z_{n,h}, \mathbf y_h- \mathbf y_\sigma, \bm \mu_h)dt\\
&-\int_{t_{n-1}}^{t_n}\mathbf b(\mathbf z_{n,h},\mathbf y_{n,h}-\mathbf y_{n+1,h},\bm \mu_h)dt-\int_{t_{n-1}}^{t_n}\mathbf b(\mathbf z_{n,h},\mathbf y_{n+1,h},\bm \zeta^{\bm \mu}_{\tau})dt-\tau_n\mathbf b(\mathbf z_{n,h}, \mathbf y_{n+1,h},\bm \mu_h(t_{n-1} )-\bm \mu_h(t_{n}))\\
&-(\tau_n-\tau_{n+1})\mathbf b(\mathbf z_{n,h}, \mathbf y_{n+1,h}, \bm \mu_h(t_n))-\tau_{n+1}\mathbf b(\mathbf z_{n,h}, \mathbf y_{n+1,h},\bm \eta^{\bm \mu}_{n+1,h})-\int_{t_{n-1}}^{t_n}\mathbf b(\mathbf y_h-\mathbf y_\sigma, \mathbf z_{n,h}, \bm \mu_h)dt\\
&-\int_{t_{n-1}}^{t_n}\mathbf b(\mathbf y_{n,h},\mathbf z_{n,h},\bm\zeta^{\bm\mu}_\tau)dt-\tau_n\mathbf b(\mathbf y_{n,h}-\mathbf y_{n-1,h},\mathbf z_{n,h}, \bm \mu_h(t_{n-1}))-\tau_n\mathbf b(\mathbf y_{n-1,h}, \mathbf z_{n,h}, \bm \eta^{\bm \mu}_{n,h})\\
&-\int_{t_{n-1}}^{t_n}b(\mathbf z_{n,h}, \zeta^\theta_\tau, \kappa_h)dt-\int_{t_{n-1}}^{t_n}b(\mathbf z_{n,h}, \theta_h(t_n)-\theta_h(t_{n+1}), \kappa_h)dt-\int_{t_{n-1}}^{t_n}b(\mathbf z_{n,h}, \theta_h(t_{n+1}), \zeta^{\kappa}_\tau)dt\\
&-\tau_n b(\mathbf z_{n,h}, \theta_h(t_{n+1}), \kappa_h(t_{n-1})-\kappa_h(t_{n}))-(\tau_n-\tau_{n+1})b(\mathbf z_{n,h}, \theta_h(t_{n+1}), \kappa_h(t_n))-\tau_{n+1} b(\mathbf z_{n,h}, \theta_h(t_{n+1}),\eta_{n+1,\tau}^{\kappa})\\
&-\tau_{n+1}b(\mathbf z_{n,h}, \theta_h(t_{n+1})-\theta_{n+1,h},\kappa_{n+1,h})+\int_{t_{n-1}}^{t_n}(\mathbf y_h-\mathbf y_\sigma, \mathbf z_{n,h})dt,\\
&\chi\tau_n a(\eta_{n,\tau}^\kappa, \xi_{n,h})+\tau_n\eta\gamma(\eta^\kappa_{n,\tau}, \xi_{n,h})_\Gamma=-(\eta_{n,\tau}^\kappa-\eta_{n+1, \tau}^\kappa, \xi_{n,h})-\chi \int_{t_{n-1}}^{t_n} a(\zeta_\tau^\kappa, \xi_{n,h})dt -\eta\gamma\int_{t_{n-1}}^{t_n}(\zeta_\tau^\kappa, \xi_{n,h})_\Gamma dt\\
&-\int_{t_{n-1}}^{t_n}b(\mathbf y_h-\mathbf y_\sigma, \xi_{n,h}, \kappa_h ) dt-\int_{{t_{n-1}}}^{t_n}b(\mathbf y_\sigma, \xi_{n,h}, \zeta^\kappa_{\tau})dt-\tau_n b(\mathbf y_{n,h}-\mathbf y_{n-1,h}, \xi_{n,h}, \kappa_h(t_{n-1}))\\
&-\tau_n b(\mathbf y_{n-1,h}, \xi_{n,h}, \eta^\kappa_{n,\tau})-\int_{t_{n-1}}^{t_n}\beta(\bm \zeta^{\bm \mu}_\tau,  \xi_{n,h}\mathbf g)dt-(\tau_{n}-\tau_{n+1})\beta(\bm \mu_h(t_{n-1}), \xi_{n,h}\mathbf g)-\tau_{n+1}\beta(\bm \mu_h(t_{n-1})-\bm \mu_h(t_n),\xi_{n,h}\mathbf g)\\
&- \tau_{n+1}\beta(\bm \eta^{\bm \mu}_{n+1,\tau}, \xi_{n,h}\mathbf g))+\int_{t_{n-1}}^{t_n}(\theta_h-\theta_\sigma, \xi_{n,h}\mathbf g)dt
\end{align*}
with $\theta(t_{N_\tau+1}):=\theta(t_{N_\tau})$. We substitute the above equation into \eqref{Full_err:MUkappa:time:L2:1} and perform the necessary simplifications to obtain
\begin{align}
&\|\bm \eta^{\bm  \mu}_\tau\|^2_{L^2(I;\mathbb L^2(\Omega))}+\| \eta^\kappa_\tau\|^2_{L^2(I;L^2(\Omega))}=\sum_{n=1}^{N_\tau}\Big(-\nu\int_{t_{n-1}}^{t_n}\mathbf a(\bm \zeta^{\bm \mu}_\tau, \mathbf z_{n,h})dt-\chi \int_{t_{n-1}}^{t_n} a(\zeta_\tau^\kappa, \xi_{n,h})dt -\eta\gamma\int_{t_{n-1}}^{t_n}(\zeta_\tau^\kappa, \xi_{n,h})_\Gamma dt\Big)\nonumber\\
&+\sum_{n=1}^{N_\tau}\Big(-\int_{t_{n-1}}^{t_n}\mathbf b(\mathbf z_{n,h}, \mathbf y_h- \mathbf y_\sigma, \bm \mu_h)dt-\int_{t_{n-1}}^{t_n}\mathbf b(\mathbf z_{n,h},\mathbf y_{n,h}-\mathbf y_{n+1,h},\bm \mu_h)dt-\int_{t_{n-1}}^{t_n}\mathbf b(\mathbf z_{n,h},\mathbf y_{n+1,h},\bm \zeta^{\bm \mu}_{\tau})dt\nonumber\\
&-\tau_n\mathbf b(\mathbf z_{n,h}, \mathbf y_{n+1,h},\bm \mu_h(t_{n-1} )-\bm \mu_h(t_{n}))-(\tau_n-\tau_{n+1})\mathbf b(\mathbf z_{n,h}, \mathbf y_{n+1,h}, \bm \mu_h(t_n))\Big)\nonumber\\
&+\sum_{n=1}^{N_\tau}\Big(-\int_{t_{n-1}}^{t_n}\mathbf b(\mathbf y_h-\mathbf y_\sigma, \mathbf z_{n,h}, \bm \mu_h)dt-\int_{t_{n-1}}^{t_n}\mathbf b(\mathbf y_{n,h},\mathbf z_{n,h},\bm\zeta^{\bm\mu}_\tau)dt-\tau_n\mathbf b(\mathbf y_{n,h}-\mathbf y_{n-1,h},\mathbf z_{n,h}, \bm \mu_h(t_{n-1}))\Big)\nonumber\\
&+\sum_{n=1}^{N_\tau}\Big(-\int_{t_{n-1}}^{t_n}b(\mathbf z_{n,h}, \zeta^\theta_\tau, \kappa_h)dt-\int_{t_{n-1}}^{t_n}b(\mathbf z_{n,h}, \theta_h(t_n)-\theta_h(t_{n+1}), \kappa_h)dt-\int_{t_{n-1}}^{t_n}b(\mathbf z_{n,h}, \theta_h(t_{n+1}), \zeta^{\kappa}_\tau)dt\nonumber\\
&-\tau_n b(\mathbf z_{n,h}, \theta_h(t_{n+1}), \kappa_h(t_{n-1})-\kappa_h(t_{n}))-(\tau_n-\tau_{n+1})b(\mathbf z_{n,h}, \theta_h(t_{n+1}), \kappa_h(t_n))\nonumber\\
&-\tau_{n+1}b(\mathbf z_{n,h}, \theta_h(t_{n+1})-\theta_{n+1,h},\kappa_{n+1,h})\Big)+\sum_{n=1}^{N_\tau}\Big(-\int_{t_{n-1}}^{t_n}b(\mathbf y_h-\mathbf y_\sigma, \xi_{n,h}, \kappa_h ) dt-\int_{{t_{n-1}}}^{t_n}b(\mathbf y_\sigma, \xi_{n,h}, \zeta^\kappa_{\tau})dt\nonumber\\
&-\tau_n b(\mathbf y_{n,h}-\mathbf y_{n-1,h}, \xi_{n,h}, \kappa_h(t_{n-1})) \Big)
+\sum_{n=1}^{N_\tau}\Big(\int_{t_{n-1}}^{t_n}(\mathbf y_h-\mathbf y_\sigma, \mathbf z_{n,h})dt-\int_{t_{n-1}}^{t_n}\beta(\bm \zeta^{\bm \mu}_\tau,  \xi_{n,h}\mathbf g)dt\nonumber\\
&-(\tau_{n}-\tau_{n+1})\beta(\bm \mu_h(t_{n-1}), \xi_{n,h}\mathbf g)-\tau_{n+1}\beta(\bm \mu_h(t_{n-1})-\bm \mu_h(t_n),\xi_{n,h}\mathbf g)+\int_{t_{n-1}}^{t_n}(\theta_h-\theta_\sigma, \xi_{n,h}\mathbf g)dt\Big)\nonumber\\
&=J_1+J_2+J_3+J_4+J_5+J_6.\label{Full_err:MUkappa:time:L2:2}
\end{align}
Based on the previous stability analysis, the terms $J_i\ (i = 1, \cdots, 6)$ in \eqref{Full_err:MUkappa:time:L2:2} can be estimated as follows:
\begin{align*}
J_1&\leq \nu\|\bm \zeta^{\bm \mu}_\tau\|_{L^2(I;\mathbb L^2(\Omega))}\Big(\sum_{n=1}^{N_\tau}\tau_n\|\mathbf A_h \mathbf z_{n,h}\|^2\Big)^\frac{1}{2}+\|\zeta^\kappa_\tau\|_{L^2(I;L^2(\Omega))}\Big(\sum_{n=1}^{N_\tau}\tau_n \|\Delta_h\xi_{n,h}\|^2\Big),\\
J_2+J_3&\leq C\Big(\|\mathbf y_h-\mathbf y_\sigma\|_{L^2(I;\mathbb L^2(\Omega))}+\tau^\frac{1}{2}\Big(\sum_{n=1}^{N_\tau}\|\mathbf y_{n-1,h}-\mathbf y_{n,h} \|^2\Big)^\frac{1}{2}+ \|\bm \zeta^{\bm \mu}_\tau\|_{L^2(I;\mathbb L^2(\Omega))}+\tau\|\partial_t \bm \mu_h\|_{L^2(I;\mathbb L^2(\Omega))}\\
&+\max_{1\leq n\leq N_\tau}|\tau_n-\tau_{n+1}|\tau^{-1}\Big)\Big(\Big(\sum_{n=1}^{N_\tau}\tau_n\|\nabla \mathbf z_{n,h}\|^2\Big)^\frac{1}{4}\Big(\sum_{n=1}^{N_\tau}\tau_n\|\mathbf A_h \mathbf z_{n,h}\|^2\Big)^\frac{1}{4}+\max_{1\leq n\leq N_\tau} \|\nabla \mathbf z_{n,h}\|\Big),\\
J_4\leq &C \Big(\|\zeta^\theta_\tau\|_{L^2(I;L^2(\Omega))}+ \tau\|\partial_t \theta_h\|_{L^2(I;L^2(\Omega))}+\|\zeta^\kappa_\tau\|_{L^2(I;L^2(\Omega))}+\tau\|\partial_t \kappa_h\|_{L^2(I;L^2(\Omega))}+\|\eta^\theta_\tau\|_{L^2(I;L^2(\Omega))}\\
&+\max_{1\leq n\leq N_\tau}|\tau_n-\tau_{n+1}|\tau^{-1}\Big)\Big(\Big(\sum_{n=1}^{N_\tau}\tau_n\|\nabla \mathbf z_{n,h}\|^2\Big)^\frac{1}{4}\Big(\sum_{n=1}^{N_\tau}\tau_n\|\mathbf A_h \mathbf z_{n,h}\|^2\Big)^\frac{1}{4}+\max_{1\leq n\leq N_\tau} \|\nabla \mathbf z_{n,h}\|\Big),\\
J_5&\leq C \Big(\|\mathbf y_h-\mathbf y_\sigma\|_{L^2(I;\mathbb L^2(\Omega))}+\|\zeta^\kappa_\tau\|_{L^2(I;L^2(\Omega))}+\tau^\frac{1}{2}\Big(\sum_{n=1}^{N_\tau}\|\mathbf y_{n,h}-\mathbf y_{n-1,h}\|^2\Big)^\frac{1}{2}\Big)\Big(\max_{1\leq n\leq N_\tau} \|\xi_{n,h}\|_{H^1(\Omega)}\\
&+\big(\sum_{n=1}^{N_\tau}\tau_n\|\xi_{n,h}\|^2_{H^1(\Omega)}\big)^\frac{1}{4}\big(\sum_{n=1}^{N_\tau}\tau_n\|\Delta_h \xi_{n,h}\|^2\big)^\frac{1}{4}\Big),\\
J_6&\leq C\Big(\|\mathbf y_h-\mathbf y_\sigma\|_{L^2(I;\mathbb L^2(\Omega))}+\|\bm \zeta^{\bm \mu}_\tau\|_{L^2(I;\mathbb L^2(\Omega))}+\tau^\frac{1}{2}\|\partial_t\bm \mu_h\|_{L^2(I;\mathbb L^2(\Omega))}+\|\theta_h-\theta_\sigma\|_{L^2(I;L^2(\Omega))}\\
&+\max_{1\leq n\leq N_\tau}|\tau_n-\tau_{n+1}|\tau^{-1}\Big)\Big(\sum_{n=1}^{N_\tau}\tau_n\big[\|\mathbf z_{n,h}\|^2+\|\xi_{n,h}\|^2\big]\Big).
\end{align*}
Combining the above estimates with \eqref{Full_err:MUkappa:time:L2:2} and  \eqref{Full_err:MuKappa:time:L2:bu:3}, using assumption \((\mathbf B) \), Lemmas \ref{int:L2:time:err},  \ref{lem:Full:YTheta:bu} and \ref{Full_err:YTetha:time:L2} the conclusion can be derived.
\end{proof}

Combining Lemmas \ref{Lem:semi:state:err}, \ref{Lem:semi:adjoint:err},  \ref{semi_err:MuKappa:L2},  \ref{lem:Full:YTheta:time:err},  \ref{lem:err:Full:adjoint} and  \ref{Full_err:YTetha:time:L2}, we can obtain the following theorem regarding both temporal and spatial errors.
\begin{theorem}\label{thm:All:TimeSpace:staAdj} Under the assumptions of Lemmas \ref{semi_err:MuKappa:L2} and \ref{Full_err:Mukappa:time:L2},  the following estimates hold:
\begin{align}
\|\mathbf y-\mathbf y_{\sigma}\|_{L^\infty(I;\mathbb L^2(\Omega))}&+\|\mathbf y-\mathbf y_{\sigma}\|_{L^2(I;\mathbb H^1(\Omega))}\nonumber\\
&+\|\theta-\theta_\sigma\|_{L^\infty(I;L^2(\Omega))}+\|\theta-\theta_\sigma\|_{L^2(I;H^1(\Omega))}\leq C\big(\tau^\frac{1}{2}+h\big),\label{Full:YTheta:timespace:err}\\
\|\bm \mu-\bm \mu_{\sigma}\|_{L^\infty(I;\mathbb L^2(\Omega))}&+\|\bm \mu-\bm \mu_{\sigma}\|_{L^2(I;\mathbb H^1(\Omega))}\nonumber\\
&+\|\kappa-\kappa_\sigma\|_{L^\infty(I;L^2(\Omega))}+\|\kappa-\kappa_\sigma\|_{L^2(I;H^1(\Omega))}\leq C\big(\tau^\frac{1}{2}+h\big),\label{Full:MuKappa:timespace:err}\\
&\|\kappa-\kappa_\sigma\|_{L^2(I; L^2(\Gamma))}\leq C\big(\tau^\frac{3}{4}+h^\frac{3}{2}\big).\label{Full:bound:Mukappa:timespace:err}
\end{align}
\end{theorem}

\begin{proof}The derivation of \eqref{Full:YTheta:timespace:err} and \eqref{Full:MuKappa:timespace:err} is straightforward. Below, we present the derivation of \eqref{Full:bound:Mukappa:timespace:err}.  Using the triangle inequality and the trace inequality, we deduce
\begin{align*}
    \|\kappa-\kappa_\sigma\|_{L^2(I; L^2(\Gamma))}&\leq \|\kappa-\kappa_h\|_{L^2(I; L^2(\Gamma))}+\|\kappa_h-\kappa_\sigma\|_{L^2(I; L^2(\Gamma))}\\
    &\leq C\|\kappa-\kappa_h\|^\frac{1}{2}_{L^2(I;L^2(\Omega))}\|\kappa-\kappa_h\|^\frac{1}{2}_{L^2(I;H^1(\Omega))}+C\|\kappa_h-\kappa_\sigma\|^\frac{1}{2}_{L^2(I;L^2(\Omega))}\|\kappa_h-\kappa_\sigma\|^\frac{1}{2}_{L^2(I;H^1(\Omega))}.
\end{align*}
Then \eqref{Full:bound:Mukappa:timespace:err} can be obtained from the above inequality along with Lemmas \ref{Lem:semi:adjoint:err}, \ref{semi_err:MuKappa:L2}, \ref{lem:err:Full:adjoint} and \ref{Full_err:Mukappa:time:L2}.
\end{proof}

\subsubsection{Analysis of the discrete control problem} 
In this subsection, we derive some convergence results for the discrete optimal control problem.  The proofs of Theorems \ref{thm:opt:conv_1} and \ref{thm:opt:conv_2} are similar to those of Theorems 4.13 and 4.15 in \cite{Casas_Chrysafinos_2015}, respectively, and are therefore omitted. The following theorem demonstrates the convergence of the discrete optimal control problem $(\mathrm{P_\sigma})$ to the continuous problem $(\mathrm{P})$.

\begin{theorem}\label{thm:opt:conv_1} Let $\bar{u}_\sigma$ be a global solution to the problem $\mathrm{(P_\sigma)}$. Then, the sequence $\{\bar{u}_\sigma\}_\sigma$ is bounded in $L^2(I; L^2(\Gamma))$ and there exist subsequences, denoted in the same way, that converge weakly to a point $\bar{u}$ in $L^2(I; L^2(\Gamma))$, where $\bar{u}$ is a global solution of $\mathrm{(P)}$.  Moreover, we have 
\begin{align}
\lim_{\sigma\to 0}\|\bar{u}_\sigma-\bar{u}\|_{L^2(I;L^2(\Gamma))}=0, \quad \lim_{\sigma\to 0} J_\sigma(\bar{u}_\sigma)=J(\bar{u}). \label{thm_sub:opt:lim:con:1}
\end{align}
\end{theorem}

The next theorem asserts that strict local solutions of the optimal control problem $\mathrm{(P)}$ can be approximated by local solutions of the discrete optimal control problem $\mathrm{(P_\sigma)}$.
\begin{theorem}\label{thm:opt:conv_2} Let $\bar{u}$ be a strict local minimum of problem $\mathrm{(P)}$. Then, there exists a sequence $\{\bar{u}_\sigma\}_{\sigma}$ of local minima of problem $\mathrm{(P_\sigma)}$ such that \eqref{thm_sub:opt:lim:con:1} holds.
\end{theorem}


In Theorem \ref{thm:opt:conv_3}, we replace ``strict local minimum" with ``local minimum" and investigate whether the same conclusion holds. We are unable to prove such a general result because of the lack of uniqueness of the solution. Instead, we can only provide a weaker result for each individual local minimum of \(\mathrm{(P)}\). To this end, we employ a technique from Barbu \cite{Barbu_1984}.   Let $\bar{u}$ be a local minimum of the problem $\mathrm{(P)}$.
 We associate with this minimum the adapted optimal control problem 
\begin{equation*}
    \mathrm{(\widetilde{P}_{\sigma})}\qquad \begin{aligned}
\min_{u \in \mathcal{U}_{ad}}{\widetilde{J}_\sigma(u):=}J_\sigma(u)+\frac{1}{2}\|u-\bar{u}\|^2_{L^2(I;L^2(\Gamma))}.
\end{aligned}
\end{equation*}
Similar to problem \(\mathrm{(P_\sigma)}\), we can prove that the optimal control problem \(\mathrm{(\widetilde{P}_{\sigma})}\) admits a solution.

\begin{theorem}\label{thm:opt:conv_3} Let \( \bar{u} \) be a local minimum of the problem \( \mathrm{(P)} \). Then, there exists a sequence \( \{\widetilde{u}_\sigma\}_{\sigma} \) of local minima of problem \( \mathrm{(\widetilde{P}_\sigma)} \) such that
\[
\lim_{\sigma \to 0} \|\widetilde{u}_\sigma - \bar{u}\|_{L^2(I; L^2(\Gamma))} = 0, \quad \lim_{\sigma \to 0} \widetilde{J}_\sigma(\widetilde{u}_\sigma) = J(\bar{u}).
\]
\end{theorem}
\begin{proof}
We introduce a local adapted optimal control problem 
\begin{equation*}
    \mathrm{(\widetilde{P}_{\sigma, \delta})}\qquad \begin{aligned}
\min_{u\in \mathcal{U}_{ad}\cap \bar{\mathcal{B}}_{\delta}(\bar{u}) } \widetilde{J}_\sigma(u).
 \end{aligned}
\end{equation*}
It follows readily that this problem admits at least one solution. Let $\widetilde{u}_\sigma$ be a global solution to the problem $\mathrm{(\widetilde{P}_{\sigma, \delta})}$. Therefore, we can conclude the existence of subsequences that converge weakly in \( L^2(I; L^2(\Gamma)) \) to some point \( \widetilde{u}\in \mathcal{U}_{ad}\cap \bar{\mathcal{B}}_{\delta}(\bar{u}) \).  Furthermore, utilizing Lemmas \ref{thm:Weak_Weak:S}, \ref{appB:Lem:full:state:err} and the compact embedding $\mathbb{W} \hookrightarrow L^2(I; \mathbb{L}^2(\Omega)) \times L^2(I; L^2(\Omega))$, we can deduce that $(\mathbf{y}_\sigma(\widetilde{u}_\sigma), \theta_\sigma(\widetilde{u}_\sigma)) \to (\mathbf{y}_{\widetilde{u}}, \theta_{\widetilde{u}})$ in $L^2(I; \mathbb{L}^2(\Omega)) \times L^2(I; L^2(\Omega))$. Additionally, by the weakly sequential lower semicontinuity of norms, we obtain
\begin{align*}
J(\bar{u})+\frac{1}{2}\|\widetilde{u}-\bar{u}\|^2_{L^2(I;L^2(\Gamma))} \leq J(\widetilde{u})+\frac{1}{2}\|\widetilde{u}-\bar{u}\|^2_{L^2(I;L^2(\Gamma))} \leq \liminf\limits_{\sigma \to 0} \widetilde{J}_\sigma(\widetilde{u}_\sigma)\leq \limsup\limits_{\sigma \to 0} \widetilde{J}_\sigma(\bar{u}) = J(\bar{u}).
\end{align*}
This implies $\widetilde{u}=\bar{u}$,  $\lim_{\sigma \to 0} \|\widetilde{u}_\sigma - \bar{u}\|_{L^2(I; L^2(\Gamma))} = 0$ and ${\lim\limits_{\sigma\to 0} \widetilde{J}_\sigma(\widetilde{u}_\sigma)=J(\bar{u})}$. It follows that $\widetilde{u}_\sigma$ is a local solution to the problem $\mathrm{(\widetilde{P}_{\sigma,\delta})}$ for a sufficiently small $\sigma$.  This completes the proof. 
\end{proof}

Let \(\{\bar{u}_\sigma\}_{\sigma}\) be a sequence of local minima of the discrete control problem \(\mathrm{(P_\sigma)}\) such that \(\bar{u}_\sigma \to \bar{u}\) in \(L^2(I;L^2(\Gamma))\) as \(\sigma \to 0\), where \(\bar{u}\) denotes a local solution of the problem \(\mathrm{(P)}\). In the following, we will derive the error estimates for the optimal control along with the corresponding state and adjoint variables. Before presenting the error estimate, we state the following auxiliary lemma. 

\begin{lemma}\label{Lem:optim:Aux:ineq} Let \(\bar{u}\) satisfy the second-order optimality condition \eqref{second:equi:cond}. Then, there exists $\sigma_{*}>0$ such that we can extract a subsequence from $\{\bar{u}_\sigma\}_\sigma$, denoted in the same way, such that
\begin{align}
2^{-1}\min\{\alpha, \lambda\}\|\bar{u}-\bar{u}_{\sigma}\|^2_{L^2(I;L^2(\Gamma))}\leq \big[J'(\bar{u}_\sigma)-J'(\bar{u})\big](\bar{u}_\sigma-\bar{u})\quad \text{for\ any } \,|\sigma| <\sigma_*.
\end{align}
\end{lemma}

\begin{proof} 
Define $v_\sigma=\|\bar{u}_\sigma-\bar{u}\|_{L^2(I;L^2(\Gamma))}^{-1}(\bar{u}_\sigma-\bar{u})$. We observe that \( v_\sigma \rightharpoonup v \) in \( L^2(I; L^2(\Gamma)) \) along a subsequence as \( \sigma \to 0 \). Without loss of generality, we may denote this subsequence in the same manner, and the limit point \( v\) satisfies the conditions \eqref{cone:cond:1}-\eqref{cone:cond:2}.  Next, we prove that \( J'(\bar{u}) \neq 0 \) implies \( v = 0 \). To this end, using Lemma \ref{Full:lem:Lip:MuKappa} and Theorem \ref{thm:All:TimeSpace:staAdj}, we can deduce that 
\[
\|\bar{\kappa} - \bar{\kappa}_{\sigma}\|_{L^2(I; L^2(\Gamma))} \leq \|\bar{\kappa} - \kappa_\sigma(\bar{u})\|_{L^2(I; L^2(\Gamma))} + \|\kappa_\sigma(\bar{u}) - \bar{\kappa}_\sigma\|_{L^2(I; L^2(\Gamma))} \to 0 \quad \text{as} \quad \sigma \to 0.
\]
Thus, combining the above convergence results, we obtain
\begin{align*}
J'(\bar{u})v=\lim_{\sigma\to 0}J'_\sigma(\bar{u}_\sigma)v_\sigma=\lim_{\sigma\to 0}\|\bar{u}_\sigma-\bar{u}\|_{L^2(I;L^2(\Gamma))}^{-1}J'_\sigma(\bar{u}_\sigma)(\bar{u}_\sigma-\bar{u})\leq 0.
\end{align*}

Furthermore, since \( v \) satisfies conditions \eqref{cone:cond:1}-\eqref{cone:cond:2}, this implies that \( \mathcal{J}(t,x) v(t,x) \geq 0 \), where \( \mathcal{J} = J'(\bar{u}) \).  Consequently, if $\mathcal J(t,x)\neq 0$, then $v=0$.  

Applying the Mean Value Theorem, we obtain
\begin{align}\label{Mean:value:Theorem}
\big[J'(\bar{u}_\sigma)-J'(\bar{u})\big](\bar{u}_\sigma-\bar{u})=J''(\hat{u}_\sigma)(\bar{u}_\sigma-\bar{u})^2, \quad \hat{u}_\sigma=\bar{u}_\sigma+\theta_{\sigma}(\bar{u}_\sigma-\bar{u}),
\end{align}
where $\theta_\sigma\in (0,1)$. Moreover, by applying Lemmas \ref{lem:estimate:Z} and \ref{lem:Mu_Kappa}, and the bound \( \|\mathbf{z}_{\hat{u}_\sigma, v_\sigma}\|_{L^\infty(I; \mathbb{H}^1(\Omega))} +\|\kappa_{\hat{u}_\sigma}\|_{L^\infty(I;H^1(\Omega))}+\|\xi_{\hat{u}_\sigma,v_\sigma}\|_{L^2(I;H^1(\Omega))}\leq C \) that follows from the standard Galerkin method, we can conclude from \eqref{thm:J:diff2} and \eqref{second:equi:cond} that
\begin{align*}
\lim_{\sigma\to 0}J''(\hat{u}_\sigma)v^2_\sigma&=\lim_{\sigma\to 0}\Big(\int_{\Omega_T}\left(|\mathbf z_{\hat{u}_\sigma,v_\sigma}|^2+|\xi_{\hat{u}_\sigma,v_\sigma}|^2-2(\mathbf z_{\hat{u}_\sigma,v_\sigma}\cdot\nabla)\mathbf z_{\hat{u}_\sigma,v_\sigma}\cdot\bm \mu_{\hat{u}_\sigma}-2(\mathbf z_{\hat{u}_\sigma,v_\sigma}\cdot\nabla)\xi_{\hat{u}_\sigma,v_\sigma} \kappa_{\hat{u}_\sigma}\right)dxdt\\
&+\alpha\Big)=J''(\bar{u})v^2+\alpha(1-\|v\|^2_{L^2(I;L^2(\Gamma))})\\
&\geq\alpha+(\lambda-\alpha)\|v\|^2_{L^2(I;L^2(\Gamma))}.
\end{align*}
Given that \( \|v\|^2_{L^2(I;L^2(\Gamma))} \leq 1 \), it follows that \( \lim_{\sigma \to 0} J''(\hat{u}_\sigma) v_\sigma^2 \geq \min\{\alpha, \lambda\} >0\). Therefore, there exists \( \sigma_* \) such that for all \( |\sigma| < \sigma_* \), we have \( J''(\hat{u}_\sigma) v_\sigma^2 \geq \frac{1}{2} \min\{\alpha, \lambda\} \). Finally, combining with \eqref{Mean:value:Theorem}, the conclusion follows.
\end{proof}

Next, the error estimates for the optimal control are provided along with the corresponding state and adjoint variables.
\begin{theorem}\label{thm:estimate:main_result} Let \(\bar{u}\) satisfy the second-order optimality condition \eqref{second:equi:cond}. Then, there exists \( \sigma_* > 0 \) such that for any $|\sigma|\leq \sigma_*$, the following error estimates hold:
\begin{align}
&\|\bar{u}-\bar{u}_\sigma\|_{L^2(I;L^2(\Gamma))}\leq C(\tau^\frac{3}{4}+h^\frac{3}{2}),\label{ineq:Final:err:cont}\\
\|\bar{\mathbf y}-\bar{\mathbf y}_{\sigma}\|_{L^\infty(I;\mathbb L^2(\Omega))}&+\|\bar{\mathbf y}-\bar{\mathbf y}_{\sigma}\|_{L^2(I;\mathbb H^1(\Omega))}\nonumber\\
&+\|\bar{\theta}-\bar{\theta}_\sigma\|_{L^\infty(I;L^2(\Omega))}+\|\bar{\theta}-\bar{\theta}_\sigma\|_{L^2(I;H^1(\Omega))}\leq C\big(\tau^\frac{1}{2}+h\big),\label{ineq:Final:err:state}\\
\|\bar{\bm \mu}-\bar{\bm \mu}_{\sigma}\|_{L^\infty(I;\mathbb L^2(\Omega))}&+\|\bar{\bm \mu}-\bar{\bm \mu}_{\sigma}\|_{L^2(I;\mathbb H^1(\Omega))}\nonumber\\
&+\|\bar{\kappa}-\bar{\kappa}_\sigma\|_{L^\infty(I;L^2(\Omega))}+\|\bar{\kappa}-\bar{\kappa}_\sigma\|_{L^2(I;H^1(\Omega))}\leq C\big(\tau^\frac{1}{2}+h\big).\label{ineq:Final:err:adjoint}
\end{align}
\end{theorem}
\begin{proof}We define the piecewise linear interpolation of \(\{\kappa_{n,h}\}_{n=1}^{N_{\tau}+1}\) on \(I^l_n\) for \(n = 1, \dots, N_\tau\) as follows:
\[
\widetilde{\kappa}_\sigma(t,x) := \frac{t_n - t}{\tau_n} \bar{\kappa}_{n,h}(x) + \frac{t - t_{n-1}}{\tau_n} \bar{\kappa}_{n+1,h}(x).
\]
Moreover, using \eqref{Full:stab:adjoint:ineq2} and \eqref{Full:stab:adjoint:ineq3}, we obtain the following estimates:
\begin{equation}\label{ineq:Final:conRate:1}
\begin{aligned}
~\|\widetilde{\kappa}_\sigma-\bar{k}_\sigma\|^2_{L^2(I;L^2(\Omega))}&=\sum_{n=1}^{N_\tau}\int_{t_{n-1}}^{t_n}\|\widetilde{\kappa}_\sigma-\bar{k}_\sigma\|^2dt=\frac{1}{3}\sum_{n=1}^{N_\tau}\tau_n\|\bar{\kappa}_{n,h}-\bar{\kappa}_{n+1,h}\|^2\leq C \tau^2,\\
\|\widetilde{\kappa}_\sigma-\bar{k}_\sigma\|^2_{L^2(I;H^1(\Omega))}&\leq C\tau.
\end{aligned}
\end{equation}
We set $\widetilde{u}_\sigma=\mathrm{Proj}_{[u_a,u_b]}\Big(-\frac{\eta}{\alpha}\widetilde{\kappa}_\sigma\Big)$.  Since \(\bar{u}_\sigma \to \bar{u}\) in \(L^2(I; L^2(\Gamma))\) as \(\sigma \to 0\), it follows that 
\begin{align}
\|\bar{u}_\sigma\|_{L^2(I; L^2(\Gamma))} \leq C.\label{ineq:Final:conRate:bu:1}
\end{align} 
Furthermore, by Lemma \ref{lem:stab:Full:adjoint:bu}, we obtain 
\[
\|\widetilde{\kappa}_\sigma\|_{H^1(I; L^2(\Omega))} + \|\widetilde{\kappa}_\sigma\|_{L^2(I; H^1(\Omega))} \leq C.
\]
Again, using Lemma 2.1 of \cite[p. 80]{Malanowski_1982}, we obtain
\[
\|\widetilde{\kappa}_\sigma\|_{H^{\frac{1}{4}}(I; L^2(\Gamma))} + \|\widetilde{\kappa}_\sigma\|_{L^2(I; H^{\frac{1}{2}}(\Gamma))} \leq C.
\]
Thus, we conclude (cf. \cite[p. 1735, Lemma 3.3]{Kunisch_Vexler_2007})
\begin{align}
\|\widetilde{u}_\sigma\|_{H^{\frac{1}{4}}(I; L^2(\Gamma))} + \|\widetilde{u}_\sigma\|_{L^2(I; H^{\frac{1}{2}}(\Gamma))} \leq C.\label{ineq:Final:conRate:2}
\end{align}

With Lemma \ref{Lem:optim:Aux:ineq}, \eqref{first_cond:varition_ineq} and \eqref{first_cond:discrete:varition_ineq}, every subsequence of $\{\bar{u}_\sigma\}_\sigma$ has a further subsequence, denoted in the same way, that satisfies 
\begin{align}2^{-1}\min\{\alpha, \lambda\}\|\bar{u}-\bar{u}_{\sigma}\|^2_{L^2(I;L^2(\Gamma))}&\leq J'(\bar{u}_\sigma)(\bar{u}_\sigma-\bar{u})\leq [J'(\bar{u}_\sigma)-J'_\sigma(\bar{u}_\sigma)](\bar{u}_\sigma-\bar{u})\nonumber\\
&= [J'(\bar{u}_\sigma)-J'(\widetilde{u}_\sigma)](\bar{u}_\sigma-\bar{u})+[J'(\widetilde{u}_\sigma)-J'_\sigma(\widetilde{u}_\sigma)](\bar{u}_\sigma-\bar{u})\nonumber\\
&+[J'_\sigma(\widetilde{u}_\sigma)-J'_\sigma(\bar{u}_\sigma)](\bar{u}_\sigma-\bar{u})\label{ineq:Final:conRate:3}
\end{align}
for every $|\sigma|< \sigma_*$. For the first term, we deduce from Lemma \ref{lem:Mu_Kappa}, \eqref{ineq:Final:conRate:1}, \eqref{ineq:Final:conRate:bu:1} and \eqref{ineq:Final:conRate:2} that
\begin{align*}
[J'(\bar{u}_\sigma)-J'(\widetilde{u}_\sigma)](\bar{u}_\sigma-\bar{u})&\leq C\big(\eta\|\kappa_{\widetilde{u}_\sigma}-\kappa_{\bar{u}_\sigma}\|_{L^2(I;L^2(\Gamma))}+\alpha\|\bar{u}_\sigma-\widetilde{u}_\sigma\|_{L^2(I;L^2(\Gamma))}\big)\|\bar{u}-\bar{u}_\sigma\|_{L^2(I;L^2(\Gamma))}\\
&\leq C\|\bar{u}_\sigma-\widetilde{u}_\sigma\|_{L^2(I;L^2(\Gamma))}\|\bar{u}-\bar{u}_\sigma\|_{L^2(I;L^2(\Gamma))}\\
&\leq C\|\widetilde{\kappa}_\sigma-\bar{k}_\sigma\|_{L^2(I;L^2(\Gamma))}\|\bar{u}-\bar{u}_\sigma\|_{L^2(I;L^2(\Gamma))}\\
&\leq C\|\widetilde{\kappa}_\sigma-\bar{k}_\sigma\|_{L^2(I;L^2(\Omega))}\|^\frac{1}{2}\|\widetilde{\kappa}_\sigma-\bar{k}_\sigma\|^\frac{1}{2}_{L^2(I;H^1(\Omega))}\|\bar{u}-\bar{u}_\sigma\|_{L^2(I;L^2(\Gamma))}\\
&\leq C \tau^\frac{3}{4}\|\bar{u}-\bar{u}_\sigma\|_{L^2(I;L^2(\Gamma))}.
\end{align*}
For the second term, from Theorem \ref{thm:All:TimeSpace:staAdj} and \eqref{ineq:Final:conRate:2} it follows that
\begin{align*}
[J'(\widetilde{u}_\sigma)-J'_\sigma(\widetilde{u}_\sigma)](\bar{u}_\sigma-\bar{u})&\leq C\eta\|\kappa_\sigma(\widetilde{u}_\sigma)-\kappa(\widetilde{u}_\sigma)\|_{L^2(I;L^2(\Gamma))}\|\bar{u}-\bar{u}_\sigma\|_{L^2(I;L^2(\Gamma))}\\
&\leq C\big(\tau^\frac{3}{4}+h^\frac{3}{2}\big)\|\bar{u}-\bar{u}_\sigma\|_{L^2(I;L^2(\Gamma))}.
\end{align*}
For the last term, using Lemma \ref{Full:lem:Lip:MuKappa}, \eqref{ineq:Final:conRate:1}, \eqref{ineq:Final:conRate:bu:1} and \eqref{ineq:Final:conRate:2}, we deduce 
\begin{align*}
[J'_\sigma(\widetilde{u}_\sigma) - J'_\sigma(\bar{u}_\sigma)](\bar{u}_\sigma - \bar{u}) \leq C \tau^{\frac{3}{4}} \|\bar{u} - \bar{u}_\sigma\|_{L^2(I; L^2(\Gamma))}.
\end{align*}
Combining the above estimates with \eqref{ineq:Final:conRate:3},  we can conclude that 
every subsequence of $\{\bar{u}_\sigma\}_\sigma$ has a further subsequence that satisfies 
\begin{align}
\|\bar{u}-\bar{u}_\sigma\|_{L^2(I;L^2(\Gamma))}\leq C(\tau^\frac{3}{4}+h^\frac{3}{2})\label{ineq:Final:subseq}.
\end{align}
To derive \eqref{ineq:Final:err:cont}, we proceed by contradiction and assume that it does not hold. This implies the existence of a sequence $\sigma$ such that
$$
\lim\limits_{\sigma \to 0} \frac{\|\bar{u} - \bar{u}_\sigma\|_{L^2(I;L^2(\Gamma))}}{\tau^\frac{3}{4} + h^\frac{3}{2}} = +\infty.
$$
However, this contradicts \eqref{ineq:Final:subseq}.

The estimates \eqref{ineq:Final:err:state} and \eqref{ineq:Final:err:adjoint} follow directly from Lemmas \ref{Full:lem:Lip:YTheta}, \ref{Full:lem:Lip:MuKappa}, and Theorem \ref{thm:All:TimeSpace:staAdj}.
\end{proof}

\begin{remark} Consider the objection functionals involving the vorticity of the fluid  
\begin{equation}\begin{aligned}
		\min_{\substack{\tiny u\in \mathcal U_{ad}}}\quad J(u)=\frac{1}{2}\int_{I}\int_\Omega|\mathrm{curl} \,\mathbf y_u|^2 dx dt+\frac{\alpha}{2}\int_{I}\int_\Gamma| u|^2ds dt,
	\end{aligned}\label{Goal:Vort}
\end{equation}
where $(\mathbf y_u, \theta_u)$ is the unique solution pair  of \eqref{intr:Rb_state}. Using Theorem \ref{lem:conL2:Regualrit:YTheta}, we can define $\mathcal{S}: L^2(I;L^2(\Gamma))\rightarrow \mathbb V(I) \times W(I)$ such that $u\mapsto \mathcal{S}(u):=(\mathbf y_u, \theta_u)$.  Similarly, we obtain the same result as in Theorem~\ref{thm:estimate:main_result}.
\end{remark}

\section{Numerical Experiments}\setcounter{equation}{0}
In this section, we will report an example to verify our theoretical results in Theorem \ref{thm:estimate:main_result}. Numerical simulations are conducted using the open-source finite-element software NGSolve \cite{Schoberl2014}, which is available at \url{https://ngsolve.org/}. For solving the algebraic equations in all examples, we employ the sparse direct solver provided by PARDISO \cite{Schenk2004}. We shall focus on Mini ﬁnite elements; however, the adaptation to  Taylor-Hood element can be carried out in a similar fashion. Throughout the experiment, we used the projection gradient method with a fixed step size $\frac{1}{\alpha}$ and set the tolerance as $10^{-6}$.  The corresponding codes for all numerical experiments can be found in the git repository \url{https://github.com/TyphoohYe/OptimalControlWithBoussinesqEquations.git}.

We consider examples with pointwise control constraints  where the exact solutions are unknown and take numerical solutions with time step size $\tau=2^{-9}$ and space size $ h=2^{-7}$ as reference solution $(\bar{\mathbf y}^*, \bar{\theta}^*, \bar{\bm \mu}^*,\bar{\kappa}^*,\bar{u}^*)$ to compute the convergence order.  We set $\Omega = (0, 1) \times (0, 1)$, and the parameters $u_a=-0.2$, $u_b=0.2$, $T = 1$, $\nu = 0.1$, $\mathbf g=(-10, 10)^T$, $\beta=\chi=\gamma=\eta=1$ and $\alpha = 0.1$. We now consider the functions  
\begin{align*}
&\mathbf y_d=\big(50 x^2 (x-1)^2 (2y(y-1)^2+2y^2(y-1)), -50 y^2(y-1)^2(2x(x-1)^2+2x^2(x-1))\big)^T,\\
& \theta_d=0, ~\mathbf h=\mathbf y_0=(0, 0)^T, ~f=\theta_0=0.
\end{align*}
 
\begin{figure}[H]
\newcommand{\scale}{.18\textwidth}

\begin{center}

\includegraphics[width=.19\textwidth]{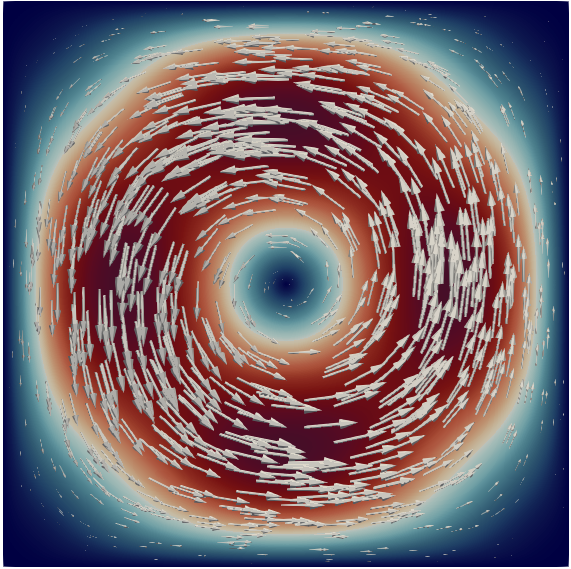}\hspace{-0.5ex}
\includegraphics[scale=0.30]{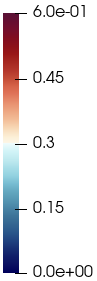}
\caption{The magnitude of the target velocity $\mathbf y_d$}\label{Fig:1}
\end{center}\vspace{-3ex}
\setcounter{subfigure}{0}
\subfloat[][$t=0.125$]{
\includegraphics[width=\scale]{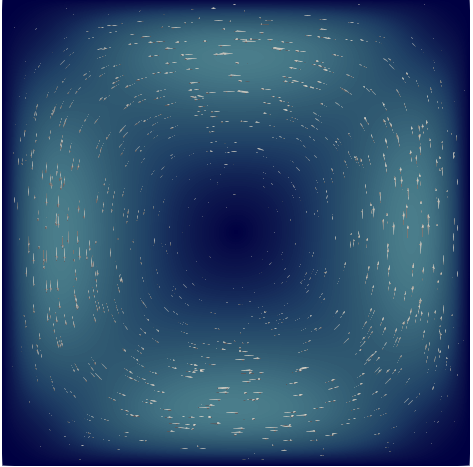}
}
\subfloat[][$t=0.25$]{
\includegraphics[width=\scale]{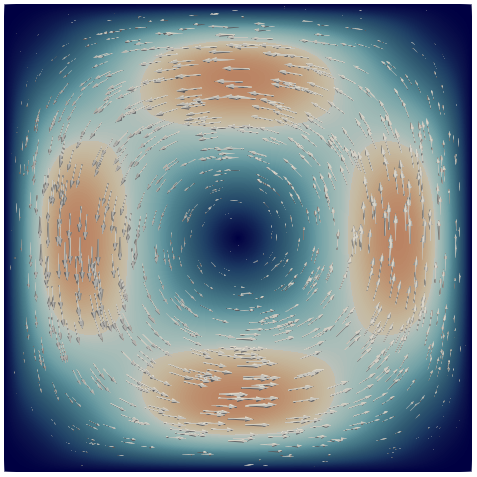}
}\hfill
\subfloat[][$t=0.5$]{
\includegraphics[width=\scale]{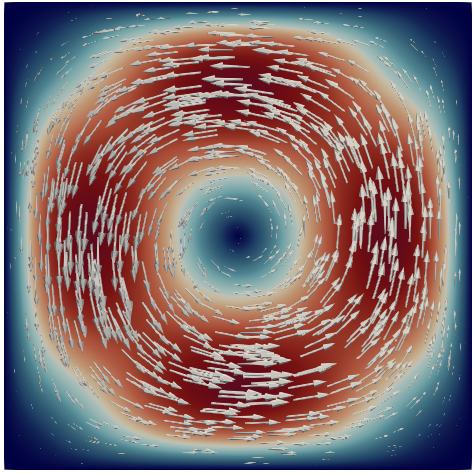}
}
\subfloat[][$t=0.75$]{
\includegraphics[width=\scale]{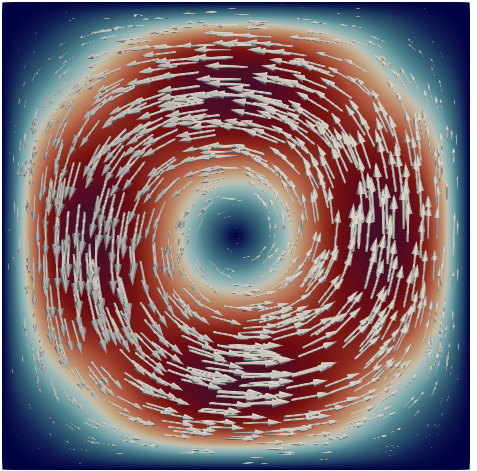}
}\hfill
\subfloat[][$t=1$]{
\includegraphics[width=\scale]{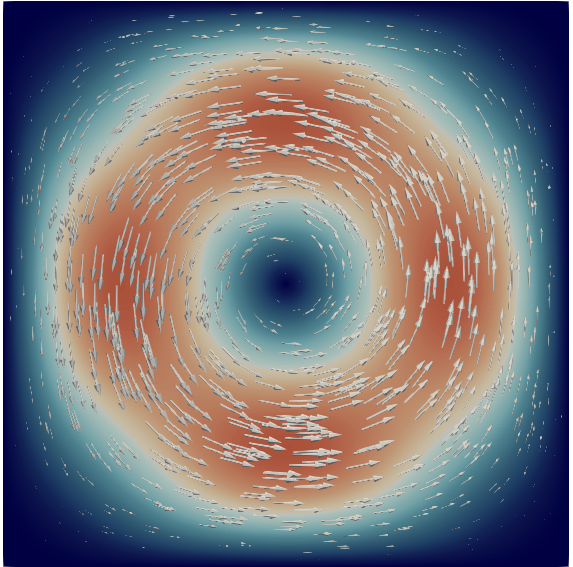}
}
\vspace{-2ex}
\begin{center}
\subfloat{
\includegraphics[scale=0.3]{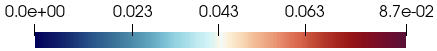}
}
\end{center}\vspace{-2ex}
\setcounter{subfigure}{0}
\subfloat[][$t=0.125$]{
\includegraphics[width=\scale]{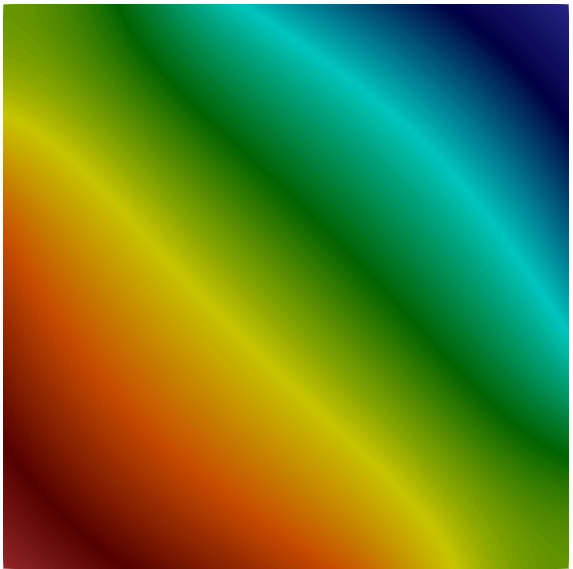}
}
\subfloat[][$t=0.25$]{
\includegraphics[width=\scale]{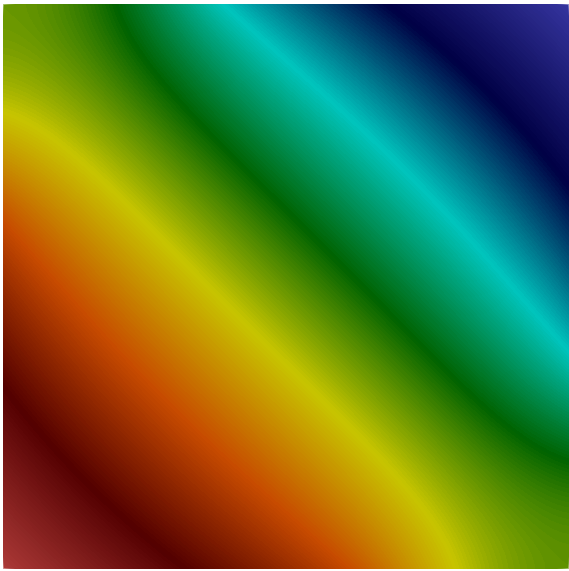}
}\hfill
\subfloat[][$t=0.50$]{
\includegraphics[width=\scale]{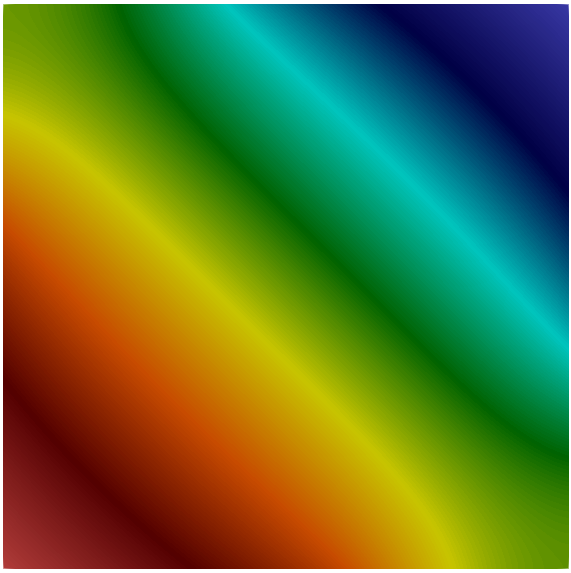}
}
\subfloat[][$t=0.75$]{
\includegraphics[width=\scale]{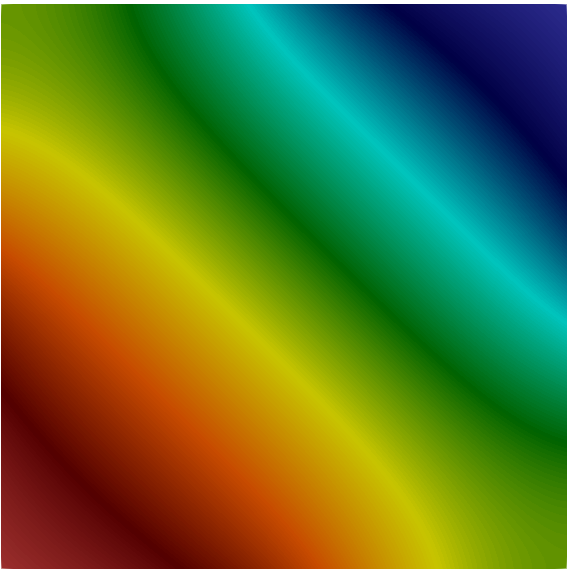}
}\hfill
\subfloat[][$t=1$]{
\includegraphics[width=\scale]{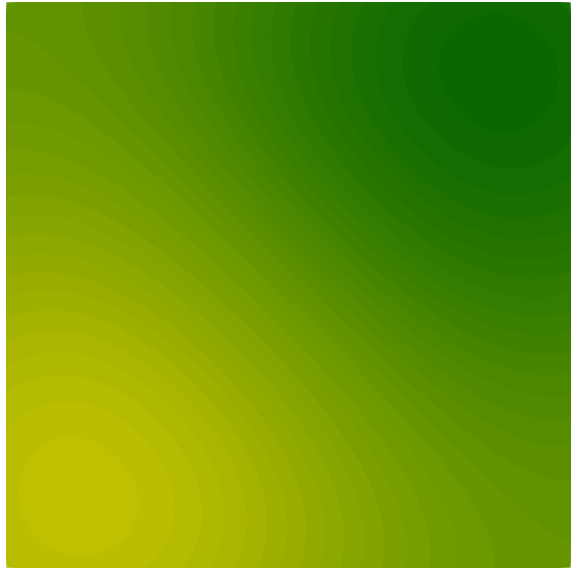}
}\vspace{-2ex}
\begin{center}
\subfloat{
\includegraphics[scale=0.3]{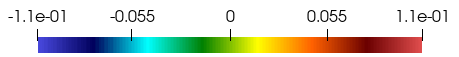}
}
\end{center}\vspace{-2ex}

\newcommand{\scala}{.18\textwidth}
\setcounter{subfigure}{0}
\subfloat[][$t=0.125$]{
\includegraphics[width=\scala]{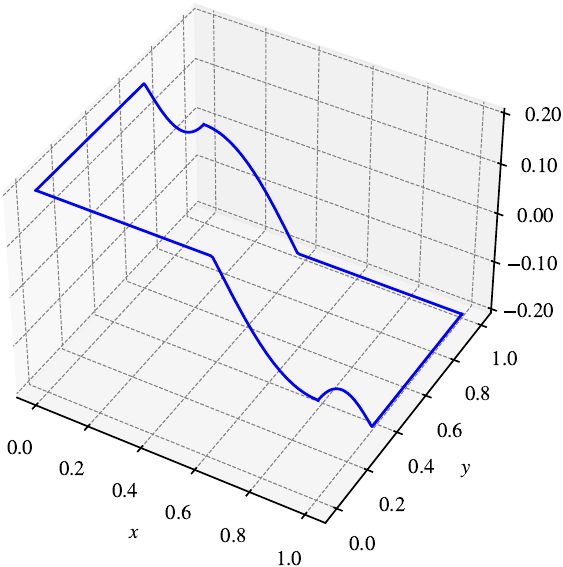}
}
\subfloat[][$t=0.25$]{
\includegraphics[width=\scala]{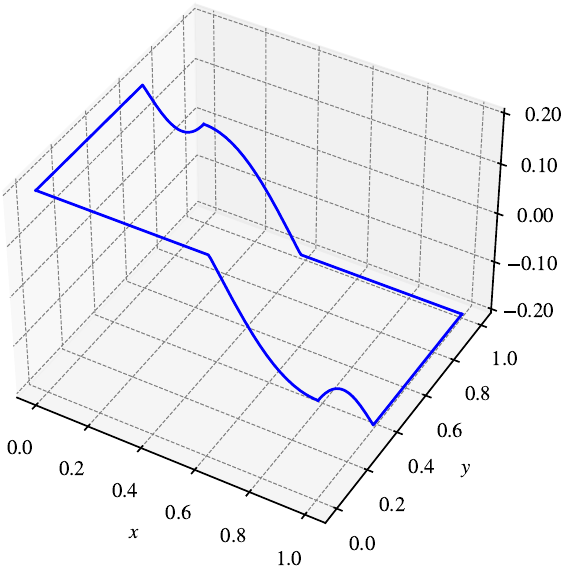}
}\hfill
\subfloat[][$t=0.5$]{
\includegraphics[width=\scala]{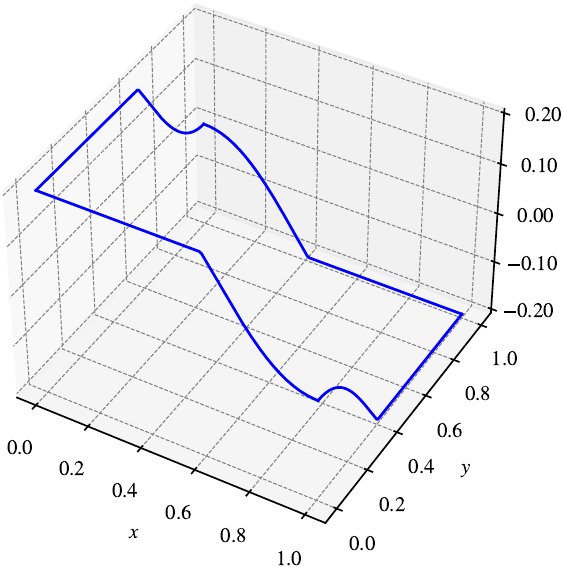}
}
\subfloat[][$t=0.75$]{
\includegraphics[width=\scala]{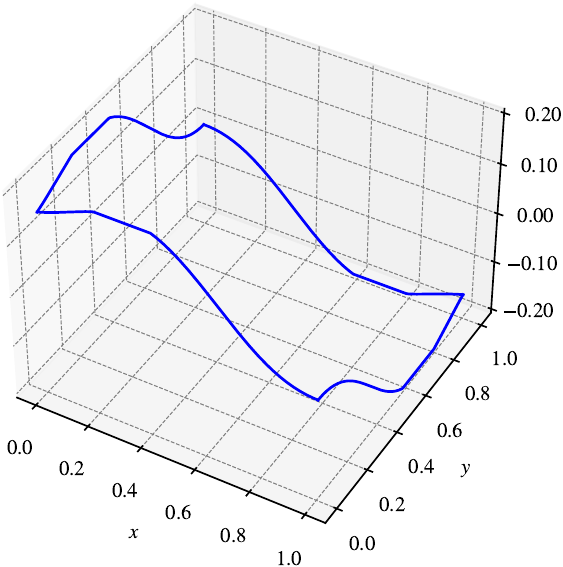}
}\hfill
\subfloat[][$t=1$]{
\includegraphics[width=\scala]{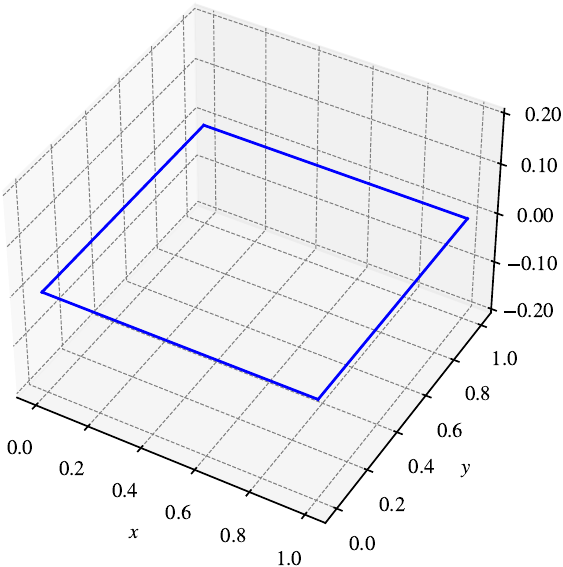}
}
\caption{The magnitude of the optimal velocity $\bar{\mathbf{y}}$, and the corresponding optimal temperature $\bar{\theta}$ and optimal control $\bar{u}$ at $t = 0.125$, $t = 0.25$, $t = 0.5$, $t=0.75$ and $t = 1$.}\label{Fig:2}
\end{figure}


The target velocity is given in Fig. \ref{Fig:1}, and the optimal solution at $t = 0.125, ~0.25, ~0.5,~0.75,~1$ is given in Fig. \ref{Fig:2}.  We define $\mathcal{E}^{\mathbf{y}}_X = \|\bar{\mathbf{y}}_\sigma - \bar{\mathbf{y}}^*\|_X$, analogously to $\mathcal{E}^{\theta}_X$, $\mathcal{E}^{\kappa}_X$, $\mathcal{E}^{\bm \mu}_X$ and $\mathcal{E}^{u}_X$.   The errors and corresponding convergence rates are summarized in Tables \ref{tab:eoc:h} and \ref{tab:eoc:tau}.  In Table \ref{tab:eoc:h}, we fix the time step size at $\tau = 2^{-9}$ while varying the spatial step size $h$ from $2^{-3}$ to $2^{-6}$.  The data show that with spatial refinement, the error exhibits an approximately quadratic convergence rate in the norms $L^2(I; \mathbb L^2(\Omega))$ (or $L^2(I;L^2(\Omega))$), $L^\infty(I; \mathbb L^2(\Omega))$, and $L^2(I;L^2(\Gamma))$. In contrast, convergence in the $L^2(I;\mathbb H^1(\Omega))$ norm follows only a linear rate.   In Table \ref{tab:eoc:tau}, we fix the spatial step size at $h = 2^{-7}$, while varying the time step size $\tau$ from $2^{-3}$ to $2^{-7}$.
 The data show that the convergence with respect to the time step is linear for all  the norms considered.  From these results, we observe that the empirical convergence rates for both the spatial and temporal discretizations of the control variable are higher than the theoretically predicted rates. This discrepancy can be attributed to the smoothness of the test data used in our numerical experiments.

\begin{sidewaystable}
\caption{Experimental order of convergence with $h$ ($\tau=2^{-9}$)}\label{tab:err:l_2}\label{tab:eoc:h}
\scalebox{0.82}{\begin{tabular}{c|ccccccccccccccccccc}
$h$ &$\mathcal{{E}}^{\mathbf y}_{L^\infty(\mathbb L^2)}$& $\mathrm{Rate}$ &$\mathcal{{E}}^{\theta}_{L^\infty( L^2)}$ &$\mathrm{Rate}$&$\mathcal{{E}}^{\mathbf y}_{L^2(\mathbb H^1)}$& $\mathrm{Rate}$ &$\mathcal{{E}}^{\theta}_{L^2( H^1)}$ &$\mathrm{Rate}$&$\mathcal{{E}}^{\bm \mu}_{L^\infty(\mathbb L^2)}$ &$\mathrm{Rate}$& $\mathcal{{E}}^{\kappa}_{L^\infty(L^2)}$ &$\mathrm{Rate}$&$\mathcal{{E}}^{\bm \mu}_{L^2(\mathbb H^1)}$ &$\mathrm{Rate}$& $\mathcal{{E}}^{\kappa}_{L^2(H^1)}$ &$\mathrm{Rate}$&  $\mathcal{{E}}^{u }_{L^2(L^2)}$ &$\mathrm{Rate}$& \\ \hline
        $2^{-3}$ & 5.087E-03 & ~ & 8.899E-04 & ~ & 1.123E-01 & ~ & 1.091E-02 & ~ & 4.290E-03 & ~ & 2.747E-03 & ~ & 9.845E-02 & ~ & 1.667E-02 & ~ & 7.107E-03 & ~ & ~ \\ \hline
        $2^{-4}$  & 1.093E-03 & 2.22 & 1.726E-04 & 2.37 & 5.308E-02 & 1.08 & 5.304E-03 & 1.04 & 8.843E-04 & 2.28 & 5.794E-04 & 2.25 & 4.496E-02 & 1.13 & 7.449E-03 & 1.16 & 1.318E-03 & 2.43 & ~ \\ \hline
        $2^{-5}$  & 2.595E-04 & 2.08 & 4.050E-05 & 2.09 & 2.643E-02 & 1.01 & 2.689E-03 & 0.98 & 2.106E-04 & 2.07 & 1.416E-04 & 2.03 & 2.254E-02 & 1.00 & 3.862E-03 & 0.95 & 3.251E-04 & 2.02 & ~ \\ \hline
        $2^{-6}$  & 5.391E-05 & 2.27 & 8.444E-06 & 2.26 & 1.285E-02 & 1.04 & 1.309E-03 & 1.04 & 4.435E-05 & 2.25 & 2.832E-05 & 2.32 & 1.139E-02 & 0.98 & 1.937E-03 & 1.00 & 6.659E-05 & 2.29 & 
\end{tabular}}

\vspace{8em} 

\caption{Experimental order of convergence with $\tau$ ($h=2^{-7}$)}\label{tab:err:l_2}\label{tab:eoc:tau}
\scalebox{0.82}{\begin{tabular}{c|ccccccccccccccccccc}
$\tau$ &$\mathcal{{E}}^{\mathbf y}_{L^\infty(\mathbb L^2)}$& $\mathrm{Rate}$ &$\mathcal{{E}}^{\theta}_{L^\infty( L^2)}$ &$\mathrm{Rate}$&$\mathcal{{E}}^{\mathbf y}_{L^2(\mathbb H^1)}$& $\mathrm{Rate}$ &$\mathcal{{E}}^{\theta}_{L^2( H^1)}$ &$\mathrm{Rate}$&$\mathcal{{E}}^{\bm \mu}_{L^\infty(\mathbb L^2)}$ &$\mathrm{Rate}$& $\mathcal{{E}}^{\kappa}_{L^\infty(L^2)}$ &$\mathrm{Rate}$&$\mathcal{{E}}^{\bm \mu}_{L^2(\mathbb H^1)}$ &$\mathrm{Rate}$& $\mathcal{{E}}^{\kappa}_{L^2(H^)}$ &$\mathrm{Rate}$&  $\mathcal{{E}}^{u }_{L^2(L^2)}$ &$\mathrm{Rate}$& \\ \hline
$2^{-3}$  & 2.025E-02 & ~ & 2.944E-02 & ~ & 7.022E-02 & ~ & 2.958E-02 & ~ & 2.529E-02 & ~ & 9.692E-03 & ~ & 3.490E-02 & ~ & 1.570E-02 & ~ & 5.248E-02 & ~ & ~ \\ \hline
        $2^{-4}$  & 1.096E-02 & 0.89 & 2.184E-02 & 0.43 & 3.658E-02 & 0.94 & 1.628E-02 & 0.86 & 1.559E-02 & 0.70 & 5.291E-03 & 0.87 & 1.903E-02 & 0.88 & 8.013E-03 & 0.97 & 2.698E-02 & 0.96 & ~ \\ \hline
        $2^{-5}$  & 5.623E-03 & 0.96 & 1.431E-02 & 0.61 & 1.841E-02 & 0.99 & 8.780E-03 & 0.89 & 8.642E-03 & 0.85 & 2.684E-03 & 0.98 & 9.830E-03 & 0.95 & 3.982E-03 & 1.01 & 1.351E-02 & 1.00 & ~ \\ \hline
        $2^{-6}$  & 2.705E-03 & 1.06 & 8.315E-03 & 0.78 & 8.857E-03 & 1.06 & 4.609E-03 & 0.93 & 4.351E-03 & 0.99 & 1.288E-03 & 1.06 & 4.849E-03 & 1.02 & 1.899E-03 & 1.07 & 6.489E-03 & 1.06 & ~ \\ \hline
        $2^{-7}$  & 1.179E-03 & 1.20 & 4.203E-03 & 0.98 & 3.927E-03 & 1.17 & 2.282E-03 & 1.01 & 1.942E-03 & 1.16 & 5.599E-04 & 1.20 & 2.247E-03 & 1.11 & 8.377E-04 & 1.18 & 2.890E-03 & 1.17 & ~ 
\end{tabular}}
\end{sidewaystable}

\bibliographystyle{abbrv}

\begin{appendices}
\section{Fundamental Lemmas}\label{Appendix:A}
The following properties of the trilinear form $\mathbf b$ can be found in various books (cf. \cite{Girault_1986,Temam_1970})
\begin{lemma}\label{lem:b:xingzhi}
The trilinear form \( \mathbf b \) satisfies the following properties:
\begin{enumerate}
    \item  \( \mathbf b(\mathbf u, \mathbf v, \mathbf w) = \mathbf c(\mathbf u, \mathbf v, \mathbf w) = -\mathbf c(\mathbf u, \mathbf w, \mathbf v) \quad \forall \mathbf u \in \mathbb X, \quad \forall \mathbf v,\mathbf w \in \mathbb H^1(\Omega) \).

\item \( \mathbf b(\mathbf u, \mathbf v, \mathbf w) = -\mathbf b(\mathbf u, \mathbf w, \mathbf v) \quad \forall \mathbf u \in \mathbb X, \quad \forall \mathbf v, \mathbf w \in \mathbb H^1(\Omega) \).

\item \( \mathbf b(\mathbf u, \mathbf v, \mathbf v) = 0 \quad \forall \mathbf u \in\mathbb X, \quad \forall \mathbf v \in \mathbb H^1(\Omega) \).
\end{enumerate}
\end{lemma}

We present some useful results regarding the projection operators (cf. \cite{Girault_1986}).
\begin{lemma}\label{Lem:stokes:err} Let \( \mathbf{u} \in \mathbb{H}^2(\Omega) \cap \mathbb{H}^1_0(\Omega) \), \( p \in H^1(\Omega) \cap L_0^2(\Omega) \), the Stokes-Ritz projection $R^S_h(\mathbf u,p)$ satisfies
\begin{align*}
&\|\mathbf u-\mathbf R^S_h(\mathbf u,p)\|\leq Ch\big(\|\mathbf u\|_{\mathbb H^1(\Omega)}+\|p\|\big),\quad\|\mathbf u-\mathbf R^S_h(\mathbf u,p)\|\leq Ch^2\big(\|\mathbf u\|_{\mathbb H^2(\Omega)}+\|p\|_{H^1(\Omega)}\big).
\end{align*}
\end{lemma}


\begin{lemma}\label{Lem:L2:err} Let \( \mathbf u \in \mathbb H^2(\Omega) \), \(  u \in  H^s(\Omega) \) with $s\in [1,2]$,  we have the following estimates for the $L^2$ projections \( \mathbf P_h \),  \(P_h\) and the Ritz projection \(R_h\)
\begin{align*}
&\|\mathbf u-\mathbf P_h \mathbf u\|+h\|\nabla \big(\mathbf u-\mathbf P_h \mathbf u\big)\|\leq Ch^2\|u\|_{\mathbb H^2(\Omega)},\\
&\|u-P_h u\|+h\|\nabla \big( u- P_h u\big)\|\leq C h^s\|u\|_{H^s(\Omega)},\\
&\|u-R_h u\|+h\|\nabla \big( u- R_h u\big)\|\leq C h^2\|u\|_{H^2(\Omega)}.
\end{align*}
\end{lemma}

Below is the error estimate for the interpolation operators $\Pi^r_\tau$, $\Pi^l_\tau$, and the $L^2$-projection $P_\tau$, as discussed in \cite[p. 762-763]{Liang_Gong_Xie_2025}.
\begin{lemma}\label{Proj:L2:time:err}For arbitrary $s\in [0, 1]$ and $w\in H^s(I;L^2(\Omega))$, there holds
\begin{align*}
\|w-P_\tau w\|_{L^2(I;L^2(\Omega))}\leq C \tau^s \|w\|_{H^s(I;L^2(\Omega))}.
\end{align*}
\end{lemma}

\begin{lemma}\label{int:L2:time:err}For any $s \in [\tfrac{1}{2}, 1]$, if $w \in H^s(I; L^2(\Omega)) \cap C(\bar{I}; L^2(\Omega))$, the following estimate holds:
\begin{align*}
\|w-\Pi^r_\tau w\|_{L^2(I;L^2(\Omega))}+\|w-\Pi^l_\tau w\|_{L^2(I;L^2(\Omega))}\leq C \tau^s \|w\|_{H^s(I;L^2(\Omega))}.
\end{align*}
\end{lemma}

\section{Supplementary Lemma}\label{Appendix:B}

\begin{lemma}\label{appB:Lem:full:state:err}Given $u\in L^2(I;L^2(\Gamma))$, let \((\mathbf y,\theta)\in \mathbb V(I)\times \big(L^2(I; {H}^{\frac{3}{2}}(\Omega))\cap (H^{\frac{3}{4}}(I; L^2(\Omega))\cap H^1(I;H^{1}(\Omega)^*))\big) \) be the solution of \eqref{Exist:weak:solution} and $\big\{(\mathbf y_{n,h}, \theta_{n,h})\big\}_{n=1}^{N_\tau}$ be the solution of  the discrete state equation \eqref{Full_discrete:state}. Then 
\begin{equation}\label{Lem:full:lowRe:state:err:ineq}
\begin{aligned}
&\|\mathbf y-\mathbf y_\sigma\|_{L^2(I;\mathbb H^1(\Omega))}+\|\theta-\theta_\sigma\|_{L^2(I; H^1(\Omega))}\leq C\big(\tau^{\frac{1}{4}}+ h^\frac{1}{2}\big).
\end{aligned}
\end{equation}
\end{lemma}

\begin{proof} We start with the following decomposition:
\begin{equation}\label{App:FULL:decomp:YTheta}
\begin{aligned}
\mathbf y-\mathbf y_{\sigma}=\mathbf y -\Pi^{\mathbf y}_\sigma \mathbf y+\Pi^{\mathbf y}_\sigma \mathbf y-\mathbf y_\sigma=\bm \zeta^{\mathbf y}_\sigma+\bm \eta^{\mathbf y}_\sigma, \quad
\theta-\theta_{\sigma}=\theta-\Pi^\theta_\sigma\theta+\Pi^\theta_\sigma\theta- \theta_\sigma=\zeta^{\theta}_\sigma+\eta^{\theta}_\sigma,
\end{aligned}
\end{equation}
where we set $\Pi^{\mathbf y}_\sigma \mathbf y=\mathbf P_h \Pi^r_\tau \mathbf y$, $\Pi^{\theta}_\sigma \theta=\Pi_\tau^rP_h\theta$,   $\bm \zeta^{\mathbf y}_{n, h}=\bm \zeta^{\mathbf y}_{\sigma}(t_n)$ and $\zeta^{\theta}_{n,h}=\zeta^\theta_\sigma(t_n)$. Integrating from $t_{n-1}$ to $t_n$ for \eqref{Regualrit:equa:reform}, \eqref{Exist:weak:solution:c},  subtracting \eqref{Full_discrete:state} from it with $(\mathbf v,\psi)=(\mathbf v_h,\psi_h)\equiv (\bm \eta^{\mathbf y}_{n, h}, \eta^\theta_{n, h})$ and using the decomposition \eqref{App:FULL:decomp:YTheta},  we obtain
\begin{align}
&\frac{1}{2}\|\bm \eta^{\mathbf y}_{n,h}\|^2-\frac{1}{2}\|\bm \eta^{\mathbf y}_{n-1,h}\|^2+\frac{1}{2}\|\eta^\theta_{n,h}\|^2-\frac{1}{2}\| \eta^\theta_{n-1,h}\|+\nu\tau_n\|\nabla \bm  \eta^{\mathbf y}_{n,h}\|^2+\tau_n\chi\|\nabla \eta^\theta_{n,h}\|^2+\tau_n\eta\gamma\|\eta^\theta_{n,h}\|^2_{\Gamma}\leq\nonumber\\
&-\int_{t_{n-1}}^{t_n}\big[\nu\mathbf a(\bm \zeta^{\mathbf y}_\sigma, \bm \eta^{\mathbf y}_{n,h})-(p,\mathrm{div}\,\bm\eta^{\mathbf y}_{n,h})-\chi a( \zeta^{\theta}_\sigma,  \eta_{n,h}^{\theta})-\eta\gamma(\zeta^{\theta}_\sigma,\eta_{n,h}^{\theta})_\Gamma-\beta (\zeta^\theta_\sigma+P_h(\theta(t_n)-\theta(t_{n-1}))+\eta^\theta_{n-1,h},\bm \eta^{\mathbf y}_{n,h}\cdot \mathbf g)\big] dt\nonumber\\
&+\int_{t_{n-1}}^{t_n}\big[\mathbf b(\mathbf y_{n-1,h},\mathbf y_{n,h},\bm \eta^{\mathbf y}_{n,h})-\mathbf b(\mathbf y, \mathbf y,\bm  \eta^{\mathbf y}_{n,h}) \big]dt+\int_{t_{n-1}}^{t_n}\big[\tau_n b(\mathbf y_{n-1,h}, \theta_{n,h},\eta^\theta_{n,h})-b(\mathbf y, \theta, \eta^\theta_{n,h})\big]dt\nonumber\\
&=J_1+J_2+J_3.\label{App:err:lem:1}
\end{align}
Using Assumption \textbf{(A2)}, we can obtain 
\begin{align*}
J_1&\leq \frac{\nu}{4}\tau_n\|\nabla\bm \eta^{\mathbf y}_{n,h}\|^2+\frac{\chi}{6}\tau_{n}\|\nabla  \eta^{\theta}_{n,h}\|^2+\frac{\eta\gamma}{6}\tau_{n}\|\eta^\theta_{n,h}\|^2_\Gamma+C\Big(\tau_n\|\eta^\theta_{n-1,h}\|^2+\int_{t_{n-1}}^{t_n}\big[\|\nabla\bm \zeta^{\mathbf y}_\sigma\|^2+\|\zeta^\theta_\sigma\|^2_{H^1(\Omega)}\\
&+h^2\|p\|^2_{H^1(\Omega)}\big]dt+\tau_n^2\|\partial_t\theta\|^2_{H^{1}(\Omega)^*}\Big).
\end{align*}
Using an argument similar to \eqref{Full:err:YTheta:bu:1}, we have the estimates for $J_2$ and $J_3$
\begin{align*}\footnotesize
&J_2\leq \frac{\nu\tau_n}{4}\|\nabla \bm \eta^{\mathbf y}_{n,h}\|^2+\frac{\nu\tau_{n-1}}{8}\|\nabla \bm \eta^{\mathbf y}_{n-1,h}\|^2+C\big(\int_{t_{n-1}}^{t_n}\|\nabla \bm \zeta^{\mathbf y}_{\sigma}\|^2dt+\tau_{n}\|\bm \eta^{\mathbf y}_{n-1,h}\|^2+\tau_n \int_{t_{n-1}}^{t_n}\|\partial_t \mathbf y\|dt\,\|\nabla(\mathbf y(t_n)-\mathbf y(t_{n-1}))\|\big),\\
&J_3\leq\frac{\nu}{8}\tau_{n-1}\|\nabla \bm \eta^{\mathbf y}_{n-1,h}\|^2+ \frac{\chi}{6}\tau_{n}\|\nabla  \eta^{\theta}_{n,h}\|^2+\frac{\eta\gamma}{6}\tau_{n}\|\eta^\theta_{n,h}\|^2_\Gamma+C\Big(\int_{t_{n-1}}^{t_n}\|\partial_t \mathbf y\|dt\Big(\int_{t_{n-1}}^{t_n}\|\theta\|^2_{H^1(\Omega)}dt\\
&+\int_{t_{n-1}}^{t_n}\|\Pi^\theta_\sigma\theta\|^2_{H^1(\Omega)}dt\Big)+\int_{t_{n-1}}^{t_n}\|\zeta_\sigma^\theta\|^2_{H^1(\Omega)}dt+\tau_{n}\|\Pi^\theta_\sigma \theta\|^2_{H^1(\Omega)}\|\bm \eta^{\mathbf y}_{n-1,h}\|^2+\tau_{n}\|\Pi^\theta_\sigma \theta_{n,h}\|^2_{H^1(\Omega)}\|\eta_{n,h}^\theta\|^2\Big).
\end{align*}
 Thus, combining the above estimates with \eqref{App:err:lem:1} and summing over $n = 1, \dots, k$ for any $1 \leq k \leq N_\tau$, we obtain the following bound
\begin{align}
&\|\bm \eta_{k,h}^{\mathbf y}\|^2+\| \eta_{k,h}^{\theta}\|^2+\nu\sum_{n=1}^k\tau_n\|\nabla \bm \eta_{n,h}^{\mathbf y}\|^2+\chi\sum_{n=1}^k\tau_n\|\nabla \eta_{n, h}^\theta\|^2+\eta\gamma\sum_{n=1}^k\tau_n\|\eta_{n,h}^\theta\|^2_\Gamma\nonumber\\
&\leq C  \Big(\int_{0}^{t_k}\big[\|\nabla \bm \zeta^{\mathbf y}_\sigma \|^2+\|\zeta_{\sigma}^\theta\|^2_{H^1(\Omega)}+h^2\|p\|_{H^1(\Omega)}^2\big]dt+\tau^\frac{1}{2}\big(\int_0^{T}\|\partial_t\mathbf  y\|^2dt\big)^\frac{1}{2}\sum_{n=1}^{k}\tau_n\|\nabla\big( \mathbf y(t_n)-\mathbf y(t_{n-1})\big)\|\nonumber\\
&+\tau_n^2\|\partial_t\theta\|^2_{H^{1}(\Omega)^*}+\tau^\frac{1}{2}\big(\int_0^{T}\|\partial_t\mathbf  y\|^2dt\big)^\frac{1}{2}\int_{0}^{t_k}\big[\|\theta\|^2_{H^1(\Omega)}+\|\Pi^\theta_\sigma \theta\|^2_{H^1(\Omega)}\big]dt\nonumber\\
&+\sum_{n=1}^{k}\tau_{n}\big(\|P_h \theta(t_n)\|^2_{H^1(\Omega)}\|\bm \eta^{\mathbf y}_{n-1,h}\|^2+\tau_n\|\eta^\theta_{n-1,h}\|^2+\|P_h \theta(t_n)\|^2_{H^1(\Omega)}\|\eta_{n,h}^\theta\|^2)\Big).\label{App:err:B10}
\end{align}
Since $L^2(I; {H}^{\frac{3}{2}}(\Omega))\cap H^1(I;H^{1}(\Omega)^*)) \hookrightarrow C(\bar{I}; L^2(\Omega))$, it follows that $\theta, P_h\theta \in C(\bar{I}; L^2(\Omega))$, and consequently $P_h\theta \in C(\bar{I}; H^1(\Omega))$. Additionally, using the interpolation space $\big[L^2(I;H^\frac{3}{2}(\Omega)), H^\frac{3}{4}(I;L^2(\Omega))\big]_\frac{1}{3} = H^\frac{1}{4}(I;H^1(\Omega))$ (cf. \cite[Chapter 1, Equation 9.24]{Lions_1972}), we have $\theta \in H^\frac{1}{4}(I;H^1(\Omega))$. Next, employing the $H^{1}$-stability of the $L^2$ projection, we obtain
\[\|P_h \theta\|_{H^\frac{1}{4}(I;H^1(\Omega))}\leq C \|\theta\|_{H^\frac{1}{4}(I;H^1(\Omega))}.\]
Therefore,  applying Lemmas \ref{Lem:L2:err} and \ref{int:L2:time:err}, we have 
\begin{equation}
\begin{aligned}
\|\theta-\Pi^\theta_\sigma\theta\|_{L^2(I;H^1(\Omega))}\leq\|\theta-P_h\theta\|_{L^2(I;H^1(\Omega))}+\|P_h\theta-\Pi^\theta_\sigma\theta\|_{L^2(I;H^1(\Omega))}\leq C\big(\tau^\frac{1}{4}+h^\frac{1}{2}\big).\label{App:err:B11}
\end{aligned}
\end{equation}
Similar to Remark \ref{remark:conv:Theta},  \eqref{App:err:B11} implies that $\tau_n\|P_h \theta(t_n)\|^2_{H^1(\Omega)} \to 0$ uniformly in $n$ as $\sigma \to 0$. Hence, we can choose the temporal and space  mesh sizes fine enough such that $C  \tau_n \|P_h \theta(t_n)\|^2_{H^1(\Omega)} \leq \frac{1}{2}$ holds for all $n$. Combining \eqref{App:err:B10} and \eqref{App:err:B11} with Lemmas \ref{Lem:L2:err} and \ref{int:L2:time:err}, and applying Gronwall's inequality together with the triangle inequality, completes the proof.
\end{proof}

\end{appendices}

\end{document}